\documentclass[11pt,a4paper]{amsart}

\usepackage{amssymb,amscd,amsmath,young, stmaryrd,tikz,hyperref}
\usepackage{mathrsfs, mathdots}
\usepackage[mathcal]{eucal}
\usepackage{float,mathabx}
\usepackage{soul,caption}
\usepackage{cancel}
\usepackage[normalem]{ulem}
\usepackage{hyperref}

\usepackage{longtable,array,multirow,makecell,graphicx}

  \theoremstyle{plain}
  \newtheorem{thm}{Theorem}[section]
  \newtheorem{lem}[thm]{Lemma}
  \newtheorem{pro}[thm]{Proposition}
  \newtheorem{cor}[thm]{Corollary}
  \newtheorem{hyp}[thm]{Hypothesis}

  \newtheorem*{con*}{Conjecture}
  \newtheorem{lemma}[thm]{Lemma}

  \theoremstyle{remark}
  \newtheorem{rem}[thm]{Remark}
  
  \newtheorem{exm}[thm]{Example}

  \newtheorem{dfn}[thm]{Definition}
  \newtheorem*{acknowledgements}{Acknowledgements}

  \numberwithin{equation}{section}
  \numberwithin{table}{section}

  \newcommand{\N}{\mathbb{N}}
  \newcommand{\Z}{\mathbb{Z}}
  \newcommand{\Q}{\mathbb{Q}}
  
  \newcommand{\C}{\mathbb{C}}
  
  \newcommand{\R}{\mathbb{R}}

  \newcommand{\mfg}{\mathfrak{g}}
  \newcommand{\mff}{\mathfrak{f}}

  \newcommand{\mfp}{\mathfrak{p}}

  \newcommand{\lri}{\mathfrak{o}}
  \newcommand{\Lri}{\mathfrak{O}}

  \newcommand{\ol}{\overline}

  \newcommand{\Gri}{\ensuremath{\mathcal{O}}}

  \renewcommand{\epsilon}{\varepsilon}
  
  \renewcommand{\phi}{\varphi}

  \newcommand{\mcO}{\mathcal{O}}

  \newcommand{\rarr}{\rightarrow}
  \newcommand{\epsi}{\varepsilon}
  
  \newcommand{\ideal}{\triangleleft}
  
  \newcommand{\zidealo}{\zeta^{\triangleleft\,\lri}}
  \newcommand{\Rho}{\mathrm{P}}

  \DeclareMathOperator{\WO}{WO}
  
  \DeclareMathOperator{\Des}{Des}
  
  \DeclareMathOperator{\ad}{ad}

  \DeclareMathOperator{\rk}{rk}

  \DeclareMathOperator{\Mat}{Mat}

  \def \wo {n}

  \def \bsnu {\boldsymbol{\nu}}
  \def \bfz {{\bf 0}}
  \def \bfo {{\bf 1}}

  \def \bff {{\bf f}}

  \def \bfX {{\bf X}}
  \def \bfy {{\bf y}}
  \def \bfY {{\bf Y}}
  
  \def \bfone {{\bf 1}}

  \def \mcD {\ensuremath{\mathcal{D}}}
  
  \def \mcL {\ensuremath{\mathcal{L}}}
  \def \mfL {\ensuremath{\mathfrak{L}}}

  \def \Fp {\ensuremath{\mathbb{F}_p}}
  
  \def \mcP {\ensuremath{\mathcal{P}}}

  \def \mfo {\ensuremath{\mathfrak{o}}}
  \def \mfO {\ensuremath{\mathfrak{O}}}

  \def \Fq {\ensuremath{\mathbb{F}_q}}
  \def \Zp  {\mathbb{Z}_p}
  \def \Z {\mathbb{Z}}
  \def \g {\mathfrak{g}}
  
  \def \pnum {h}
  \def \klim {Z}
  \newcommand{\gp}[1]{\frac{#1}{1-#1}}
  \newcommand{\gpf}[1]{\mathrm{gp}(#1)}
  \newcommand{\gpzero}[1]{\mathrm{gp_0}(#1)}

  \def \la {\langle}
  \def \ra {\rangle}
    
  \author{Angela Carnevale} \address{School of Mathematical and Statistical Sciences,  University of  Galway, Ireland}
  \email{angela.carnevale@universityofgalway.ie}
  \author{Michael M. Schein} \address{Department of Mathematics, Bar-Ilan University, Ramat Gan 5290002, Israel} 
  \email{mschein@math.biu.ac.il}
  \author{Christopher Voll} \address{Fakult\"at f\"ur Mathematik,
    Universit\"at Bielefeld, D-33501 Bielefeld, Germany}
    \email{C.Voll.98@cantab.net}



\begin{document}
   \title[Generalized Igusa functions and ideal growth in nilpotent
     Lie rings]{Generalized Igusa functions and ideal growth in
     nilpotent Lie rings} 

   \begin{abstract}
     We introduce a new class of combinatorially defined rational
     functions and apply them to deduce explicit formulae for local
     ideal zeta functions associated to the members of a large class
     of nilpotent Lie rings which contains the free class-2-nilpotent
     Lie rings and is stable under direct products. Our results unify
     and generalize a substantial number of previous computations. We
     show that the new rational functions, and thus also the local
     zeta functions under consideration, enjoy a self-reciprocity
     property, expressed in terms of a functional equation upon
     inversion of variables. We establish a conjecture of Grunewald,
     Segal, and Smith on the uniformity of normal zeta functions of
     finitely generated free class-$2$-nilpotent groups.
\end{abstract}
  \maketitle
\thispagestyle{empty}

  \setcounter{tocdepth}{4}

\section{Introduction}
The objective of this paper is twofold.  The first aim is
to introduce a new class of combinatorially defined multivariate
rational functions and to prove that they satisfy a self-reciprocity
property, expressed in terms of a functional equation upon inversion
of variables. The second is to apply these rational functions to
obtain an explicit description of the local ideal zeta functions associated to a
class of combinatorially defined Lie rings. We start with a discussion
of the latter application before formulating and explaining the new
class of rational functions.

  \subsection{Finite uniformity for ideal zeta functions of nilpotent Lie rings}
  Given an additively finitely generated ring $\mcL$, i.e.\ a finitely
  generated $\Z$-module with some bi-additive, not necessarily
  associative multiplication, the ideal zeta function of $\mcL$ is the
  Dirichlet generating series
  \begin{equation}\label{def:ideal.zeta}
    \zeta^{\ideal}_{\mcL}(s) = \sum_{I \ideal \mcL}|\mcL:I|^{-s},
  \end{equation}
  where $I$ runs over the (two-sided) ideals of $\mcL$ of finite additive
  index in $\mcL$ and $s$ is a complex variable. Prominent examples of
  ideal zeta functions include the Dedekind zeta functions,
  enumerating ideals of rings of integers of algebraic number fields
  and, in particular, Riemann's zeta function $\zeta(s)$.

  It is not hard to verify that, for a general ring $\mcL$, the ideal
  zeta function $\zeta^{\ideal}_{\mcL}(s)$ satisfies an Euler product
  whose factors are indexed by the rational primes:
$$\zeta^{\ideal}_{\mcL}(s) = \prod_{p \textup{
    prime}}\zeta^{\ideal}_{\mcL(\Zp)}(s),$$ where, for a prime $p$,
  $$\zeta^{\ideal}_{\mcL(\Zp)}(s) = \sum_{I \ideal
    \mcL(\Zp)}|\mcL(\Zp): I|^{-s}$$ enumerates the ideals of finite
  index in the completion $\mcL(\Zp) := L\otimes_\Z \Zp$ or,
  equivalently, the ideals of finite $p$-power index in $\mcL$. Here
  $\Zp$ denotes the ring of $p$-adic integers; note that ideals of
  $\mcL(\Zp)$ are, in particular, $\Zp$-submodules of $\mcL(\Zp)$. It
  is, in contrast, a deep result that the Euler factors
  $\zeta^{\ideal}_{\mcL(\Zp)}(s)$ are rational functions in the
  parameter~$p^{-s}$; cf.\ \cite[Theorem~3.5]{GSS/88}.

  Computing these rational functions explicitly for a given ring
  $\mcL$ is, in general, a very hard problem. Solving it is usually
  rewarded by additional insights into combinatorial, arithmetic, or
  asymptotic aspects of ideal growth. It was shown by du Sautoy and Grunewald
  \cite{duSG/00} that the problem, in general, involves the
  determination of the numbers of $\Fp$-rational points of finitely
  many algebraic varieties defined over $\Q$. Only under
  additional assumptions on $\mcL$ may one hope that these numbers are
  given by finitely many polynomial functions in $p$. We say that the
  ideal zeta function of $\mcL$ is \emph{finitely uniform} if there
  are finitely many rational functions
  $W^{\ideal}_1(X,Y),\dots,W^{\ideal}_N(X,Y) \in \Q(X,Y)$ such that
  for any prime $p$ there exists $i\in\{1,\dots,N\}$ such that
  $$\zeta^{\ideal}_{\mcL(\Zp)}(s) = W^{\ideal}_i(p,p^{-s}).$$ If a
  single rational function suffices (i.e.\ $N=1$), we say that the
  ideal zeta function of $\mcL$ is \emph{uniform}. While finite
  uniformity dominates among low-rank examples, including most of
  those included in the book \cite{duSWoodward/08} and those computed
  by Rossmann's computer algebra package
  $\mathsf{Zeta}$~\cite{rossmannzeta, Rossmann/18}, it is not
  ubiquitous: for a non-uniform example in rank~$9$, see
  \cite{duS-ecII/01} and~\cite{Voll/04}.  In general, the ideal zeta
  function of a direct product of rings is not given by a simple
  function of the ideal zeta functions of the factors. It is not even
  clear whether (finite) uniformity of the latter implies (finite)
  uniformity of the former.

  \subsubsection{Main results}

We now restrict to the case of Lie rings, namely rings in which the multiplication is anti-symmetric and satisfies the Jacobi identity; note that the Jacobi identity holds trivially for all nilpotent rings of class at most two.    
In this paper we give constructive proofs of (finite) uniformity of
ideal zeta functions associated to the members of a large class of
nilpotent Lie rings of nilpotency class at most two.

\begin{dfn}\label{def:L}
  Let $\mfL$ denote the class of nilpotent Lie rings of nilpotency
  class at most two which is closed under direct products and contains
  the following Lie rings:
  \begin{enumerate}
  \item the free class-$2$-nilpotent Lie rings $\mff_{2,d}$ on $d$
    generators, for $d\geq 2$;
    cf.\ Section~\ref{subsec:free.nilpotent}.
  \item the free class-$2$-nilpotent products $\mfg_{d,d'} = \Z^d *
    \Z^{d'}$, for $d,d' \geq 0$;
    cf.\ Section~\ref{subsec:rel.free.prod}.
    \item the higher Heisenberg Lie rings $\mathfrak{h}_d$ for $d \geq
      1$; cf.~Section~\ref{subsec:higher.heisenberg}.
  \end{enumerate}
\end{dfn}
Note that $\mcL$ contains the free abelian Lie rings $\Z^d =
\mfg_{d,0}=\mfg_{0,d}$.

Our main ``global'' result produces explicit formulae for almost all
Euler factors of the ideal zeta functions associated to Lie rings
obtained from the members of $\mfL$ by base extension with general
rings of integers of number fields. In particular, we show that these
zeta functions are finitely uniform and, more precisely, that the
variation of the Euler factors is uniform among
unramified primes with the same decomposition behaviour in the
relevant number field.
\begin{thm}\label{thm:main.global}
  Let $\mcL$ be an element of $\mfL$, and let
  $\bff = (f_1,\dots,f_g) \in\N^g$ be a $g$-tuple for some $g \in
  \N$. There exists an explicitly described rational function
  $W^{\ideal}_{\mcL,\bff} \in \Q(X,Y)$ such that the following holds:

 Let $\mcO$ be the ring of integers of a number field of degree $n$, and set
 $\mcL(\mcO) = \mcL \otimes \mcO$.  If a rational prime $p$ factorizes
 in $\mcO$ as $p \mcO =\mfp_1 \mfp_2 \cdots\mfp_g$, for pairwise
 distinct prime ideals $\mfp_i$ in $\mcO$ of inertia degrees
 $(f_1,\dots,f_g)$, then
 $$\zeta^{\ideal}_{\mcL(\mcO),p}(s) =
 W^{\ideal}_{\mcL,\bff}(p,p^{-s}).$$ In particular,
 $\zeta^{\ideal}_{\mcL(\mcO)}(s)$ is finitely uniform and
 $\zeta^{\ideal}_{\mcL}(s)=\zeta^{\ideal}_{\mcL(\Z)}(s)$ is uniform.  Moreover, the rational function $ W^{\ideal}_{\mcL,\bff}$ satisfies the functional equation
\begin{equation} \label{equ:global.functional.equation}
W^{\ideal}_{\mcL,\bff}(X^{-1},Y^{-1}) = (-1)^{n \, \mathrm{rk}_{\Z}\mcL} X^{\binom{n \, \mathrm{rk}_{\Z}\mcL}{2}} Y^{n \, (\mathrm{rk}_{\Z}\mcL + \mathrm{rk}_{\Z}(\mcL / Z(\mcL)))} W^{\ideal}_{\mcL,\bff}(X,Y).
\end{equation}
\end{thm}

A special case of Theorem~\ref{thm:main.global} establishes part of a
conjecture of Grunewald, Segal, and Smith on the normal subgroup
growth of free nilpotent groups under extension of
scalars. In~\cite{GSS/88}, they introduced the concept of the
\emph{normal zeta function}
$$\zeta_G^{\ideal}(s) = \sum_{H\ideal G}|G:H|^{-s}$$ of a torsion-free
finitely generated nilpotent group $G$, enumerating the normal
subgroups of $G$ of finite index in $G$. As $G$ is nilpotent, it also
satisfies an Euler product decomposition
$$\zeta_G^{\ideal}(s) = \prod_{p \textup{
    prime}}\zeta_{G,p}^{\ideal}(s),$$ whose factors enumerate the
normal subgroups of $G$ of $p$-power index. If $G$ has nilpotency
class two, then its normal zeta function coincides with the ideal zeta
function of the associated Lie ring $\mcL_G := G/Z(G) \oplus Z(G)$;
see the remark on p.~206 of \cite{GSS/88} and the more detailed discussion in~\cite[\S3.1]{BKO/even}.  Thus, $\zeta^{\ideal}_G(s) =
\zeta^{\ideal}_{\mcL_G}(s)$. Moreover, every class-$2$-nilpotent Lie
ring $\mcL$ arises in this way and gives rise to a torsion-free
finitely generated nilpotent group $G(\mcL)$; see
\cite[Section~1.2]{Voll/19} for details.
Theorem~\ref{thm:main.global} thus has a direct corollary pertaining
to the normal zeta functions of the finitely generated
class-$2$-nilpotent groups corresponding to the Lie rings in $\mfL$.
Since the groups associated to the free class-$2$-nilpotent Lie rings
$\mff_{2,d}$ are the finitely generated free class-$2$-nilpotent
groups~$F_{2,d} = G(\mff_{2,d})$, Theorem~\ref{thm:main.global}
implies the Conjecture on p.~188 of \cite{GSS/88} for the case $*
=\ideal$ and class $c=2$. The conjecture for normal zeta functions had
previously been established only for $d=2$ (\cite[Theorem~3]{GSS/88};
see also Section~\ref{subsubsec:previous}). We are not aware of any
other case for which the conjecture has been proven or refuted.

For any class-$2$-nilpotent Lie ring $\mcL$, it is known~\cite[Theorem~C]{Voll/10} that the Euler factors of $\zeta^{\ideal}_{\mcL}(s)$ at almost all primes $p$ are realized by rational functions admitting functional equations with the same symmetry factor $(-1)^{\mathrm{rk}_{\Z} \mcL} X^{\binom{\mathrm{rk}_{\Z} \mcL}{2}} Y^{\mathrm{rk}_{\Z} \mcL + \mathrm{rk}_{\Z} (\mcL / Z(\mcL))}$.  In particular, the functional equation~\eqref{equ:global.functional.equation} of Theorem~\ref{thm:main.global} shows that, for the Lie rings $\mathcal{L}(\mathcal{O})$, where $\mcL$ lies in our class $\mathfrak{L}$ and $\mathcal{O}$ is a number ring, the finitely many primes excluded by~\cite[Theorem~C]{Voll/10} must ramify in $\mathcal{O}$.  We suspect that they are exactly the primes ramifying in $\mathcal{O}$; see Remark~\ref{rmk:intro} below.

Theorem~\ref{thm:main.global} is a consequence of the following
uniform ``local'' result. Throughout the paper, $\lri$ will denote a
compact discrete valuation ring of arbitrary characteristic and
residue field of characteristic $p$ and cardinality $q$. Thus, $\lri$
may, for instance, be a finite extension of the ring $\Zp$ of $p$-adic
integers (of characteristic zero) or a ring of formal power series of
the form $\Fq\llbracket T\rrbracket$ (of positive characteristic).
The \emph{$\lri$-ideal zeta function}
  $$\zeta^{\ideal\,\lri}_L (s)= \sum_{I \triangleleft L}|L:I|^{-s}$$
of an $\lri$-algebra $L$ of finite $\lri$-rank is defined as in
\eqref{def:ideal.zeta}, with $I$ ranging over the $\lri$-ideals of
$L$, viz.\ $(\ad L)$-invariant $\lri$-submodules of~$L$. Note that
every element $\mcL$ of $\mfL$ may, after tensoring over $\Z$ with
$\lri$, be considered a free and finitely generated $\lri$-Lie
algebra. Given an $\lri$-module~$R$, we write~$L(R)=L\otimes_\lri R$.

\begin{thm}\label{thm:main.local}
 Let $\boldsymbol{{\mcL}} = (\mcL_1,\dots,\mcL_g)$ be a family of
 elements of $\mfL$ and $\boldsymbol{f}=
 (f_1,\dots,f_g)\in\N^g$. There exists an explicit rational function
 $W^{\ideal}_{\boldsymbol{\mcL},\boldsymbol{f}}\in\Q(X,Y)$ such that
 the following holds:

 Let $\lri$ be a compact discrete valuation ring and
 $(\Lri_1,\dots,\Lri_g)$ be a family of finite unramified extensions
 of $\lri$ with inertia degrees~$(f_1,\dots,f_g)$. Consider the
 $\lri$-Lie algebra
 $$L = \mcL_1(\Lri_1)\times\dots \times \mcL_g(\Lri_g).$$ For every
 finite extension $\Lri$ of~$\lri$, of inertia degree $f$ over $\lri$,
 say, the $\Lri$-ideal zeta function of $L(\Lri)$ satisfies
 $$\zeta^{\ideal\,\Lri}_{L(\Lri)}(s) =
 W^{\ideal}_{\boldsymbol{\mcL},\boldsymbol{f}}(q^f,q^{-fs}).$$ The
 rational function $W^{\ideal}_{\boldsymbol{\mcL},\boldsymbol{f}}$
 satisfies the functional equation
\begin{equation}\label{equ:funeq.local.newer}
  W^{\ideal}_{\boldsymbol{\mcL},\boldsymbol{f}}(X^{-1},Y^{-1}) =
  (-1)^{N_0} X^{\binom{N_0}{2}} Y^{N_0+N_1}
  W^{\ideal}_{\boldsymbol{\mcL},\boldsymbol{f}}(X,Y),
 \end{equation}
 where
 \begin{equation*}
   N_0 = \mathrm{rk}_{\mfo} L = \sum_{i=1}^g{f_i} \rk_\Z(\mcL_i) \quad \textup{
     and } \quad N_1 = \mathrm{rk}_{\mfo} (L/Z(L)) = \sum_{i=1}^g{f_i} \rk_\Z(\mcL_i/Z(\mcL_i)).
 \end{equation*}
\end{thm}

Theorem~\ref{thm:main.global} is readily deduced from
Theorem~\ref{thm:main.local}. Indeed, let $\mcL$ be a nilpotent Lie
ring as in the statement of Theorem~\ref{thm:main.global}, and let
$\mcO$ be the ring of integers of a number field.  Suppose that the
rational prime $p$ is unramified in $\mcO$ and decomposes as $p \mcO =
\mathfrak{p}_1 \mathfrak{p}_2 \cdots \mathfrak{p}_g$, where the
$\mathfrak{p}_i$ are distinct prime ideals of $\mcO$ of inertia
degrees~$f_i$.  Then $\mcO \otimes_{\Z} \Z_p \simeq \mfO_1 \times
\cdots \times \mfO_g$, where each $\mfO_i / \Z_p$ is an unramified
extension of inertia degree~$f_i$.  Therefore,
$$ \mcL(\mcO \otimes_{\Z} \Z_p) \simeq \mcL(\mfO_1) \times \cdots
\times \mcL(\mfO_g).$$ Hence, by Theorem~\ref{thm:main.local} we
have
$$ \zeta^{\ideal}_{\mcL(\mcO),p}(s) = \zeta^{\ideal\;\Zp}_{\mcL(\mcO
  \otimes_{\Z} \Z_p)}(s) =
W^{\ideal}_{(\mcL,\dots,\mcL),(f_1,\dots,f_g)}(p,p^{-s})
$$ for an explicit rational function
$W^{\ideal}_{(\mcL,\dots,\mcL),(f_1,\dots,f_g)}\in \Q(X,Y)$.  Setting
$W^{\ideal}_{\mcL,\boldsymbol{f}} =
W^{\ideal}_{(\mcL,\dots,\mcL),(f_1,\dots,f_g)}$, we obtain
Theorem~\ref{thm:main.global}.  The functional equation of Theorem~\ref{thm:main.global} follows from that of Theorem~\ref{thm:main.local} since $n = \sum_{i = 1}^g f_i$ as $p$ is unramified in $\mathcal{O}$.

\begin{rem} 
  Our description of the rational function
  $W^{\ideal}_{\boldsymbol{\mcL},\boldsymbol{f}}$ is so explicit that
  one may, in principle, read off the (local) \emph{abscissa of
    convergence} $\alpha^{\triangleleft\,\Lri}_{L(\Lri)}$ of
  $\zeta^{\triangleleft\,\Lri}_{L(\Lri)}(s)$, viz.
$$\alpha^{\triangleleft\,\Lri}_{L(\Lri)} := \inf\left\{ \alpha\in
  \R_{>0} \mid \zeta^{\triangleleft\,\Lri}_{L(\Lri)}(s) \textup{
    converges on }\{ s\in\C \mid \Re(s) > \alpha\} \right\} \in
  \Q_{>0};$$ cf.\ Remark~\ref{rem:local.abs.con}.
  \end{rem}

\begin{rem} \label{rmk:intro}
 We emphasize that Theorem~\ref{thm:main.local} makes no restriction
 on the residue characteristic of $\lri$. In this regard it
 strengthens, for the class of Lie rings under consideration, the
 result \cite[Theorem~1.2]{Voll/19}, which establishes the functional
 equation~\eqref{equ:funeq.local.newer} for all $\lri$ whose residue
 characteristic avoids finitely many prime numbers;
 cf.\ \cite[Corollary~1.3]{Voll/19} and see also~\cite[Theorem~1.7]{LeeVoll/21}. In the global contexts of ideal
 zeta functions of rings of the form $\mcL(\Gri)$ for number rings
 $\Gri$, Theorem~\ref{thm:main.local} shows that the finitely many
 Euler factors for which the functional
 equation~\eqref{equ:funeq.local.newer} fails must be among those
 indexed by primes that ramify in $\Gri$.
 
 In~\cite[Conjecture~1.4]{SV1/15} it was suggested that a functional
 equation should hold for {\emph{all}} local factors
 $\zeta^{\ideal}_{\mathfrak{f}_{2,2}(\mcO),p}(s)$, where
 $\mathfrak{f}_{2,2}$ is the Heisenberg Lie ring and $\mcO$ is a
 number ring; if $p$ ramifies in $\mcO$, then the symmetry factor must
 be modified from that of~\eqref{equ:funeq.local.newer}.  Some cases
 of the conjecture were proved in~\cite[Corollary~3.13]{SV2/16}.
 There is computational evidence, due to T.~Bauer, that other Lie
 rings in the class $\mfL$ also exhibit the remarkable property of the
 local factors $\zeta^{\ideal}_{\mcL(\mcO),p}(s)$ at ramified primes
 $p$ being described by rational functions satisfying functional
 equations.  However, these local factors cannot be computed by the
 methods of this paper; see Remark~\ref{rem:ram}.  Bauer's
 computations, together with the results of this paper, suggest the
 following natural question: how do the local factors
 $\zeta^{\ideal}_{\mcL(\mcO),p}(s)$ behave at ramified primes, and how
 does the structure of $\mcL$ govern their behaviour?
 
 Another natural problem is to improve upon Definition~\ref{def:L} by
 giving a conceptual characterization of the class of Lie rings to
 which our method, or a mild generalization thereof, applies.  For
 instance, forthcoming work of T.~Bauer extends our argument to
 explicitly compute the ideal zeta functions of central products of
 finitely many copies of Lie rings in the class $\mathfrak{L}$.  By
 contrast, non-uniform examples such as those of~\cite{duS-ecII/01,
   Voll/04} provide a limit on the applicability of these methods.
\end{rem}

\subsubsection{Previous and related work}\label{subsubsec:previous}
Theorems~\ref{thm:main.global} and~\ref{thm:main.local} generalize and
unify several previously known results.
  \begin{enumerate} 
  \item The most classical may be the formula for the $\lri$-ideal
    zeta function
     \begin{equation}\label{equ:abelian}
       \zeta_{\lri^n}(s) := \zeta^{\ideal\,\lri}_{\lri^n}(s) =
       \prod_{i=1}^n\frac{1}{1-q^{-s+i-1}}
     \end{equation}  of the (abelian Lie) ring $\lri^n = \mfg_{0,n}(\lri) =
       \mfg_{n,0}(\lri)$; cf.\ \cite[Proposition~1.1]{GSS/88}.
     \item The ideal zeta functions of the so-called \emph{Grenham Lie
       rings} $\mfg_{1,d}$ were given in \cite[Theorem~5]{Voll/05}.
     \item Formulae for the ideal zeta functions of the free
       class-$2$-nilpotent Lie rings $\mff_{2,d}$ on $d$ generators
       are the main result of \cite{Voll/05a}.
   \item The paper \cite{SV1/15} contains formulae for all local
     factors of the ideal zeta functions of the Lie rings
     $\mff_{2,2}(\mcO) = \mfg_{1,1}(\mcO) = \mathfrak{h}_1(\mcO)$,
     i.e.\ the \emph{Heisenberg Lie ring} over an arbitrary number
     ring~$\mcO$, which are indexed by primes unramified in
     $\mcO$. The uniform nature of these functions had already been
     established in \cite[Theorem~3]{GSS/88}. Formulae for factors
     indexed by non-split primes are given in \cite{SV2/16}.
     \item 
     The ideal zeta functions of the Lie rings $\mathfrak{h}_d \times
     \mfo^r$ were computed in~\cite[Proposition~8.4]{GSS/88}, whereas
     for the direct products $\mathfrak{h}_d \times \cdots \times
     \mathfrak{h}_d$ they were computed in~\cite{Bauer/13}.
   \item The ideal zeta function of the Lie ring $\mfg_{2,2}$ was
     computed in \cite[Theorem~11.1]{Paajanen/08}.
 \end{enumerate}

 Some of the members of the family of Lie rings $\mfL$ have previously
 been studied in the context of related counting problems, each
 leading to a different class of zeta functions. We mention
 specifically four such classes: first, the \emph{subring zeta
   function} of a (class-$2$-nilpotent Lie) ring~$\mcL$, enumerating
 the finite index subrings of~$\mcL$; second, the \emph{proisomorphic
   zeta function} of~$G(\mcL)$, the finitely generated nilpotent group
 associated to $\mcL$ via the Mal'cev correspondence, enumerating the
 subgroups of finite index of $G(\mcL)$ whose profinite completions
 are isomorphic to that of $G(\mcL)$; third, the \emph{representation
   zeta function} of $G(\mcL)$, enumerating the twist-isoclasses of
 complex irreducible representations of $G(\mcL)$; fourth, the
 \emph{class number zeta function} of $G(\mcL)$, enumerating the class
 numbers (i.e.\ numbers of conjugacy classes) of congruence quotients
 of this group (see \cite{Lins1/19}).

 The subring zeta functions of the Grenham Lie rings $\mfg_{1,d}$ were
 computed in \cite{VollBLMS/06}. Those of the free class-$2$-nilpotent
 Lie rings $\mff_{2,d}$ are largely unknown, apart from $d=2$
 (\cite{GSS/88}) and $d=3$ (\cite[Theorem~2.16]{duSWoodward/08}, due
 to G.\ Taylor).  The proisomorphic zeta functions of the members of a
 combinatorially defined class of groups that includes the Grenham
 groups $G(\mfg_{1,d})$ were computed in \cite{BermanKlopschOnn/18},
 their normal zeta functions in \cite{Voll/20}.  Moreover,
 {\emph{all}} Euler factors of the proisomorphic zeta functions of
 $G(\mathfrak{f}_{2,d}(\mathcal{O}))$ and
 $G(\mathfrak{h}_d(\mathcal{O}))$, where $\mathcal{O}$ is an arbitrary
 number ring, as well as of the base extensions to $\mathcal{O}$ of
 the groups considered in~\cite{BermanKlopschOnn/18} and some other
 families of nilpotent groups of unbounded class, were computed
 in~\cite{BGS/22}.  The representation zeta functions of the free
 class-$2$-nilpotent groups $F_{2,d}(\mcO) =G(\mff_{2,d}(\mcO))$ were
 computed in \cite[Theorem~B]{StasinskiVoll/14}, those of the groups
 $G(\mfg_{d,d'}(\mcO))$ in \cite[Theorem~A]{Zordan/17}.  In these
 cases, not only is there a fine Euler decomposition, but the rational
 function realizing the fine Euler factors is independent of $\mcO$
 and of the prime.  The class number zeta functions of the groups
 $F_{2,d}(\mcO)$ and $G(\mfg_{d,d}(\mcO))$, which may be found in
 \cite[Corollary~1.5]{Lins2/20}, satisfy the same properties.
 
 Combinatorial structures similar to those employed in the present
 article were also used in~\cite{RV/19}. In that paper, they were used to
 produce explicit formulae for zeta functions enumerating conjugacy
 classes of the cographical groups defined in~\cite[\S3.4]{RV/19}.

\subsubsection{Methodology}\label{subsubsec:methodology}
 Our approach to computing the explicit rational functions mentioned
 in Theorems~\ref{thm:main.local} and \ref{thm:main.global} hinges on
 the following considerations. Fix a prime $p$ and a
 class-$2$-nilpotent Lie ring $\mcL$ and consider, for simplicity, the
 pro-$p$ completion $L=\mcL(\Zp)$ of~$\mcL$. Given a $\Zp$-sublattice
 $\Lambda\leq L$, set $\ol{\Lambda} := ({\Lambda}+L')/L'$ and
 $\Lambda':= \Lambda \cap L'$. Here we write $L'=[L,L]$ for the
 commutator subring of $L$. Clearly, $\Lambda$ is a $\Zp$-ideal of $L$
 if and only if $[\ol{\Lambda},L] \subseteq \Lambda'$. This allows us,
 for fixed $\ol{\Lambda}$, to reduce the problem of enumerating such
 $\Lambda'$\ to the problem of enumerating subgroups of the finite
 abelian $p$-group $L'/ [\ol{\Lambda},L]$. The isomorphism type of the
 latter is given by the ($\Zp$)-\emph{elementary divisor type} of
 $[\ol{\Lambda},L]$ in~$L'$, viz.\ the partition $\lambda(\Lambda) =
 (\lambda_1,\dots,\lambda_{c})$ with the property that
\begin{equation*}
L^\prime / [\ol{\Lambda}, {L}] \simeq \Zp / (p^{\lambda_1})\times
\cdots \times \Zp / (p^{\lambda_{c}}).
\end{equation*}
 For general Lie rings $\mcL$, controlling this type for varying
 $\Lambda$ is a hard problem that may be dealt with by studying
 suitably defined $p$-adic integrals with sophisticated tools from
 algebraic geometry, including Hironaka's resolution of singularities
 in characteristic zero.

 If, however, $\mcL$ is an element of the class $\mfL$, then the
 elementary divisor type of $[\overline{\Lambda}, L]$ is determined,
 in a complicated but \emph{combinatorial} manner, by so-called
 ``projection data''; cf.\ Definition~\ref{dfn:projection.data}.
 These are the respective elementary divisor types of the projections
 of $\ol{\Lambda}$ onto various direct summands of $L/L'$.  The
 technical tool we use to keep track of the resulting infinitude of
 finite enumerations are the {\emph{generalized Igusa functions}}
 introduced in Section~\ref{sec:general.igusa}. An intrinsic advantage
 of this combinatorial point of view over the general (and typically
 immensely more powerful) algebro-geometric approach is that,
 structurally, $\Zp$ only enters as a compact discrete valuation
 ring. The effect of passage to various other such local rings,
 including those of positive characteristic, is therefore easy to
 control.

 For an informal overview of the combinatorial aspects of our approach
 to counting $\lri$-ideals, see Section~\ref{subsec:overview}.

 \subsection{Counting ideals with generalized Igusa
   functions}\label{subsec:intro.gen.igusa}
Our key to Theorem~\ref{thm:main.local} is the systematic deployment
of a new class of combinatorially defined multivariate rational
functions, which we call \emph{generalized Igusa functions}. Expecting
that they will be of interest independently of questions pertaining to
ideal growth in rings, we explain them here separately.

Generalized Igusa functions interpolate between two well-used classes
of rational functions:
  \begin{enumerate}
  \item A function we refer to as the \emph{Igusa zeta function of degree
      $n$} plays a key role in numerous previous computations (for
    instance \cite{Voll/05, VollBLMS/06, Voll/05a, Paajanen/08,
      StasinskiVoll/14, SV1/15, SV2/16, CSV/18, Voll/20}):
  $$I_n(Y;X_1,\dots,X_n) = \sum_{I\subseteq \{1,\dots,n\}}
  \binom{n}{I}_Y \prod_{i\in I}\frac{X_i}{1-X_i} \in
  \Q(Y,X_1,\dots,X_n).$$ Here, $\binom{n}{I}_Y$ denotes the Gaussian
  multinomial; see \eqref{def:gaussian.multi}. For instance,
  \begin{equation} \label{equ:on.igusa.identity}
  \zeta_{\lri^n}(s) =
  I_n(q^{-1};((q^{n-i-s})^{i})_{i=1}^n);
  \end{equation}
   cf.\ \eqref{equ:abelian} and \cite[Example~2.20]{Voll/11}.

\item In \cite{SV1/15}, the \emph{weak order zeta function}
  \begin{equation}\label{equ:wozeta}I^{\textup{wo}}_n((X_I)_{I\in \mcP([n])\setminus\{\varnothing\}}) =
  \sum_{I_1\subsetneq \dots \subsetneq I_l \subseteq [n]}
  \prod_{j=1}^l \frac{X_{I_j}}{1-X_{I_j}} \in \Q((X_I)_{I\in
    \mcP([n])\setminus\{\varnothing\}})\end{equation} played a decisive role;
  cf.\ \cite[Definition~2.9]{SV1/15}.
  \end{enumerate}

  The main protagonist of Section~\ref{sec:general.igusa} is the
  \emph{generalized Igusa function}
  $I^{\textup{wo}}_{\underline{n}}(Y_1,\dots,Y_m;\bfX)$, a rational
  function associated to a composition
  $\underline{n}=(n_1,\dots,n_m)$, with variables $\bfX$ indexed by
  the subwords of the word $a_1^{n_1}\dots a_m^{n_m}$ in ``letters''
  $a_1,\dots,a_m$; cf.\ Definition~\ref{def:igusa.wo} for details. It
  interpolates between the two classes of rational functions just
  mentioned: the Igusa function of degree $n$ for the trivial
  composition $(n)$ and the weak order zeta function for the all-one
  composition $(1,\dots,1)$ of $n$; see Example~\ref{exm:gen.igusa}.

  \begin{rem}
    Igusa functions are not to be confused with, but are related to, a
    class of $p$-adic integrals known as Igusa's local zeta function;
    cf.~\cite{Denef/91}. For a detailed explanation of the connection
    between $I_n$ and work of Igusa, as well as further
    generalizations and applications, see \cite{KlopschVoll/09}.
  \end{rem}

  \subsection{Organization and notation}

  \subsubsection{} In Section~\ref{sec:prelim} we recall a number of
  preliminary notions and results used to enumerate lattices and
  finite abelian $p$-groups. In Section~\ref{sec:general.igusa} we
  define the generalized Igusa functions and prove that they satisfy
  functional equations. In Section~\ref{sec:main.results}, these new
  functions are put to use to compute a general formula
  (cf.\ Theorem~\ref{thm:zeta.explicit}) for local ideal zeta
  functions of Lie rings satisfying the general combinatorial
  Hypothesis~\ref{hypothesis}. In Section~\ref{sec:application} we
  verify that the members of the class $\mfL$
  (cf.\ Definition~\ref{def:L}) satisfy Hypothesis~\ref{hypothesis},
  complete the proof of Theorem~\ref{thm:main.local}, and attend
  to a number of special cases.
  
  \subsubsection{} We write $\N=\{1,2,\dots\}$ and, for a subset
  $X\subseteq \N$, set $X_0=X \cup \{0\}$. For $m,n\in\N_0$ we denote
  $[n]=\{1,\dots,n\}$, $[n,m]=\{n,n+1,\dots,m\}$, and
  $(n,m)=\{n+1,\dots,m-1\}$.  Given a finite subset $J\subseteq \N_0$,
  we write $J=\{j_1,\dots,j_r\}_<$ to imply that $j_1 < \dots <
  j_r$. We write $J-n$ for the set $\{j-n \mid j\in J\}$. The power
  set of a set $S$ is denoted $\mcP(S)$.

  A \emph{composition of $n$ with $r$ parts} is a sequence
  $(\lambda_1,\dots,\lambda_r)\in\N_0^r$ such that
  $\sum_{i=1}^r \lambda_i = n$. A \emph{partition of $n$ with $r$
    parts} is a composition of $n$ with $r$ parts such that
  $\lambda_1\geq \dots \geq \lambda_r$. We occasionally obtain
  partitions from multisets by arranging their elements in
  non-ascending order. Our notation for the dual partition of a
  partition $\lambda$ is~$\lambda'$. Given partitions
  $\mu=(\mu_1,\dots,\mu_c)$ and $\lambda=(\lambda_1,\dots,\lambda_c)$
  we write $\mu\leq\lambda$ if $\mu_i\leq \lambda_i$ for all
  $i\in[c]$, i.e.\ if the Young diagram of $\mu$ is included in the
  Young diagram of $\lambda$.

 We write $t=q^{-s}$, where $s$ denotes a complex variable.

\section{Preliminaries}\label{sec:prelim}

In this preliminary section, we collect some fundamental notions.

\subsection{Gaussian binomials and classical Igusa functions}\label{subsec:classical.igusa}
For a variable $Y$ and integers $a,b\in\N_0$ with $a\geq b$, the
associated \emph{Gaussian binomial} is
\begin{equation*}\label{def:gauss}
  \binom{a}{b}_Y = \frac{\prod_{i=a-b+1}^a (1-Y^i)}{\prod_{i=1}^b
    (1-Y^i)}\in \Z[Y].
\end{equation*}
A simple computation shows that
\begin{equation}\label{eq:funeq.gauss}
\binom{a}{b}_{Y^{-1}} = Y^{b(b-a)} \binom{a}{b}_Y.
\end{equation}
Given $n\in\N$ and a subset $J = \{j_1,\dots,j_r\}_< \subseteq [n-1]$,
the associated \emph{Gaussian multinomial} is defined as
\begin{equation}\label{def:gaussian.multi} \binom{n}{J}_Y = \binom{n}{j_{r}}_Y
  \binom{j_{r}}{j_{r-1}}_Y \cdots \binom{j_2}{j_1}_Y \in \Z [Y].
  \end{equation}

We omit the proof of the following simple lemma, which is similar to
\cite[Lemma~2.14]{SV1/15}.
\begin{lem}\label{lem:binom}
  Let $n\in\N$ and $P = \{p_1,\dots,p_r\}_< \subseteq J \subseteq
  [n-1]$. Then
  $$\binom{n}{J}_{Y} = \binom{n}{P}_Y \prod_{j=1}^r
  \binom{p_j-p_{j-1}}{J \cap (p_{j-1},p_j)-p_{j-1}}_Y.$$
\end{lem}

\begin{dfn}(\cite[Definition~2.5]{SV1/15})\label{def:igusa} 
  Let $\wo\in\N$. Given variables $Y$ and $\bfX=(X_1,\dots,X_\wo)$, we
  define the \emph{Igusa functions of degree $n$}
\begin{align*}
I_\wo(Y; \bfX)& =
\frac{1}{1-X_\wo}\sum_{I\subseteq[\wo -1]} \binom{\wo}{I}_{Y} \prod_{i\in
  I}\frac{X_i}{1-X_i} = \sum_{I \subseteq [\wo]}\binom{\wo}{I}_{Y} \prod_{i\in
  I}\frac{X_i}{1-X_i} \;\in\Q(Y,X_1,\dots,X_\wo),\\ 
 I_\wo^\circ (Y; \bfX) & =
\frac{X_\wo}{1-X_\wo}\sum_{I\subseteq[\wo-1]} \binom{\wo}{I}_Y \prod_{i\in
  I}\frac{X_i}{1-X_i}\;\in\Q(Y,X_1,\dots,X_\wo).
\end{align*}
\end{dfn}
An important feature of these functions is that they satisfy a
functional equation upon inversion of the variables; it is immediate
from~\cite[Theorem~4]{Voll/05} that, for all $\wo \in \N$,
\begin{align}
  I_\wo(Y^{-1};\bfX^{-1})& = (-1)^\wo
                           X_{\wo}Y^{-\binom{\wo}{2}}I_\wo(Y;\bfX)\label{eq:funeq.igusa},
  \\ I^\circ_\wo(Y^{-1};\bfX^{-1})& =
                                    (-1)^\wo X_\wo^{-1}Y^{-\binom{\wo}{2}}I^\circ_\wo(Y;\bfX).\label{eq:funeq.igusa.circ}
\end{align}

\subsection{Subgroups of finite abelian groups, Birkhoff's formula,
  and Dyck words}\label{subsec:prelim.dyck}
It is well-known that, given a pair of partitions $\mu\leq \lambda$
and a prime $p$, the number $a(\lambda,\mu;p)$ of finite abelian
$p$-groups of isomorphism type $\mu$ contained in a fixed finite
abelian $p$-group of isomorphism type $\lambda$ is given by a
polynomial in $p$. More precisely, set
\begin{equation}\label{def:birkhoff.poly}
  \alpha(\lambda,\mu;Y) = \prod_{k\geq 1}
  Y^{\mu_k'(\lambda_k'-\mu_k')}\binom{\lambda_k'-\mu_{k+1}'}{\lambda_k'-\mu_k'}_{Y^{-1}}
  \in\Q[Y],
\end{equation}
{where $\lambda'$ and $\mu'$ are the dual partitions
of  $\lambda$ and $\mu$, respectively.}  Then, by a result going back to
work of Birkhoff \cite{Birkhoff/35}, $a(\lambda,\mu;p) =
\alpha(\lambda,\mu;p)$ (see \cite[Lemma~1.4.1]{Butler/94}; cf.\ also
\cite{Dyubyuk/48, Delsarte/48, Yeh/48}).

In practical applications invoking infinitely many instances of this
formula, as in \cite{SV1/15, LeeVoll/18}, it proved advantageous to
sort pairs of partitions by their ``overlap types'' indexed by Dyck
words, as we now recall.

Let~$c\in\N$. A \emph{Dyck word of length $2c$} is a word
$$w=\bfz^{L_1}\bfone^{M_1}\bfz^{L_2
  -L_1}\bfone^{M_2-M_1}\cdots\bfz^{L_r-L_{r-1}}\bfone^{M_r-M_{r-1}}$$
in letters $\bfo$ and $\bfz$, both occurring $c$ times each (hence
$L_r=M_r=c$), and, crucially, no initial segment of $w$ contains more
ones than zeroes (or, equivalently, $M_i \leq L_i$ for all
$i\in[r]$). Here, both the $L_i$ and $M_i$ are assumed to be
positive. Below, we will use the notational conventions $M_0=L_0=0$
and $L_{r+1}=L_r=c$, $M_{r+1}=M_r=c$. We write $\mcD_{2c}$ for the set
of all Dyck words of length $2c$. See \cite[Section~2.4]{SV1/15} or
\cite[Example~6.6.6]{Stanley/99} for further details on Dyck words.

We say that two partitions $\lambda$ and $\mu$, both with $c$ parts
and satisfying $\mu\leq \lambda$, have \emph{overlap type
  $w\in\mcD_{2c}$}, written $w(\lambda,\mu)=w$, if
\begin{multline} \label{mult:limi} \lambda_1 \geq \cdots \geq
  \lambda_{L_1} \geq \mu_1 \geq \cdots \geq \mu_{M_1} > \lambda_{L_1 +
    1} \geq \cdots \geq \lambda_{L_2} \geq \mu_{M_1 + 1} \geq \cdots
  \geq \mu_{M_2} > \cdots \\ > \lambda_{L_{r-1}+1} \geq \cdots \geq
  \lambda_c \geq \mu_{M_{r-1} + 1} \geq \cdots \geq \mu_c.
\end{multline}
In Definition~\ref{dfn:overlaps} we slightly modify this definition to
suit the specific needs of this paper.

\subsection{Gaussian multinomials and symmetric groups}

In Section~\ref{sec:general.igusa}, the following Coxeter group
theoretic interpretation of the Gaussian
multinomials will be useful. Recall that the symmetric group $W=S_n$
of degree $n$ is a Coxeter group, with Coxeter generating system
$S=(s_1,\dots,s_{n-1})$, where $s_i=(i\,i+1)$ denotes the standard
transposition. The \emph{Coxeter length} $\ell(w)$ of an element $w\in
S_n$ is the length of a shortest word for $w$ with elements
from~$S$. We define the (\emph{right}) \emph{descent set} $\Des(w) =
\{ i\in[n-1] \mid \ell(w s_i) < \ell(w) \}$. It is well-known
(\cite[Proposition~1.7.1]{Stanley/12}) that the Gaussian multinomials
\eqref{def:gaussian.multi} satisfy
\begin{equation}\label{eq:coxlength}
  \binom{n}{J}_Y = \sum_{w\in S_n, \, \Des(w)\subseteq J} Y^{\ell(w)}.
\end{equation}
Let $w_0$ denote the unique $\ell$-longest element in $S_n$, of length
$\ell(w_0)=\binom{n}{2}$. Then, for all $w\in S_n$,
\begin{equation}
  \ell(w w_0) = \ell(w_0)-\ell(w), \quad \quad
  \Des(w w_0)  =[n-1]\setminus \Des(w); \label{eq:cox.identities}
\end{equation}
cf.\ \cite[Section~1.8]{Humphreys/90}.

\subsection{A note on ramification}
Let $\mfo$ be a compact discrete valuation ring of arbitrary
characteristic.  Let $\mathfrak{m}$ denote the maximal ideal of $\mfo$
and let $\pi \in \mfo$ be a uniformizer, i.e.\ any element such that
$\mathfrak{m} = \pi \mfo$.  Let $\mfO$ be a finite extension of
$\mfo$, with maximal ideal $\mathfrak{M}$ and uniformizer $\Pi$.  Let
$f = [\mfO / \mathfrak{M} : \mfo / \mathfrak{m}]$ be the inertia
degree of the extension $\mfO / \mfo$, and let $e$ be its ramification
index; this means that $\pi \mfO = \mathfrak{M}^e$.  We will need the
following standard fact.

\begin{lem} \label{lem:oedt.to.zp}
Let $\mfO$ be a finite extension of $\mfo$ with ramification index $e$
and inertia degree~$f$.  Let 
$\tau \in \N_0$.  Suppose that $\tau = g e + h$, where $g \in \N_0$ and
$h \in [e - 1]_0$.  Then the following isomorphism of $\mfo$-modules
holds:
$$ \mfO / \mathfrak{M}^\tau \simeq \left( \mfo / \mathfrak{m}^{g+1}
\right)^{hf} \times \left( \mfo / \mathfrak{m}^g \right)^{(e-h)f}.$$ In
particular, if $\mfO / \mfo$ is unramified (i.e.~$e = 1$), then $\mfO
/ \mathfrak{M}^\tau \simeq (\mfo / \mathfrak{m}^\tau)^f$ as $\mfo$-modules.
\end{lem}
\begin{proof}
Let $\beta_1, \dots, \beta_f \in \mfO$ be a collection of elements
whose reductions modulo $\mathfrak{M}$ constitute an $\mfo /
\mathfrak{m} $-basis of the residue field $\mfO / \mathfrak{M}$.  The set $\{ \beta_i \Pi^j \mid i \in [f], j \in [e-1]_0 \}$
provides a basis for $\mfO$ as an $\mfo$-module; see, for instance,
the proof of~\cite[Proposition II.6.8]{Neukirch/99}.  Now it is clear
that $\mathfrak{M}^\tau = \Pi^\tau \mfO$ is the $\mfo$-linear span of
the set
\begin{equation*}
  \{ \pi^{g+1} \beta_i \Pi^j \mid i \in [f], j \in [0, h-1] \} \cup \{
  \pi^g \beta_i \Pi^j \mid i \in [f], j \in [h, e-1] \}. \qedhere
\end{equation*} 
\end{proof}

\begin{dfn} \label{dfn:dictionary}
For $\tau \in \N_0$ and $e, f \in \N$, let $\{ \tau\}_{e,f} = \{
(g+1)^{(hf)},g^{((e-h)f)}\}$ be the $ef$-element multiset consisting
of the element $g+1$ with multiplicity $hf$ and the element $g$ with
multiplicity $(e-h)f$, where $\tau = ge + h$ and $h \in [e-1]_0$,
as in Lemma~\ref{lem:oedt.to.zp}.
\end{dfn}

\section{Generalized Igusa functions}\label{sec:general.igusa}

In Section~\ref{subsec:gen.igusa.funeq} we introduce generalized Igusa
functions and prove that they satisfy functional equations. In
Section~\ref{subsec:weak.igusa} we record an identity involving weak
order zeta functions, motivated by our applications of Igusa functions
in ideal growth in Section~\ref{sec:application}.

\subsection{Generalized Igusa functions and their functional equations}\label{subsec:gen.igusa.funeq}
Let $\underline{n}=(n_1,\dots,n_m)$ be a composition of $N =
\sum_{i=1}^m n_i$ with $m$ parts.  Consider the poset
$C_{\underline{n}}$ of subwords of the word
$v_{\underline{n}}:=a_1^{n_1}a_2^{n_2}\dots a_m^{n_m}$ in ``letters''
$a_1,a_2,\dots,a_m$, each occurring with respective
multiplicity~$n_i$. This poset is naturally isomorphic to the
lattice $$C_{n_1}\times\dots \times C_{n_m},$$ the product of the
chains of lengths $n_i$ with the product order, which we denote
by~``$\leq$''. We write $\widehat{1}=v_{\underline{n}}$ and
$\widehat{0}$ for the empty word.

We denote by $\WO_{\underline{n}}$ the chain (or order) complex of
$C_{\underline{n}}$.  An element $V\in \WO_{\underline{n}}$ is a
(possibly empty) chain, or flag, of non-empty subwords of
$v_{\underline{n}}$, of the form $V=\{v_1 < \dots < v_t\}$. On
$\WO_{\underline{n}}$ we consider the partial order defined by
refinement of flags, also denoted by~``$\leq$''. Consider the natural
map
\begin{align*}
  \underline\pi:C_{\underline{n}} &\rarr
               [n_1]_0\times\cdots\times[n_m]_0,\\ v =
               a_1^{\alpha_1}\dots a_m^{\alpha_m} &\mapsto
               (\alpha_1,\dots,\alpha_m)=:(\pi_{1}(v),\dots,\pi_{m}(v)).
  \end{align*}
{The degree of the word  $v =
               a_1^{\alpha_1}\dots a_m^{\alpha_m}$ is $|v|:=\sum_{i=1}^m \alpha_i$.}
  \begin{dfn}
    We consider the induced morphism of posets
\begin{align*}
 \underline \varphi:\WO_{\underline{n}} &\to \prod_{i=1}^m
  \mathcal P ([n_i-1]),\\ V = \{v_1 < \dots < v_t\} &\mapsto \left(
  \left\{\pi_{i}(v_j) \mid j\in[t] \right\} \cap [n_i-1] \right)_{i=1}^m
  =: \left( \phi_{i}(V)\right)_{i=1}^m.
\end{align*} 
 We say that
$V\in \WO_{\underline{n}}$ has \emph{full projections} if
$$ \underline\varphi (V) = ([n_1-1
  ], \dots, [n_m-1])=:K.$$
\end{dfn}

\begin{center}
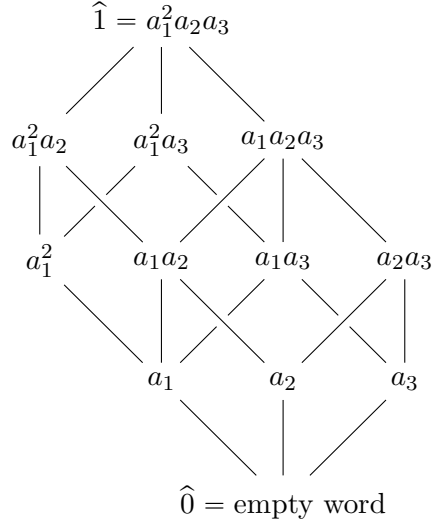

\begin{tikzpicture}[scale=0.8]
  \node (max) at (-2,6) {$\widehat{1}=a_1^2a_2a_3$};
  \node (h) at (-2,4) {$a_1^2a_3$};
  \node (m) at (0,4) {$ a_1a_2a_3$};
  \node (i) at (-4,2) {$a_1^2$};
  \node (ii) at (-4,4) {$a_1^2a_2$};
  \node (a) at (-2,2) {$a_1a_2$};
  \node (b) at (0,2) {$a_1a_3$};
  \node (c) at (2,2) {$a_2a_3$};
  \node (d) at (-2,0) {$a_1$};
  \node (e) at (0,0) {$a_2$};
  \node (f) at (2,0) {$a_3$};
  \node (min) at (0,-2) {$\widehat{0}=\text{empty word}$};
  \draw (min) -- (d) -- (a) -- (m) -- (b) -- (f) (i) -- (h) -- (b)
  (e) -- (min) -- (f) -- (c) -- (m) (max) -- (ii) -- (a)
  (h) -- (max) -- (m) 
  (ii) -- (i) -- (d) -- (b)  ;
  \draw[preaction={draw=white, -,line width=6pt}] (a) -- (e) -- (c) (m)--(a)--(ii);
\end{tikzpicture}
\captionof{figure}{The poset $C_{\underline{n}}$ for $\underline{n} = (2,1,1)$.}
\end{center}

\begin{rem}\label{rem:squarefree}
 We observe that the flag $V=\{v_1<\dots<v_t\}\in\WO_{\underline{n}}$
 has full projections if, and only if, for all $j\in[t]_0$, the word
 $v_{j+1} / v_{j}$ is squarefree, i.e.\ contains at most one copy of
 each letter $a_1, \dots, a_m$.
  \end{rem}

\begin{dfn} \label{dfn:chain.binomial} Let
  $V=\{v_1  <  \dots <  v_t\}\in \WO_{\underline{n}}$. We define
  $$W_{V}(\bfX) = \prod_{j=1}^t \frac{X_{v_j}}{1-X_{v_j}} \in
  \Q(X_{v_1},\dots,X_{v_t})$$ and
$$\binom{\underline{n}}{V}_\bfY = \prod_{i=1}^m
\binom{n_i}{\varphi_{i}(V)}_{Y_i}\in\Q(Y_1,\dots,Y_m),$$
where $ \underline\varphi(V)=\left(\varphi_{1}(V),\dots,\varphi_{m}(V)\right)$.
\end{dfn}

\begin{exm}
 Let $\underline{n}=(3,2,2)$. The flag $V=\{a_2a_3 < a_1a_2^2 a_3\}
 \in \WO_{(3,2,2)}$ does not have full projections, as $
 \underline\varphi(V)=(\{1\},\{1\},\{1\})$. We note that
 $$W_V(\bfX)= \frac{X_{a_2 a_3}X_{a_1 a_2^2 a_3}}{(1-X_{a_2
     a_3})(1-X_{a_1 a_2^2 a_3})}$$ and
$$\binom{\underline{n}}{V}_\bfY =\binom{3}{1}_{Y_1} \binom{2}{1}_{Y_2} \binom{2}{1}_{Y_3} = (1 + Y_1 + Y_1^2)(1 + Y_2)(1+Y_3). $$
\end{exm}

The following is the key combinatorial tool of this paper.
\begin{dfn}\label{def:igusa.wo} The \emph{generalized Igusa function
    associated with the composition $\underline{n}$} is
$$I^{\textup{wo}}_{\underline{n}}(\bfY;\bfX) := \sum_{V\in
    \WO_{\underline{n}}} \binom{\underline{n}}{V}_\bfY W_V(\bfX) \in
  \Q(Y_1,\dots,Y_m,(X_r)_{r \leq v_{\underline{n}}}),$$
\end{dfn}

\begin{exm} \label{exm:gen.igusa}
  \ \begin{enumerate}
  \item For $\underline{n}=(N)$, the trivial composition of $N$, we recover
    $I^{\textup{wo}}_{(N)}(\bfY;\bfX) = I_{N}(Y;\bfX)$, the classical
    Igusa zeta function recalled in Definition~\ref{def:igusa}.
  \item For $\underline{n}=(1,\dots,1)$, the all-one composition of
    $N$, we recover $I^{\textup{wo}}_{(1,\dots,1)}(\bfY;\bfX) =
    I^{\textup{wo}}_{N}(\bfX)$, the weak order zeta function recalled
    in \eqref{equ:wozeta}. We note that the variables $\bfY$ do not
    appear in this case, as all the polynomials
    $\binom{\underline{n}}{V}_{\bfY}$ are equal to the constant $1$.
    \item For $\underline{n}=(2,1)$ we obtain
\begin{multline*}
   I^{\textup{wo}}_{(2,1)}(\bfY;\bfX)=\frac{1}{1-X_{a_1^2 a_2}}\left(1+\frac{X_{a_2}}{1-X_{a_2}}+\frac{X_{a_1^2}}{1-X_{a_1^2}}+\right.\\
    (1+Y_1)\left(\frac{X_{a_1}}{1-X_{a_1}}+\frac{ X_{a_1a_2}}{1-X_{a_1a_2}}+\frac{ X_{a_1}}{1-X_{a_1}}\frac{X_{a_1a_2}}{1-X_{a_1a_2}}
    +\right. \\ \left.\left. \frac{X_{a_1}}{1-X_{a_1}}\frac{X_{a_1^2}}{1-X_{a_1^2}}+\frac{X_{a_2}}{1-X_{a_2}}\frac{X_{a_1 a_2}}{1-X_{a_1 a_2}}\right)\right). \end{multline*}      
  \end{enumerate}
\end{exm}

\begin{rem}
  Generalized Igusa functions associated with the all-one compositions
  also coincide with certain instances of generating functions
  associated with chain partitions in
  \cite[Section~4.9]{BeckSanyal/18}.
\end{rem}
The following ``combinatorial reciprocity theorem'' is the main result
of this section.

\begin{thm}\label{thm:funeq}
  The generalized Igusa function associated with the composition
  $\underline{n}$ of $N = \sum_{i=1}^m n_i$ satisfies the following functional
  equation:
  $$I^{\textup{wo}}_{\underline{n}}(\bfY^{-1};\bfX^{-1}) = (-1)^N
  X_{v_{\underline{n}}}\left(\prod_{i=1}^m Y_i^{-
      \binom{n_i}{2}}\right)
  I^{\textup{wo}}_{\underline{n}}(\bfY;\bfX).$$
\end{thm}

For the proof of Theorem~\ref{thm:funeq} we require a number of
preliminary results. The first records simple but crucial ``inversion
properties'' of the rational functions $W_V(\bfX)$.

\begin{lemma}\label{lem:inv.props}
 For all $V\in \WO_{\underline{n}}$,
$$W_V(\bfX ^{-1})=(-1)^{|V|}  \sum_{Q \leq  V} W_Q(\bfX).$$
\end{lemma}

\begin{proof}
  This is a trivial consequence of the observation that
  $$\frac{X^{-1}}{1-X^{-1}} = - \left( 1 +
  \frac{X}{1-X}\right).\qedhere$$
\end{proof}

We fix some notation used in the rest of this section. We let $
\WO^{\times}_{\underline{n}}$ denote the subcomplex of $
\WO_{\underline{n}}$ of flags of \emph{proper} subwords of
$v_{\underline n}$.  When dealing with tuples of sets, we will abuse
notation and use set theoretical operations for componentwise
operations. For instance, for $I=(I_1,\dots,I_m)\in
\prod_{i=1}^m\mathcal P([n_i-1])$ we write $I^c:=K\setminus I $ for $
([n_1-1]\setminus I_1,\dots,[n_m-1]\setminus I_m)$.

The following analogue of \cite[Lemma 7]{VollBLMS/06} is critical for
our analysis.

\begin{pro}\label{pro:mobius}
For all $I\in \prod_{i=1}^m\mathcal P([n_i-1])$,
\begin{equation}
\sum_{\substack{V\in \WO^{\times}_{\underline{n}} \\ \underline
    \varphi (V)\supseteq I}} W_V (\bfX ^{-1})
=(-1)^{N-1}\sum_{\substack{V\in \WO^{\times}_{\underline{n}}
    \\ \underline\varphi (V)\supseteq I^c}} W_V (\bfX ).
\end{equation} 
\end{pro}
\begin{proof}
Let $I\in \prod_{i=1}^m\mathcal P([n_i-1])$.  The inversion properties
established in Lemma~\ref{lem:inv.props} yield
$$\sum_{\substack{V\in \WO^{\times}_{\underline{n}} \\  \underline\varphi (V)\supseteq I}} W_V (\bfX ^{-1})=\sum_{\substack{V\in \WO^{\times}_{\underline{n}} \\  \underline\varphi (V)\supseteq I}} (-1)^{|V|}\sum_{Q\leq V}W_Q (\bfX)=\sum_{V\in \WO^{\times}_{\underline{n}}}W_V(\bfX)\sum_{\substack{S\geq V \\  \underline\varphi (S)\supseteq I}}(-1)^{|S|}.  $$

We are left with proving that, for all
$V\in \WO^{\times}_{\underline{n}}$,
\begin{equation}\label{eq:toprove}
  \sum_{\substack{S\geq V \\  \underline\varphi (S)\supseteq I}}(-1)^{|S|}=\begin{cases}
    (-1)^{N-1}, &\text{if }  \underline\varphi (V) \supseteq I^c, \\
    0, &\text{otherwise.}
\end{cases}
\end{equation}

Write $V=\{v_1 < \dots < v_t\}$ and set $v_0:= \widehat{0}$ and
$v_{t+1} := \widehat{1}$. Set
$$I_V:= I\cup  \underline\varphi(V) \in \prod_{i=1}^m\mathcal
P([n_i-1]).$$ The sum in~\eqref{eq:toprove} runs over refinements $S$
of the flag~$V$, subject to additional constraints on the projection
of $S$ given by $I$: we say that a refinement $S$ of $V$ is
\emph{admissible} if $ \underline\varphi(S)\supseteq I_V$. As
$ \underline\varphi$ is a poset morphism, the sum in
\eqref{eq:toprove} runs exactly over the admissible refinements
of~$V$.

We will construct such refinements of $V$ ``locally''. More precisely,
let $j\in [t]_0$. We say that $S$ is a \emph{refinement of $V$ between
  $v_j$ and $v_{j+1}$} if $S \geq V$ and $S$ and $V$ coincide outside
the interval~$[v_j,v_{j+1}]$. We further say that $S\geq V$ has
\emph{full projections between $v_j$ and $v_{j+1}$} if $
\underline\varphi(S\cap[v_j,v_{j+1}])$ is an $m$-tuple of intervals.

We set
$$I_V^{(j)} := \left( I_{V,i} \cap [\pi_{i}(v_j),\pi_{i}(v_{j+1})]\right)_{i=1}^m \in\prod_{i=1}^m\mathcal
  P([n_i-1]).$$ Informally, $I_V^{(j)}$ dictates the constraints on a
refinement $S$ of $V$ between $v_j$ and $v_{j+1}$. More precisely, we
say that a refinement $S$ of $V$ between $v_j$ and $v_{j+1}$ is
$j$-\emph{admissible} if
$ \underline\varphi(S)\supseteq I_V^{(j)}$. We further define
$$F_j(V,I) := \sum_{\substack{S\geq V\\ \textup{$j$-admissible}}}
(-1)^{|S\setminus V|} = \sum_{\substack{S\geq
    V\\ \textup{$j$-admissible}}} (-1)^{|(S\setminus V) \cup
  \{v_j,v_{j+1}\}|}.$$

Clearly, given $j$-admissible refinements $V_j$ of $V$ for all
$j\in[t]_0$, the flag $S := \bigcup_{j=0}^t V_j$ is an admissible
refinement of $V$ and any (``global'') admissible refinement of $V$
can be constructed in this way.  The sum in \eqref{eq:toprove} may
thus be rewritten as follows:
\begin{equation}\label{eq:prod}\sum_{\substack{S\geq V\\
       \underline\varphi (S)\supseteq I}}
  (-1)^{|S|}=\sum_{\substack{S\geq V\\  \underline\varphi
      (S)\supseteq I}}(-1)^{|V|+|S\setminus V|}=(-1)^t
  \sum_{\substack{S\geq V\\  \underline\varphi (S)\supseteq
      I}}(-1)^{|S\setminus V|}=(-1)^t \prod_{j=0}^t
  F_j(V,I).\end{equation} We prove~\eqref{eq:toprove} distinguishing
the two cases
\begin{enumerate}
  \item[(I)] $I_V =  \underline\phi(V)$ (equivalently, $I\subseteq \underline\varphi(V)$) and
  \item[(II)] $I_V\neq \underline\phi(V)$ (equivalently, $I\setminus
    \underline\varphi(V)\neq\emptyset$).
    \end{enumerate}

{\bf Case (I):} Assume first that
$I\subseteq \underline\varphi(V)$. In this case, the condition
$ \underline\phi(S)\supseteq I$ is trivially satisfied for any
flag $S \geq V$, as $ \underline\phi$ is a poset morphism, and
thus any refinement of $V$ is admissible.  Moreover, in this case,
$ \underline\phi(V)\supseteq I^c$ if and only if $V$ has full
projections. In other words, \eqref{eq:toprove} may be rewritten as
follows:
\begin{equation}\label{eq:toproveempty}
  \sum_{S\geq V}(-1)^{|S|}=\begin{cases} (-1)^{N-1}, &\text{if } V
  \text{ has full projections,} \\ 0, &\text{otherwise.}
\end{cases}
\end{equation}
Let $j\in[t]_0$. As in the case under consideration all local
refinements are $j$-admissible, $F_j(V,I)$ is given in terms of the
M\"obius function of the interval $[v_j,v_{j+1}]$ in the lattice
$C_{\underline{n}}$. Indeed, by Philip Hall's theorem (see, for
instance, \cite[Proposition~3.8.5]{Stanley/12}),
\begin{equation*}
F_j(V,I) = - \mu(v_j,v_{j+1})=
\begin{cases} (-1)^{|v_{j+1}|-|v_j|+1}, &\text{if } [v_j,v_{j+1}] \text{ is a Boolean algebra,} \\
0, &\text{otherwise};
\end{cases}
\end{equation*} 
cf.\ \cite[Example~3.8.4]{Stanley/12}.  Using \eqref{eq:prod} we may
therefore rewrite the LHS of \eqref{eq:toprove} as
$$(-1)^t \prod_{j=0}^t F_j(V,I)=(-1)^t \prod_{j=0}^t\left(
-\mu(v_j,v_{j+1})\right).$$ It is nonzero if and only if all of its
factors are nonzero. The interval $[v_j,v_{j+1}]$ is a Boolean algebra
if and only if the word $v_{j+1} / v_{j}$ is squarefree. By
Remark~\ref{rem:squarefree}, this happens for all $j\in[t]_0$ if and
only if $V$ has full projections. In this case we
obtain
\begin{multline*}
  \sum_{S\geq V}(-1)^{|S|} = (-1)^t \sum_{S\geq V}(-1)^{|S\setminus
    V|}= (-1)^t\prod_{j=0}^t F_j(V,I) =\\ (-1)^t \prod_{j=0}^t
  \left(-\mu(v_j,v_{j+1})\right) = (-1)^{2t+1}(-1)^{\sum_{j=0}^t
    (|v_{j+1}|-|v_j|)}=(-1)^{N-1},
\end{multline*}
proving \eqref{eq:toproveempty} and
therefore \eqref{eq:toprove} in the case
$I\subseteq \underline \varphi(V)$.

{\bf Case (II):} Assume now that $I\setminus \underline\varphi(V)\neq
\emptyset$.  Note that $ \underline\phi(V) \supseteq I^c$, the
condition invoked in~\eqref{eq:toprove}, holds if and only if $I_V=K$,
i.e.\ if and only if $I_V^{(j)}$ is a tuple of intervals for all
$j\in[t]_0$.

We claim that, in the case under consideration, the following holds
for all $j\in[t]_0$:
\begin{equation}\label{eq:ind}
  F_j(V,I)=\begin{cases} (-1)^{\lvert v_{j+1}\rvert-\lvert
    v_{j}\rvert+1}, & \mbox{if  $I_V^{(j)}$ is a tuple of intervals,}\\
  0, & \mbox{otherwise}.\end{cases}
  \end{equation}
We now prove this claim by induction on the degree of the word
$v_{j+1}/v_j$.

If $v_{j+1}$ covers $v_j$, then $F_j(V,I)=1$ trivially. So assume
that~\eqref{eq:ind} holds for $\lvert v_{j+1}/v_j \rvert\leq \ell$,
for some $1\leq\ell\in\mathbb N$, and suppose that $\lvert
v_{j+1}/v_j\rvert =\ell+1$. Let $\rho_j$ denote the number of
different letters in $v_{j+1}/ v_j$.

Assume first that $I_V^{(j)}$ is a tuple of intervals,
viz.\
$$I_V^{(j)} = \left([\pi_{i}(v_j),\pi_{i}(v_{j+1})] \cap
[n_i-1]\right)_{i=1}^m.$$ Informally, this means that a $j$-admissible
refinement $S$ of $V$ needs to have full projections between $v_j$ and
$v_{j+1}$.  This condition forces the first element of $S\setminus V$
to lie on the $\rho_j$-dimensional hypercube above $v_j$: it is
obtained by multiplying $v_j$ with at most one copy of each of the
$\rho_j$ relevant letters.  We may therefore write $F_j(V,I)$ as a sum
of $2^{\rho_j} -1$ summands, indexed by the words $v^{(1)},\dots,v
^{(2^{\rho_j} -1)}$ covering $v_j$ in $C_{\underline{n}}$:
$$F_j(V,I)=-\sum_{k=1}^{2^{\rho_j} -1} \sum_{\substack{S \geq V \textup{ $j$-adm.}, \\ \min(S\setminus V) = v^{(k)}}} (-1)^{|S\setminus V|},$$
where, for each $k\in[2^{\rho_j} -1]$, the inner sum is taken over
$j$-admissible refinements $S$ of $V$ having $v^{(k)}$ as smallest
element greater than $v_{j}$. Each of these sums is known by induction
from~\eqref{eq:ind}. Indeed, since the flags $S$ also have full
projections between $v^{(k)}$ and $v_{j+1}$, we obtain
\begin{equation*}\label{eq:fifop}
  F_j(V,I)=
  -\sum_{k=1}^{2^{\rho_j} -1} (-1)^{|v_{j+1}|-|v^{(k)}|+1} =
  (-1)^{|v_{j+1}|-|v_j|+1},
\end{equation*}
establishing \eqref{eq:ind} in the first case.

Suppose now that $I_V^{(j)}$ is not a tuple of intervals. Informally,
this means that a $j$-admissible refinement $S$ of $V$ is not required
to have full projections between $v_j $ and~$v_{j+1}$. Without loss of
generality we can assume that the first ``requirement gap'' in
$I_V^{(j)}$ is directly above $v_j$, that is if
$\underline{\alpha}=(\alpha_1,\dots,\alpha_m)$ is the $m$-tuple of
(componentwise) minima of $I_V^{(j)} \setminus \underline\pi(v_j)$,
there is at least one $i\in[m]$ with $\alpha_i>\pi_{i}(v_j)+1$. Given
a $j$-admissible refinement $S$ of $V$, the word $\min(S\setminus V)$,
the smallest word in $S$ greater than $v_j$, clearly belongs to the
interval $(v_j, v_{\underline{\alpha}}]$ of subwords of
$v_{\underline \alpha} := a_1^{\alpha_1}\dots a_m^{\alpha_m}$ which
$v_j$ strictly divides.  Consider the subset
$$Y := \{ v\in (v_j, v_{\underline \alpha} ] \mid
  [v,v_{\underline{\alpha}}] \textup{ is a Boolean algebra}\}.$$ We
  rewrite the sum defining $F_j(V,I)$ according to whether or not
  $\min(S\setminus V)\in Y$:
\begin{equation}\label{eq:sub.F}
  F_j(V,I) = \sum_{\substack{S\geq V \textup{
        $j$-adm.},\\ \min(S\setminus V)\not\in Y}} (-1)^{|S\setminus
    V|} + \sum_{\substack{S\geq V \textup{
        $j$-adm.},\\ \min(S\setminus V)\in Y}} (-1)^{|S\setminus V|}.
\end{equation}
Clearly, the first summand in \eqref{eq:sub.F} is zero. Indeed, we may
further subdivide it by fixing the minimal element
$\min(S\setminus V)$. Each of the resulting summands is zero by
applying \eqref{eq:ind} inductively to the refined flag $V\cup\{v\}$,
replacing $v_j$ by $v$.

The second summand in \eqref{eq:sub.F} is zero, too. Indeed, without
loss of generality we may assume that
$$I_V^{(j)} = \left( \left( \{\pi_{i}(v_{j})\} \cup
[\alpha_i,\pi_{i}(v_{j+1})]\right) \cap [n_i-1]\right)_{i=1}^m.$$
(Otherwise, an argument similar to the one for the first summand in
\eqref{eq:sub.F} proves the claim.) Under this assumption, the
induction hypothesis yields
\begin{multline*}
  \sum_{\substack{S\geq V \textup{ $j$-adm.},\\ \min(S\setminus V)\in
      Y}} (-1)^{|S\setminus V|} = - \sum_{[v,v_{\underline{\alpha}}]
    \textup{ Boolean}}(-1)^{|v_{j+1}| - |v|}
  = (-1)^{|v_{j+1}|-|v_{\underline{\alpha}}|+1} \sum_{Z \subseteq
    \{0,1\}^{\rho_j}}(-1)^{|Z|} = 0.
\end{multline*}

 This
proves~\eqref{eq:ind} in the second case.

Suppose now $I_V= K$. Since $I_V^{(j)}$ is a tuple of intervals for
all $j\in[t]_0$, we get, by \eqref{eq:ind},
\begin{equation*}\label{eq:fop} 
  \sum_{\substack{S\geq V\\  \underline\varphi (S)\supseteq I_V}} (-1)^{|S\setminus V|}=(-1)^t \prod_{j=0}^t F_j(V,I)=(-1)^{2t+1}(-1)^{\sum_{j=0}^t |v_{j+1}|-|v_j|}=(-1)^{N-1}
\end{equation*}
as desired.

Suppose now $I_V \neq K$. This means that there exists $j\in [t]_0$ such
that $I_V^{(j)}$ is not a tuple of intervals. By~\eqref{eq:ind} we
have $F_j(V,I)=0$, thus the product in~\eqref{eq:prod} is also zero,
proving \eqref{eq:toprove} in the last case.
\end{proof}

\begin{proof}[Proof of Theorem~\ref{thm:funeq}]
  The sum defining the generalized Igusa function can be rewritten as
  \begin{equation}\label{eqn:oldsum}
    I^{\textup{wo}}_{\underline{n}}(\bfY;\bfX) =  \sum_{V\in \WO_{\underline{n}}} \binom{\underline{n}}{V}_\bfY W_V(\bfX)= \frac{1}{1-X_{v_{\underline{n}}}}
    \sum_{V\in \WO^{\times}_{\underline{n}}} \binom{\underline{n}}{V}_\bfY W_V(\bfX).
  \end{equation}
  Inverting the variable in the factor
  $\frac{1}{1-X_{v_{\underline{n}}}}$ on the RHS of~\eqref{eqn:oldsum}
  simply gives a factor $-X_{v_{\underline{n}}}$. Thus Theorem
  \ref{thm:funeq} is equivalent to the identity
\begin{equation}\label{eq:equivfuneq}\sum_{V\in \WO^{\times}_{\underline{n}}} \binom{\underline{n}}{V}_{\bfY^{-1}}
  W_V(\bfX ^{-1})= (-1)^{N-1} \left(\prod_{i=1}^m Y_i^{-
    \binom{n_i}{2}}\right)\sum_{V\in \WO^{\times}_{\underline{n}}}
  \binom{\underline{n}}{V}_{\bfY} W_V(\bfX ).
\end{equation}
Writing $\underline S _{\underline{n}}=S_{n_1}\times \dots\times
S_{n_m}$, $\underline w= (w_1,\dots,w_m)$, $\Des(\underline
w)=\Des(w_1)\times\dots\times\Des(w_m)$, and using the identity
\eqref{eq:coxlength}, the LHS of \eqref{eq:equivfuneq} becomes
\begin{align*}
  \sum_{V\in \WO^{\times}_{\underline{n}}}
  \binom{\underline{n}}{V}_{\bfY^{-1}} W_V(\bfX ^{-1}) &=\sum_{V\in
    \WO^{\times}_{\underline{n}}} \left(\sum_{\substack{\underline{w}
      \in \underline{S} _{\underline{n}} \\ \Des
      (\underline{w})\subseteq \underline\varphi (V)}}\prod_{i=1}^m
  Y_i^{- \ell(w_i)}\right)W_V(\bfX ^{-1})\\ &= \sum_{\underline w \in
    \underline S _{\underline{n}}} \left(\prod_{i=1}^m Y_i^{-
    \ell(w_i)}\right) \sum_{\substack{ V\in
      \WO^{\times}_{\underline{n}} \\ \underline\varphi(V)\supseteq
      \Des(\underline w)}} W_V(\bfX ^{-1}).
\end{align*}
For $i\in[m]$ we denote by $w_0^{(i)}$ the longest element in
$S_{n_i}$, of length $\ell(w_0^{(i)}) = \binom{n_i}{2}$. By
Proposition~\ref{pro:mobius} and the identities
\eqref{eq:cox.identities} we can rewrite
\begin{align*}
 \lefteqn{\sum_{\underline w \in \underline S _{\underline{n}}}
   \left(\prod_{i=1}^m Y_i^{- \ell(w_i)}\right)\!\!\! \sum_{\substack{
       V\in \WO^{\times}_{\underline{n}}
       \\ \underline\varphi(V)\supseteq \Des(\underline w)}}\!\!\!
   W_V(\bfX ^{-1}) }\\&=(-1)^{N-1} \sum_{\underline w \in \underline S
     _{\underline{n}}} \left(\prod_{i=1}^m Y_i^{- \ell(w_i)}\right)
   \!\!\!\!\!\!\sum_{\substack{ V\in \WO^{\times}_{\underline{n}}
       \\ \underline\varphi(V)\supseteq \Des(\underline w)^c}}
   W_V(\bfX)\\ &=(-1)^{N-1} \left(\prod_{i=1}^m Y_i^{-
     \binom{n_i}{2}}\right)\!\!\!\sum_{\underline w \in \underline S
     _{\underline{n}}}\left(\prod_{i=1}^m Y_i^{ \ell(w_i w_0
     ^i)}\right) \!\!\!\!\!\!\sum_{\substack{ V\in
       \WO^{\times}_{\underline{n}} \\ \underline\varphi(V)\supseteq
       \Des(\underline w \underline w _0)}} \!\!\!
   W_V(\bfX)\\ &=(-1)^{N-1} \left(\prod_{i=1}^m Y_i^{-
     \binom{n_i}{2}}\right) \sum_{ V\in
     \WO^{\times}_{\underline{n}}}\binom{\underline{n}}{V}_\bfY
   W_V(\bfX),
\end{align*}
proving \eqref{eq:equivfuneq} and thus Theorem~\ref{thm:funeq}.
\end{proof}

\subsection{Weak order zeta functions and generalized Igusa
  functions}\label{subsec:weak.igusa} We record an identity between
instances of weak order zeta functions which will be useful in
Section~\ref{subsec:heisenberg} and may be of independent interest.
The identity compares instances of weak order zeta functions
associated with the all-one-compositions $\underline g$ and
$\underline{2g}$, with $g$ and $2g$ parts, respectively, and holds
when substituting for the variables monomials satisfying certain
relations.

In the current section, we call a subword of the word $\widehat{1} =
v_{\underline {2g}}:=a_1\cdots a_{2g}$ \emph{radical} if it is of the
form $w=\prod_{i \in \mathcal{J}} a_i a_{i+g}$ for some
$\mathcal{J}\subseteq [g]$; see also Definition~\ref{dfn:rad}. We
observe that any subword $r \leq v_{\underline {2g}}$ may be written
uniquely in the form $r = \sqrt r \cdot r^\prime r^{\prime \prime}$,
where $\sqrt r = \prod_{i \in \mathcal I} a_i a_{i+g}$ is a radical
word, whereas $r^\prime = \prod_{i \in \mathcal{I}^\prime} a_i$ and
$r^{\prime \prime} = \prod_{i \in {\mathcal{I}^{\prime \prime}}}
a_{i+g}$, and the subsets $\mathcal{I}, \mathcal{I}^\prime,
\mathcal{I}^{\prime \prime} \subseteq [g]$ are disjoint. Likewise, we
define the \emph{radical} $\sqrt{S}$ of a flag $S\in
\WO_{\underline{2g}}$ to be the flag of radicals of the words of~$S$.

In the following result, we omit the non-occurring variable $Y$ from
the generalized Igusa functions $I^{\mathrm{wo}}_{\underline g}$ and
$I^{\mathrm{wo}}_{\underline{2g}}$; cf.\ our remark in
Example~\ref{exm:gen.igusa}~(2).

\begin{pro} \label{pro:igusas.match} Let $g \in \N$. Suppose that the
  numerical data $\mathbf{y}$ satisfy $y_r = y_{\sqrt r} \cdot
  \prod_{i \in \mathcal{I}^\prime \cup \mathcal{I}^{\prime \prime}}
  y_{a_i}$.  Then
\begin{equation}\label{eq:match}
  I^{\mathrm{wo}}_{\underline {2g}} (\mathbf{y}) =\left( \prod_{i =
    1}^g \frac{1 + y_{a_i}}{1 -
    y_{a_i}}\right)I^{\mathrm{wo}}_{\underline {g}} (\mathbf{z}),
\end{equation} 
where $z_{\prod_{i \in \mathcal{I}} a_i} =
y_{\prod_{i\in\mathcal{I}}{a_i a_{i+g}}}$ for all $\mathcal{I}
\subseteq [g]$.
\end{pro}

\begin{proof}
  By sorting the flags in $\WO_{\underline{2g}}$ by their radicals, we
  may partition the domain of summation of the LHS of \eqref{eq:match}
  as follows:
  $$\WO_{\underline{2g}}=\bigcup_{R\in\WO_{\underline g}} \{S\in
  \WO_{\underline{2g}} \mid \sqrt S =R\}.$$ The claim is equivalent to
  showing that, for all $R \in \WO_{\underline{g}}$,
  \begin{equation}\label{eq:equiv}
    \sum_{\substack{S\in\WO_{\underline{2g}}:\\\sqrt S =  R} } W_S( \mathbf{y}) = \left(\prod_{i = 1}^g \frac{1 + y_{a_i}}{1 - y_{a_i}}\right) W_{ R} (\mathbf{z})=\prod_{i = 1}^g \left(1+ 2\frac{ y_{a_i}}{1 - y_{a_i}}\right) W_{ R} (\mathbf{z}).
  \end{equation}
  
  Let $S=\{s_1<\dots < s_t\}=\{\sqrt{s_1}\cdot s_1^\prime s_1^{\prime
    \prime} <\dots <\sqrt{s_t} \cdot s_t^\prime s_t^{\prime \prime}
  \}\in\WO_{\underline{2g}}$, where, as above, for $k\in[t]$, $
  s_k^\prime=\prod_{i \in \mathcal{I}^\prime} a_i$, $s_k^{\prime
    \prime}= \prod_{i \in {\mathcal{I}^{\prime \prime}}} a_{i+g}$ and
  $\sqrt{s_k}=\prod_ {i \in {\mathcal{I}_k}} a_{i}a_{i+g}$ is
  radical. Denote $J(S)=\{y_{s_1},\dots,y_{s_t}\}$ and, 
  for $j\in [g]$, set $y_{a_j}J(S):=\{y_{a_j}y\mid y\in J(S)\}$. As
  before we set $s_0 = \widehat{0}$ and
  $s_{t+1} = \widehat{1} = v_{\underline {2g}}$.

We claim that, for all $j\in[g]$ and all $S\in\WO_{\underline{2g}}$
with $\sqrt{S}=R$ and the property that, for all $s\in S$ if $a_j|s$
or $a_{g+j}|s$ then $a_ja_{g+j}|s$, the following identity holds:
\begin{equation}\label{eq:singlej}
    \sum_{\substack{\overline S\in\WO_{\underline{2g}}:\,
        \sqrt{\overline S}= R,\\J(\overline S) \subset J(S) \cup
        y_{a_j}J(S)}} W_{\overline S}( \mathbf{y}) = \left(1+ 2\frac{
      y_{a_j}}{1 - y_{a_j}}\right) W_{ S} (\mathbf{y}).
\end{equation}
It is easy to see that \eqref{eq:equiv} follows by repeated
application of \eqref{eq:singlej} for~$j\in [g]$.

We prove \eqref{eq:singlej} by induction on $t$, the induction base
($t=0$) being trivial; we observe that our assumption on the numerical
data implies that $y_{a_j}=y_{a_{g+j}}$. The RHS may therefore be
written as
$$\left(\prod_{l=1}^i \gp{y_{s_l}}\right) \left( 1 + \gp{y_{a_j}} +
\gp{y_{a_{g+j}}}\right)\left( \prod_{l=i+1}^t \gp{y_{s_l}}\right).$$
The summand 1 in the central factor arises from the flag $\ol{S}=S$,
with $W_S(\bfy) = \prod_{i=1}^t \gp{y_{s_l}}$. The other two summands
account for flags $\ol{S}$ with $J(\ol{S}) = y_{a_j}J(S)$, i.e.\ for
flags whose words differ from those of $S$ by at most an extra factor
$a_j$ or $a_{g+j}$ (but not both, as they share with $S$ the radical
$R$), and which do feature at least one such a ``augmented'' word. We
will call such flags \emph{$a_j$-augmentations} (\emph{of~$S$}). It
remains to show that
\begin{equation}\label{prod}
  \sum_{\substack{\ol{S}\in \WO_{\underline{2g}}: \\\textup{ $a_j$-augmentation of $S$}}} W_{\ol{S}}(\bfy)
  = \left( \prod_{l=1}^t \gp{y_{s_l}}\right) \gp{y_{a_j}};
\end{equation}
the argument for $a_{g+j}$ is identical.

We note that there exists a unique $i\in[t]$ such that $a_j| s_{i+1}$
but $a_j \nmid s_{i}$. For all $a_j$-augmentations~$\ol{S}$ of~$S$,
the last $t-i$ words coincide with $s_{i+1},\dots,s_t$. Therefore
$\prod_{l=i+1}^t \gp{y_{s_l}}$ divides all relevant
$W_{\ol{S}}(\bfy)$. Without loss of generality we may thus assume that
$i=t$, i.e. that \emph{no word of $S$ is divisible by $a_j$}.

The claimed identity in \eqref{prod} will become clear by interpreting
the trivial identity
\begin{multline}\label{eq:braidflag.new}
  \left( \prod_{l=1}^{t} \gp{y_{s_l}}\right)\frac{y_{a_j}}{1-y_{a_j}}
  = \\ \left( \prod_{l=1}^{t-1} \gp{y_{s_l}}\right)
  \left(\frac{y_{a_j}y_{s_t}}{1-y_{a_j}y_{s_t}} +
  \frac{y_{s_t}}{1-y_{s_t}}\frac{y_{a_j}y_{s_t}}{1-y_{a_j}y_{s_t}} +
  \frac{y_{a_j}}{1-y_{a_j}}\frac{y_{a_j}y_{s_t}}{1-y_{a_j}y_{s_t}}
  \right). \end{multline} Informally, the RHS
of~\eqref{eq:braidflag.new} reflects the three alternatives for the
first occurrence of $a_j$ in an $a_j$-augmentation of $\overline S$. %
\begin{enumerate}
\item The first summand arises from the $a_j$-augmentation $\ol{S} =
  \left\{ \dots < s_{t-2} < s_{t-1} < a_js_t\right\}$.
\item The second summand arises from the $a_j$-augmentation
  $\ol{S} = \left\{ \dots < s_{t-1} < s_{t} < a_j s_t\right\}$.
\item The third summand arises from all $a_j$-augmentations of $S$
  whose last \emph{two} words are divisible by $a_j$, the last one
  being $a_js_t$, viz.\ $a_j$-augmentations of
  $S\setminus\{s_t\}$. All the relevant $W_{\ol{S}}(\bfy)$ are
  therefore divisible by $\gp{y_{a_j}y_{s_t}}$. By induction
  hypothesis, \eqref{prod} yields
  $$\left( \prod_{l=1}^{t-1} \gp{y_{s_l}}\right)\gp{y_{a_j}} =
  \sum_{\substack{\ol{S}\in \WO_{\underline{2g}}: \\\textup{
        $a_j$-augmentation of $S \setminus \{s_t\}$}}}
  W_{\ol{S}}(\bfy).
  $$
\end{enumerate}
  This proves the claim, and hence the proposition.
  \end{proof}

\section{Counting $\lri$-ideals in combinatorially defined $\lri$-Lie algebras} \label{sec:main.results}
In this section we compute the $\lri$-ideal zeta functions of
$\lri$-Lie algebras satisfying a certain combinatorial condition
(Hypothesis~\ref{hypothesis}) in terms of the generalized Igusa
functions introduced in Section~\ref{sec:general.igusa}. This prepares
the proof of Theorem~\ref{thm:main.local}, given in
Section~\ref{sec:application}.

\subsection{Informal overview}\label{subsec:overview}
We start by summarizing the principal ideas behind our approach, which
greatly generalize those of~\cite{SV1/15}.  Let $L$ be an $\lri$-Lie
algebra with derived subalgebra $L^\prime = [L,L]$.  As noted in
Section~\ref{subsubsec:methodology}, if $L$ is class-$2$-nilpotent,
then an $\lri$-sublattice $\Lambda \leq L$ is an $\lri$-ideal if
$[\overline{\Lambda}, L] \leq \Lambda \cap L^\prime$, where
$\overline{\Lambda} = (\Lambda + L^\prime)/L^\prime$.  For simplicity
of exposition we will assume, in this overview, that $L^\prime =
Z(L)$, i.e.~that $L$ has no abelian direct summands.  By an argument
going back to~\cite[Lemma~6.1]{GSS/88}, the computation of $\zidealo_L
(s)$ is reduced to a summation over pairs $(\overline{\Lambda},
\mathrm{M})$, where $\overline{\Lambda} \leq L / L^\prime$ and
$\mathrm{M} \leq L^\prime$ are $\lri$-sublattices such that
$[\overline{\Lambda}, L] \leq \mathrm{M}$.  Recall that the
$\Lri$-elementary divisor type of a finite-index $\Lri$-sublattice
$\Lambda \leq \Lri^n$, where $\Lri$ is a compact discrete valuation
ring with maximal ideal $\mathfrak{M}$, is the partition $(\lambda_1,
\dots, \lambda_n)$ such that
$$ \Lri^n / \Lambda \simeq \Lri / \mathfrak{M}^{\lambda_1} \times
\cdots \times \Lri / \mathfrak{M}^{\lambda_n}.$$ Given the
$\lri$-elementary divisor type $\lambda(\overline{\Lambda})$ of
$[\overline{\Lambda},L]$, the lattices $\mathrm{M}$ satisfying this
condition are enumerated by Birkhoff's
formula~\eqref{def:birkhoff.poly}.

An essential ingredient of our method, therefore, is an effective
description of the $\lri$-elemen\-ta\-ry divisor type
$\lambda(\overline{\Lambda})$ in terms of the structure of
$\overline{\Lambda}$.  For the $\lri$-Lie algebras considered in this
paper, this is accomplished as follows.  The quotient $L/ L^\prime$
decomposes, as an $\lri$-module, into a direct sum of $m$ components,
which are viewed as free modules over finite extensions $\Lri_1,
\dots, \Lri_m$ of $\lri$.  For each component, we consider the
$\Lri_i$-elementary divisor type $\nu^{(i)}$ of the $\Lri_i$-lattice
generated by the projection of $\overline{\Lambda}$ onto that
component.  These are the projection data of
Definition~\ref{dfn:projection.data} below.  The crucial
Hypothesis~\ref{hypothesis} requires that the parts of the partition
$\lambda(\overline{\Lambda})$ be given by the minima of term-by-term
comparisons among the elementary divisor types appearing in the
projection data.  Assuming Hypothesis~\ref{hypothesis}, we deduce a
purely combinatorial expression for $\zidealo_L (s)$ in
Proposition~\ref{pro:zeta.rewrite}.

Analogously to the argument of~\cite{SV1/15}, we break up the sum in
Proposition~\ref{pro:zeta.rewrite} into finitely many pieces on which
the Gaussian multinomial coefficients---arising via the factors
$\beta(\nu^{(i)}; q_i)$ and
$\alpha(\lambda(\boldsymbol{\nu}), \mu; q)$, in the notation used
there---and the dual partitions occurring in the
definition~\eqref{def:birkhoff.poly} of
$\alpha(\lambda(\boldsymbol{\nu}), \mu; q)$ are constant.  The sum
over each piece yields a product of Gaussian multinomials and
geometric progressions; these, in turn, are assembled into generalized
Igusa functions introduced in Section~\ref{sec:general.igusa}. As in
\cite{SV1/15}, Dyck words of fixed length turn out to be suitable
indexing objects for the finitely many pieces.

The technical complexity of the current paper, in comparison
to~\cite{SV1/15}, reflects the fact the translation between projection
data and the elementary divisor type $\lambda(\overline{\Lambda})$ is
considerably more involved.  While the data determining
$\lambda(\overline{\Lambda})$ in~\cite{SV1/15} were just a collection
of integers, here they are a collection of partitions (the $\nu^{(i)}$
defined above).  A more sophisticated combinatorial machinery,
viz.\ the weak orders of Section~\ref{subsec:gen.igusa.funeq}, is
required to keep track of the relative sizes of the parts of these
different partitions; this is necessary in order to specify domains of
summation over which the dual partition
$\lambda(\overline{\Lambda})^\prime$ is constant.

In Section~\ref{subsec:projection.data} we define the concept of
projection data and enumerate lattices $\overline{\Lambda} \leq
L/L^\prime$ with fixed projection data.  In
Section~\ref{subsec:rewriting} we introduce and explain the
combinatorial structure behind Hypothesis~\ref{hypothesis} and deduce
Proposition~\ref{pro:zeta.rewrite}, giving a general formula for
$\lri$-ideal zeta functions of $\lri$-Lie algebras satisfying
Hypothesis~\ref{hypothesis}.  In Section~\ref{sec:statement} we state
the section's main result, viz.\ Theorem~\ref{thm:zeta.explicit}, and
prove it modulo an auxiliary claim,
viz.\ Proposition~\ref{thm:beef.is.here}, whose rather technical proof
is given in Section~\ref{subsec:proof.of.beef}.

Throughout, let $\mfo$ be a complete discrete valuation ring with
finite residue field of cardinality~$q$, and let $\mfO_1, \dots,
\mfO_{\pnum}$ be finite extensions of~$\mfo$.  Let $\pi \in \lri$ be a
uniformizer.  For each $i \in [\pnum]$, let $e_i$ be the ramification
index and $f_i$ be the inertia degree of $\Lri_i$ over $\lri$.  Let
$q_i = q^{f_i}$ be the cardinality of the residue field of
$\mfO_i$. For each $i\in[\pnum]$, the local ring $\mfO_i$ is a free
$\mfo$-module of rank $e_i f_i$.  Let $(n_1, \dots, n_\pnum) \in
\N_0^{\pnum}$ and set $n = \sum_{i = 1}^{\pnum} e_i f_i n_i$. Consider
a family $\widetilde{\bsnu}=(\nu^{(1)},\dots,\nu^{(\pnum)})$ of
partitions $\nu^{(i)}$, each with $n_i$ parts.

\subsection{Counting lattices with fixed
  projections} \label{subsec:projection.data} Consider the $\lri$-module
   $$\Omega = \mfO_1^{n_1} \times \cdots \times \mfO_{\pnum}^{n_{\pnum}}$$
and, for each $i \in [\pnum]$, let $\pi_i : \Omega \to \mfO_i^{n_i}$ be
the projection onto the $i$-th component.  Choosing an $\mfO_i$-basis
$(e_1^{(i)}, \dots, e_{n_i}^{(i)})$ of $\mfO_i^{n_i}$ and an
$\mfo$-basis $(\alpha_1^{(i)}, \dots, \alpha_{e_if_i}^{(i)})$ of
each~$\mfO_i$, it is clear that the collection $\left\{ \alpha_j^{(i)}
e_k^{(i)} \right\}_{ijk}$ constitutes an $\mfo$-basis of $\Omega$ that
allows us to identify $\Omega$ with $\mfo^n$.

\begin{dfn} \label{dfn:projection.data}
For an $\mfo$-sublattice $\Lambda \leq \mfo^n$, we write $\nu^{(i)} =
\nu(\pi_i(\Lambda))$ for the elementary divisor type of the
$\mfO_i$-sublattice of $\Lri_i^{n_i}$ generated by
$\pi_i(\Lambda)$. Note that $\nu^{(i)}$ is a partition with $n_i$
parts. The family
$$\bsnu(\Lambda) = (\nu^{(1)},\dots,\nu^{(\pnum)})$$ of partitions is
called the \emph{projection data} of $\Lambda$ with respect to
$\Omega$.
\end{dfn}

For any partition $\nu = (\nu_1,
\dots, \nu_N)$ with $N$ parts, set $J_{\nu} = \{ d \in [N-1] \mid
\nu_d > \nu_{d+1} \}$.  For a variable $Y$, we define
\begin{equation} \label{equ:beta.definition}
\beta(\nu;Y) = \binom{N}{J_{{\nu}}}_{Y^{-1}} Y^{\sum_{d = 1}^{N-1}
  d(N-d)(\nu_d - \nu_{d+1})} \in \Q[Y].
\end{equation}
We observe that $\beta(\nu;Y) = \alpha(\lambda, \nu; Y)$, the
``Birkhoff polynomial'' \eqref{def:birkhoff.poly}, where $\lambda$ is
any partition whose parts are all at least $\nu_1$. It follows that
$\beta(\nu;q)$ enumerates the $\mfo$-sublattices of $\mfo^N$ of
elementary divisor type ${\nu}$. The following proposition, which is
key to our method, generalizes this formula and is analogous
to~\cite[Lemma~2.4]{SV1/15}. Recall the formula \eqref{equ:abelian}
for the zeta function $\zeta_{\mfo^n}(s)$ of an abelian (Lie) algebra
of finite rank over a compact discrete valuation ring.

\begin{pro} \label{pro:sum.projection.data}
Let $\widetilde{\bsnu}=(\nu^{(1)},\dots,\nu^{(\pnum)})$ be as above.
Then
\begin{equation*}
  \sum_{\Lambda \leq \mfo^n \atop \bsnu(\Lambda) = \widetilde{\bsnu}}
  | \mfo^n : \Lambda |^{-s} = \frac{\zeta_{\mfo^n}(s)}{\prod_{i =
      1}^{\pnum} \zeta_{\mfO_i^{n_i}}(s)} \left( \prod_{i = 1}^{\pnum}
  \beta(\nu^{(i)}; q_i) \right) t^{\sum_{i=1}^{\pnum} \left(
    \sum_{j=1}^{n_i} \nu^{(i)}_{j}\right) f_i}.
\end{equation*}

\end{pro}
\begin{proof}
  Recall that for every $i \in [\pnum]$ there is a natural embedding of
  rings $\iota_i: \mfO_i \hookrightarrow \Mat_{e_i f_i}(\mfo)$ that
  sends an element $y \in \mfO_i$ to the matrix representing the
  $\mfo$-linear operator $x \mapsto xy$ on $\mfO_i$ with respect to
  the chosen $\lri$-basis $\{ \alpha_j^{(i)} \}_{j=1}^{e_if_i}$.
  Moreover, $\det \iota_i(y) = \mathrm{N}_{\mfO_i / \mfo}(y)$ for all
  $y \in \mfO_i$.  This map extends naturally to an embedding of
  matrix rings $\Mat_{n_i} (\mfO_i) \hookrightarrow \Mat_{e_if_i
    n_i}(\mfo)$ that we continue to denote by $\iota_i$.

Consider the set $\mathcal{H} = \left\{ (H_1, \dots, H_{\pnum}) \mid \forall
i\in[\pnum]: H_i \leq \mfO_i^{n_i} \right\}$.  Given $H \in \mathcal{H}$,
denote
\begin{equation*}
\Sigma_H = \sum_{\Lambda \leq \mfo^n \atop \pi_i(\Lambda) = H_i}
|\mfo^n : \Lambda |^{-s}.
\end{equation*}
Thus
\begin{equation}\label{sum.eq.sum}
  \sum_{\Lambda \leq \mfo^n \atop \bsnu(\Lambda)=\widetilde{\bsnu}} |
  \mfo^n : \Lambda |^{-s} = \sum_{H\in \mathcal{H} \atop \nu(H_i) =
    \nu^{(i)}} \Sigma_H.
\end{equation}

For every $i \in [\pnum]$, let $B_i \in \Mat_{n_i}(\mfO_i)$ be a matrix
whose rows comprise an $\mfO_i$-basis of $H_i$.  Let $B \in
\Mat_n(\mfo)$ be the block-diagonal matrix with blocks $\iota_i(B_i)$.
We observe that the map $\Mat_n(\mfo) \rightarrow \Mat_n(\mfo), \,
B^\prime \mapsto B^\prime B$ induces a bijection between the set of
$\mfo$-lattices $\Lambda \leq \mfo^n$ such that $\pi_i(\Lambda) =
\mfO_i^{n_i}$ for all $i \in [\pnum]$ and the set of lattices $\Lambda
\leq \mfo^n$ such that $\pi_i(\Lambda) = H_i$ for all $i \in [\pnum]$.
Furthermore, $\det B = \prod_{i = 1}^{\pnum} \mathrm{N}_{\mfO_i / \mfo}(\det
B_i)$; cf., for instance, \cite[Theorem 1]{KSW/99}.  The norms
preserve normalized valuation, hence $| \det B|_\mfo = \prod_{i =
  1}^{\pnum} q_i^{-\sum_{j = 1}^{n_i} \nu^{(i)}_{j}}$.  We conclude that
\begin{equation}\label{equ:aux1}
\Sigma_H = t^{\sum_{i,j} \nu^{(i)}_{j} f_i} \Sigma_\bfz = \prod_{i =
  1}^{\pnum} |\mfO_i^{n_i} : H_i|^{-s} \Sigma_\bfz,
\end{equation}
where $\bfz = (\mfO_1^{n_1}, \dots, \mfO_{\pnum}^{n_{\pnum}}) \in \mathcal{H}$.
Thus
\begin{equation}\label{equ:aux2}
\zeta_{\mfo^n}(s) = \sum_{H \in \mathcal{H}} \Sigma_H = \Sigma_\bfz
\sum_{H \in \mathcal{H}} \prod_{i = 1}^{\pnum} |\mfO_i^{n_i} :
H_i|^{-s} = \Sigma_\bfz \prod_{i = 1}^{\pnum} \zeta_{\mfO_i^{n_i}}(s).
\end{equation}
It follows immediately from \eqref{equ:aux1} and \eqref{equ:aux2} that
\begin{equation*}
\Sigma_H = \frac{\zeta_{\mfo^n}(s)}{\prod_{i = 1}^{\pnum} \zeta_{\mfO_i^{n_i}}(s)}t^{\sum_{i,j} \nu^{(i)}_{j} f_i}, 
\end{equation*}
and substitution of this expression into~\eqref{sum.eq.sum} implies our claim.
\end{proof}

\subsection{Rewriting the $\lri$-ideal zeta functions of suitable $\lri$-Lie
  algebras} \label{subsec:rewriting} Now let $L$ be a
class-$2$-nilpotent $\lri$-Lie algebra.  We assume that its derived
subalgebra $L^\prime$ is isolated, viz.\ $L/L'$ is torsion-free. Let
further $L^\prime \subseteq A \subseteq Z(L)$ be a central, isolated
subalgebra.  Suppose that
\begin{equation} \label{equ:abelianization.condition}
L/A \simeq \mfO_1^{n_1} \times \cdots \times \mfO_{\pnum}^{n_{\pnum}}.
\end{equation}
Fixing such an isomorphism, we obtain projections
$\pi_i : L/A \to \mfO_i^{n_i}$ and are in the setting of
Section~\ref{subsec:projection.data}.  Then $c^\prime$ and $c$, in the
notation of Section~\ref{subsec:projection.data}, are the ranks of the
free $\lri$-modules $L^\prime$ and $A$, respectively, whereas
$n = \sum_{i = 1}^{\pnum} n_i e_i f_i = \rk_\lri L/A$. In particular,
$n+c = \rk_\lri L$.

Given an $\mfo$-sublattice $\Lambda \leq L/A$ of finite index, the
commutator $[\Lambda, L]$ is well-defined, as $A$ is central, and of
finite index in~$L^\prime$.  Let $\lambda(\Lambda)$ be the
$\lri$-elementary divisor type of the $\lri$-submodule $[\Lambda, L] \leq L'$.

\begin{dfn}\label{def:partition.ast}
Let ${\nu}^{(1)} = (\nu_{1}^{(1)}, \dots, \nu_{n_1}^{(1)})$ and
${\nu}^{(2)} = (\nu^{(2)}_1, \dots, \nu_{n_2}^{(2)})$ be partitions
with $n_1$ and $n_2$ parts, respectively.  We define ${\nu}^{(1)} \ast
{\nu}^{(2)}$ to be the partition whose $n_1 n_2$ parts are obtained
from the multiset
$$\left\{ \min \{ \nu^{(1)}_{k}, \nu^{(2)}_{\ell} \} \right\}_{k \in [n_1],\;
  \ell \in [n_2]}.$$
  
Given, in addition, $b \in [n_1]$, we define $(\nu^{(1)})^{\ast b}$ to
be the partition whose $\binom{n_1}{b}$ parts are obtained from the
multiset
$$ \left\{ \min \{ \nu^{(1)}_i \mid i \in I \} \right\}_{I \subseteq [n_1], \;|I| = b}.$$
\end{dfn}

We observe that $\ast$ is an associative binary operation on the set
of partitions and that $(\nu^{(1)})^{\ast 2} \neq \nu^{(1)} \ast
\nu^{(1)}$.

\begin{dfn} \label{dfn:lambda.nu}
 Let $\klim \in \N_0$ and fix, for every $k \in [\klim]$, a pair
 $\widetilde{\mathfrak{S}}_k = (\mathfrak{S}_k,
 \underline{\sigma}_k)$, where $\mathfrak{S}_k = \{ s_{k1}, \dots,
 s_{k,\tau_k} \} \subseteq [\pnum]$ is a subset of cardinality
 $\tau_k$ and $\underline{\sigma}_k = (\sigma_{k1}, \dots,
 \sigma_{k,\tau_k}) \in \N^{\tau_k}$.
      
 Given a family $\widetilde{\boldsymbol{\nu}} = (\nu^{(1)}, \dots,
 \nu^{(\pnum)})$ of partitions $\nu^{(i)}$, each with $n_i$ parts,
 define $\lambda(\widetilde{\boldsymbol{\nu}})$ to be the partition
 obtained from the multiset
$$ \bigcup_{k = 1}^{\klim} \left\{ (\nu^{(s_{k1})})^{\ast \sigma_{k1}}
 \ast \cdots \ast (\nu^{(s_{k,\tau_k})})^{\ast \sigma_{k,\tau_k}}
 \right\},$$ where $\{ \nu^{(i)} \}$ denotes the multiset of parts of
 the partition $\nu^{(i)}$ and the union is a union of multisets.
\end{dfn}

We will suppose for the rest of Section~\ref{sec:main.results} that the following assumption on
$(L,A)$ holds.
\begin{hyp}  \label{hypothesis}
 The pairs $\widetilde{\mathfrak{S}}_1, \dots,
 \widetilde{\mathfrak{S}}_{\klim}$ in Definition~\ref{dfn:lambda.nu}
 may be chosen so that for any $\mfo$-sublattice $\Lambda \leq L/A$,
 the equality of partitions $\lambda(\Lambda) =
 \lambda(\boldsymbol{\nu}(\Lambda))$ holds.
 \end{hyp}

Comparing the lengths of the partitions $\lambda(\Lambda)$ and
$\lambda(\boldsymbol{\nu}(\Lambda))$, we find that
Hypothesis~\ref{hypothesis} implies that
$$ c^\prime = \sum_{k = 1}^{\klim} \binom{n_{s_{k1}}}{\sigma_{k1}}
\binom{n_{s_{k2}}}{\sigma_{k2}} \cdots
\binom{n_{s_{k,\tau_k}}}{\sigma_{k,\tau_k}}.$$

\begin{dfn} \label{dfn:m}
Let $\mathfrak{S} = \bigcup_{k = 1}^{\klim} \mathfrak{S}_k \subseteq
[\pnum]$.  Let $m = | \mathfrak{S} |$.  Renumbering the components
in~\eqref{equ:abelianization.condition} if necessary, we may suppose
without loss of generality that $\mathfrak{S} = [m]$.
\end{dfn}

We briefly discuss the motivation for Hypothesis~\ref{hypothesis}.  It
ensures that the elementary divisor type $\lambda(\Lambda)$ depends
only on the projection data $\boldsymbol{\nu}(\Lambda)$ and can be
described combinatorially in terms of $\boldsymbol{\nu}(\Lambda)$, and
that all parts of $\lambda(\Lambda)$ also appear as parts of
$\boldsymbol{\nu}(\Lambda)$.  This assumption is crucial to our method
and enables us to express the $\mfo$-ideal zeta function
$\zidealo_L(s)$ in terms of the generalized Igusa functions of
Definition~\ref{def:igusa.wo}.  A further consequence of
Hypothesis~\ref{hypothesis} is a dichotomy among the components of
$L/A$ in~\eqref{equ:abelianization.condition}.  If, on the one hand,
$i > m$, then the commutator $[\Lambda,L]$ is independent of the
component $\mfO_i^{n_i}$; this means that $\mfO_i^{n_i}$ lies in the
kernel of the projection $\mathrm{pr}: L/A \to L/Z(L)$.  If, on the
other hand, $i\leq m$, then $\mathrm{pr}(\mfO_i^{n_i})$ and
$\mfO_i^{n_i}$ have the same rank as $\mfo$-modules, namely
$n_ie_if_i$.  In particular,
\begin{equation} \label{equ:cocenter}
\sum_{i = 1}^m n_i e_i f_i = \mathrm{rk}_{\mfo} (L/Z(L)).
\end{equation}
This consequence of Hypothesis~\ref{hypothesis} is used in a subtle
but crucial way in the proof of Corollary~\ref{cor:funct.eq}, which
establishes the functional equation satisfied by
$\zeta^{\triangleleft\,\lri}_L(s)$.  Indeed,
Theorem~\ref{thm:zeta.explicit} expresses
$\zeta^{\triangleleft\,\lri}_L(s)$ as a sum of finitely many summands.
The above observation ensures that each summand satisfies a functional
equation with the same symmetry factor.

\begin{rem} \label{rem:direct.product}
We note that, trivially, Hypothesis~\ref{hypothesis} is stable under
direct products.
\end{rem}

\begin{rem}\label{rem:ram}
  Before proceeding, we observe that Hypothesis~\ref{hypothesis}
  constrains the extensions $\mfO_i$ of $\mfo$ to be unramified in
  natural examples, such as the non-abelian examples considered in
  Section~\ref{sec:application}.  Suppose that
  $L = \mcL_1(\mfO_1) \times \cdots \times \mcL_r(\mfO_r)$, where
  $\mcL_i$ is a class-$2$-nilpotent Lie ring and $\mfO_i$ is a finite
  extension of $\mfo$ for every $i \in [r]$.  Suppose that the
  subalgebra $L^\prime \leq A \leq Z(L)$ is of the form
  $A = A_1 \times \cdots \times A_r$, where each $A_i$ is an isolated
  subalgebra of $\mcL_i(\mfO_i)$; this will be true, for instance, if
  $A = L^\prime$ or $A = Z(L)$.  Then
  $L/A \simeq \mcL_1(\mfO_1) / A_1 \times \cdots \times \mcL_r(\mfO_r)
  / A_r$.  Suppose, furthermore, that we decompose
\begin{eqnarray*}
\mcL_1(\mfO_1) / A_1 & \simeq &\mfO_1^{n_1} \times \cdots \times \mfO_1^{n_{N_1}} \\
\mcL_2(\mfO_2) / A_2 & \simeq &\mfO_2^{n_{N_1 + 1}} \times \cdots \times \mfO_2^{n_{N_2}} \\
\vdots & & \vdots \\
\mcL_r(\mfO_r) / A_r & \simeq &\mfO_r^{n_{N_{r-1} + 1}} \times \cdots \times \mfO_r^{n_{N_r}}
\end{eqnarray*}
and consider the projection data with respect to the resulting
decomposition 
$$ L/A \simeq \mfO_1^{n_1} \times \cdots \times \mfO_r^{n_{N_r}}.$$
Here the number of projections is $\pnum = {N_r}$.  Assume that
Hypothesis~\ref{hypothesis} is satisfied.  We claim that $\mfO_i /
\mfo$ is unramified for all $i \in [r]$ such that $\mcL_i$ is not
abelian.

Indeed, fix uniformizers $\Pi_i \in \mfO_i$, let $\tau \in \N$, and
consider the lattice
$$ \Lambda = \Pi_1^\tau \mfO_1^{n_1} \times \cdots \times \Pi_1^\tau
\mfO_1^{n_{N_1}} \times \Pi_2^\tau \mfO_2^{n_{N_1 + 1}} \times \cdots
\times \Pi_r^\tau \mfO_r^{n_{N_r}}.$$ The projection data are
$\nu^{(i)}_j = \tau$ for all $i \in [N_r]$ and all $j \in [n_i]$.
Furthermore, it is clear that
$$ [\Lambda, L] = \Pi_1^\tau [\mcL_1(\mfO_1), \mcL_1(\mfO_1)] \times
\cdots \times \Pi_r^\tau [\mcL_r(\mfO_r), \mcL_r(\mfO_r)].$$ For every
$i \in [r]$, let $b_i$ be the rank of
$[\mcL_i(\mfO_i), \mcL_i(\mfO_i)]$ as an $\mfo$-module.  Then it is
immediate from Lemma~\ref{lem:oedt.to.zp} that the partition
$\lambda(\Lambda)$ is the disjoint union of the sets
$\{ \tau \}_{e_i, f_i}$ (cf.~Definition~\ref{dfn:dictionary}), with
respective multiplicities $b_i$.  Suppose that $\mcL_i$ is not
abelian.  Then $b_i > 0$.  If, in addition, $e_i \geq 2$, then the
elements of $\{ \tau \}_{e_i, f_i}$ are not all equal to $\tau$.
Hence there are parts of $\lambda(\Lambda)$ that do not appear as
parts of the projection data $\widetilde{\boldsymbol{\nu}}$,
contradicting Hypothesis~\ref{hypothesis}.
\end{rem}

\begin{dfn} \label{dfn:dominance.of.partitions}
Set $\varepsilon = c - c^\prime$.  Given partitions $\lambda$ and
$\mu$ with $c^\prime$ and $c$ parts, respectively, we say that $\mu
\leq \lambda$ if $\mu \leq \widetilde{\lambda}$, where
$\widetilde{\lambda}$ is any partition with $c$ parts whose parts
consist of the $c^\prime$ parts of $\lambda$ together with any
$\varepsilon$ integers $\xi_1 \geq \cdots \geq \xi_{\varepsilon} \geq
\mu_1$.  By $\alpha(\lambda, \mu; Y)$ we will mean
$\alpha(\widetilde{\lambda}, \mu; Y)$, the ``Birkhoff polynomial''
\eqref{def:birkhoff.poly}; note that both definitions are independent
of the choice of $\widetilde{\lambda}$.
\end{dfn}

Our objective, which will be attained with Theorem~\ref{thm:zeta.explicit},
is to compute the $\lri$-ideal zeta function of the $\lri$-Lie
algebra~$L$. We maintain the notation from above. Recall, in
particular, that $n = \sum_{i = 1}^{\pnum} e_i f_i n_i$ is the
$\mfo$-rank of $L/A$.  Observe that if $\Lambda \leq L/A$ as above,
then there exists an $\mfo$-sublattice $\mathrm{M} \leq A$ of elementary
divisor type $\mu$ such that $[\Lambda, L] \leq \mathrm{M}$ if and only if $\mu
\leq \lambda(\Lambda)$.  Furthermore, as $L'$ is isolated in $L$, the
number of sublattices $\mathrm{M} \leq A$ of elementary divisor type $\mu$ that
contain $[\Lambda, L]$ is given by $\alpha(\lambda(\Lambda), \mu; q)$.

Recall $m$ from Definition~\ref{dfn:m}.  Given projection data
$\widetilde{\boldsymbol{\nu}} = (\nu^{(1)}, \dots, \nu^{(\pnum)})$,
the partition $\lambda(\widetilde{\boldsymbol{\nu}})$ depends only on
the $m$-tuple $\boldsymbol{\nu} = (\nu^{(1)}, \dots, \nu^{(m)})$.
Thus we will write $\lambda(\boldsymbol{\nu})$ for
$\lambda(\widetilde{\boldsymbol{\nu}})$.

\begin{pro} \label{pro:zeta.rewrite} 
Assuming
  Hypothesis~\ref{hypothesis}, the $\lri$-ideal zeta function of $L$
  is given by
\begin{equation*}
  \zeta^{\triangleleft\,\lri}_L(s) =
  \frac{\zeta_{\mfo^{n}}(s)}{\prod_{i = 1}^{m} \zeta_{\mfO_i^{n_i}}(s)}
  \sum_{\boldsymbol{\nu}, {\mu} \atop {\mu} \leq
    {{\lambda}}(\boldsymbol{\nu})} \left( \prod_{i = 1}^{m}
  \beta(\nu^{(i)} ; q_i) \right) \alpha({{\lambda}}(\boldsymbol{\nu}),
       {\mu};q) (q^{n}t)^{\sum_{k = 1}^{c} \mu_k} t^{\sum_{i = 1}^{m}
         (\sum_{j = 1}^{n_i} \nu^{(i)}_{j})f_i}.
\end{equation*}
Here $\boldsymbol{\nu} = ({\nu}^{(1)}, \dots, {\nu}^{(m)})$ runs over
all $m$-tuples of partitions with $n_1, \dots, n_{m} $ parts,
respectively, whereas $\mu$ runs over all partitions with $c$
parts. The condition $\mu\leq\lambda(\boldsymbol{\nu})$ is to be
understood as in Definition~\ref{dfn:dominance.of.partitions}.
\end{pro}

\begin{proof}
  The quotient $L/A$ has $\lri$-rank $n$, so it follows
  from~\cite[Lemma~6.1]{GSS/88} that
\begin{equation*}
  \zidealo_L(s) = \sum_{\Lambda \leq L/A} | L/A : \Lambda |^{-s}
  \sum_{[\Lambda, L] \leq \mathrm{M} \leq A} | A :\mathrm{M}|^{n-s}.
\end{equation*}
Grouping the lattices $\Lambda \leq L/A$ by their projection data
$\boldsymbol{\nu}(\Lambda)$, we obtain
\begin{equation*}
\zidealo_L(s) = \sum_{\widetilde{\boldsymbol{\nu}}} \sum_{\Lambda \leq L/A \atop
  \bsnu(\Lambda)=\widetilde{\bsnu}} | L/A : \Lambda |^{-s} \sum_{[\Lambda, L] \leq
  \mathrm{M} \leq A} | A :\mathrm{M}|^{n-s}.
\end{equation*}
Setting ${\mu}$ to be the elementary divisor type of $\mathrm{M}$ and
recalling that $\lambda(\boldsymbol{\nu}(\Lambda))$ is the elementary
divisor type of $[\Lambda, L]$ by Hypothesis~\ref{hypothesis}, it now
follows from Proposition~\ref{pro:sum.projection.data} that
\begin{equation*}
  \zeta^{\triangleleft\,\lri}_L(s) =
  \frac{\zeta_{\mfo^{n}}(s)}{\prod_{i = 1}^{\pnum} \zeta_{\mfO_i^{n_i}}(s)}
  \sum_{\widetilde{\boldsymbol{\nu}}, {\mu} \atop {\mu} \leq
    {{\lambda}}(\widetilde{\boldsymbol{\nu}})} \left( \prod_{i = 1}^{\pnum}
  \beta(\nu^{(i)} ; q_i) \right) \alpha({{\lambda}}(\widetilde{\boldsymbol{\nu}}),
       {\mu};q) (q^{n}t)^{\sum_{k = 1}^{c} \mu_k} t^{\sum_{i = 1}^{\pnum}
         (\sum_{j = 1}^{n_i} \nu^{(i)}_{j})f_i}.
\end{equation*}

As we observed above,
$\alpha({{\lambda}}(\widetilde{\boldsymbol{\nu}}),{\mu};q)$ depends
only on the first $m$ components of the
$\pnum$-tuple~$\widetilde{\boldsymbol{\nu}}$.  Hence the sum in the
previous displayed formula may be expressed as
\begin{multline} \label{equ:intermediate}
 \sum_{\boldsymbol{\nu}, {\mu} \atop {\mu} \leq
    {{\lambda}}(\boldsymbol{\nu})} \left( \prod_{i = 1}^{m}
  \beta(\nu^{(i)} ; q_i) \right) \alpha({{\lambda}}(\boldsymbol{\nu}),
       {\mu};q) (q^{n}t)^{\sum_{k = 1}^{c} \mu_k} t^{\sum_{i = 1}^{m}
         (\sum_{j = 1}^{n_i} \nu^{(i)}_{j})f_i} \cdot \\
 \sum_{(\nu^{(m+1)}, \dots, \nu^{(\pnum)})} \left( \prod_{i = m+1}^{\pnum} \beta(\nu^{(i)}; q_i) \right) t^{\sum_{i = m+1}^{\pnum} (\sum_{j=1}^{n_i} \nu_j^{(i)} ) f_i }.
\end{multline}
Observing that 
\begin{equation*}
\sum_{\nu^{(i)}} \beta(\nu^{(i)}; q_i) t^{\sum_{j = 1}^{n_i} \nu_j^{(i)} f_i} = \sum_{\mathrm{M} \leq \mfO_i^{n_i}} [\mfO_i^{n_i} : \mathrm{M}]^{-s} = \zeta_{\mfO_i^{n_i}}(s),
\end{equation*}
we see that the second sum in~\eqref{equ:intermediate} is equal to $\prod_{i = m+1}^{\pnum} \zeta_{\mfO_i^{n_i}}(s)$.  The claim follows.
\end{proof}

Let $w\in \mathcal D_{2c}$ be a Dyck word. Recall, from
Section~\ref{subsec:prelim.dyck}, that $w$ is specified by two
$r$-tuples $(L_1,L_2,\dots,L_r)$ and $(M_1,M_2,\dots,M_r)$ satisfying
$L_i-M_i \geq 0$ for all $i\in[r]$ and $L_r=M_r=c$. Recall further
that $\varepsilon = c - c^\prime$ and define $\widetilde{L}_j = {L}_j
- \varepsilon$ for all $j \in [r]$.

\begin{dfn} \label{dfn:overlaps}
Let $\lambda$ and $\mu$ be partitions with $c^\prime$ and $c$ parts,
respectively, and let $w \in \mathcal{D}_{2 c}$ such that $L_1 \geq
\varepsilon$.  Fix a partition $\widetilde{\lambda}$ with $c$ parts as
above; without loss of generality we may take $\xi_\varepsilon \geq
\max\{ \lambda_1, \mu_1 \}$.  We say that $\lambda$ and $\mu$ have
       {\emph{overlap type $w$}}, written $w(\lambda, \mu) = w$, if
       their parts satisfy the following inequalities:
\begin{multline*} 
 \xi_1 \geq \cdots \geq \xi_{\varepsilon} \geq \lambda_1 \geq \cdots \geq \lambda_{\widetilde{L}_1} \geq \mu_1 \geq \cdots \geq
 \mu_{{M}_1} > \lambda_{\widetilde{L}_1 + 1} \geq \cdots \geq \lambda_{\widetilde{L}_2} \geq \\
 \mu_{{M}_1 + 1} \geq \cdots \geq \mu_{{M}_2} > \cdots >
 \lambda_{\widetilde{L}_{r-1}+1} \geq \cdots \geq \lambda_{\widetilde{L}_r} \geq \mu_{{M}_{r-1} + 1}
 \geq \cdots \geq \mu_{{M}_r}.
\end{multline*}
In other words, $w(\lambda, \mu) = w$ if $w(\widetilde{\lambda},\mu) = w$
in the sense of~\eqref{mult:limi}.  Note that $\widetilde{L}_1 = 0$ may
occur, if $\varepsilon > 0$.  Moreover, the set $\mathcal{D}_{2c}$
depends on $c$ and so on the choice of $A$.
\end{dfn}

Observe that $\lambda \geq \mu$ if and only if there exists a Dyck
word $w \in \mathcal{D}_{2 c}$, necessarily unique, such that
$w(\lambda, \mu) = w$.  Given $w \in \mathcal{D}_{2c}$, we define the
function
\begin{equation} \label{equ:dw.definition}
 D_w(q,t) = \sum_{\boldsymbol{\nu}} \sum_{{\mu} \leq
   {{\lambda}}(\boldsymbol{\nu}) \atop w({{\lambda}(\boldsymbol{\nu})}, {\mu}) = w}
 \left( \prod_{i = 1}^m \beta(\nu^{(i)}; q_i) \right)
 \alpha({{\lambda}}(\boldsymbol{\nu}), {\mu};q) (q^{n}t)^{\sum_{k =
     1}^{c} \mu_k} t^{\sum_{i = 1}^m (\sum_{j = 1}^{n_i}
   \nu^{(i)}_{j})f_i}.
\end{equation}

\begin{rem} \label{rem:dyck.word.epsilon}
  If $w$ is such that $L_1 < \varepsilon$, then the above sum is empty
  and so $D_w(q,t) = 0$.  In addition, the definition of the partition
  ${{\lambda}}(\boldsymbol{\nu})$ will usually impose some equalities
  among its parts.  Thus, it may happen that the set of projection
  data $\boldsymbol{\nu}$ whose associated partition
  ${{\lambda}}(\boldsymbol{\nu})$ is compatible with a given Dyck word
  $w$ is empty even if $w$ satisfies the condition $L_1 \geq
  \varepsilon$ of Definition~\ref{dfn:overlaps}.  We will see examples
  of this phenomenon below, e.g.\ in Section~\ref{subsec:g22}.
\end{rem}

Proposition~\ref{pro:zeta.rewrite} now tells us that
\begin{equation}\label{eq:zeta.dw}
\zidealo_L(s) = \frac{\zeta_{\mfo^{n}}(s)}{\prod_{i = 1}^m
  \zeta_{\mfO_i^{n_i}}(s)}\sum_{w \in \mathcal{D}_{2c}} D_w(q,t).
\end{equation}

\subsection{An explicit expression for
  $\zidealo_L(s)$} \label{sec:statement} Our aim in this section is to
give explicit formulae for the terms $D_w(q,t)$ in
\eqref{eq:zeta.dw}. We will achieve it with
Proposition~\ref{thm:beef.is.here}---a result whose proof will be
given in Section~\ref{subsec:proof.of.beef}---leading to a fully
explicit formula for the relevant $\lri$-ideal zeta functions in
Theorem~\ref{thm:zeta.explicit}.

We maintain the notation of Section~\ref{subsec:rewriting} and resume
some of the notation introduced in Section~\ref{sec:general.igusa}.  Consider the
composition $\underline{n} = (n_1, \dots, n_m)$ and a family $\bsnu =
(\nu^{(1)},\dots,\nu^{(m)})$ of partitions $\nu^{(i)}$, each with
$n_i$ parts. The natural ordering of the elements of the
multiset
\begin{equation*}
  S = \bigcup_{i = 1}^m \left\{ \nu^{(i)}_{j} \mid j \in [n_i] \right\}
\end{equation*}
gives rise to an element $V(\boldsymbol{\nu}) \in
\WO_{\underline{n}}$.  Indeed, the word $v =\prod_{i = 1}^m
a_i^{\alpha_i} \in C_{\underline{n}}$ appears in the flag
$V(\boldsymbol{\nu})$ if and only if any element of the multiset
\begin{equation} \label{equ:sr.def} S_v = \bigcup_{i = 1}^m \left\{
    \nu^{(i)}_{j} \mid j \in [\alpha_i] \right\}
\end{equation} 
is larger than any element of the complement
$S \setminus S_v = \bigcup_{i=1}^m \left\{ \nu^{(i)}_{j} \mid j \in [\alpha_i + 1, n_i] \right\}.$ 
Given a word $v \in C_{\underline{n}}$, let $m(v)$ denote a minimal
element of the multiset $S_v$.  Since, by virtue of
Definition~\ref{dfn:lambda.nu}, all parts of
${\lambda}(\boldsymbol{\nu})$ appear in $S$, we see that if $k \in \N$
satisfies $\lambda^\prime_k > \lambda^\prime_{k+1}$, then necessarily
$k = m(v)$ for some $v \in C_{\underline{n}}$.  Here we denote the
dual partition of $\lambda(\boldsymbol{\nu})$ by $\lambda^\prime$ for
brevity.  Moreover, Hypothesis~\ref{hypothesis} implies that
$\lambda^\prime_{m(v)}$ depends only on $v$ and not on the flag
$V(\boldsymbol{\nu})$ or on the actual values of the parts
$\nu^{(i)}_{j}$.

\begin{dfn} \label{dfn:rad}
Let $v \in C_{\underline{n}}$.
\begin{enumerate} 
\item Set $\ell(v) = \lambda^\prime_{m(v)}$.  In particular, $\ell(v^\prime) \leq \ell(v)$ if $v^\prime \leq v$.
\item We say that $v$ is {\emph{radical}} if
  $\ell(v^\prime) < \ell(v)$ for all proper subwords $v^\prime < v$.
\end{enumerate}
\end{dfn}

Note the following explicit formula for $\ell(v)$.

\begin{lem} \label{lem:length.function}
Let $v = \prod_{i = 1}^m a_i^{\alpha_i} \in C_{\underline{n}}$.  Then
\begin{equation*}
\ell(v) = \lambda(\boldsymbol{\nu})^\prime_{m(v)} = \sum_{k = 1}^{\klim}  \prod_{j = 1}^{\tau_k} \binom{\alpha_{s_{kj}}}{\sigma_{kj}} .
\end{equation*}
\end{lem}
\begin{proof}
This is a straightforward consequence of Definition~\ref{dfn:lambda.nu}.
\end{proof}

\begin{dfn} Let $w \in \mathcal{D}_{2c}$ be a Dyck word {with exactly
    $r$ letter changes from $\bfz$ to $\bfo$}; cf.\
  Section~\ref{subsec:prelim.dyck}. A flag
  $V = \{ v_1 < \cdots < v_t \}$ of elements of $C_{\underline{n}}$ is
  said to be {\emph{compatible}} with $w$, or simply
  $w$-\emph{compatible}, if
  \begin{itemize}
  \item $t = r$,
  \item for all $j \in [r]$, the word $v_j$ is radical and satisfies
    $\ell(v_j) = \widetilde{L}_j$.
  \end{itemize}
\end{dfn} 

\begin{rem} \label{rmk:rho.ir} 
It follows from Hypothesis~\ref{hypothesis} that all parts of
$\boldsymbol{\nu}$ participate in the minima that determine the parts
of $\lambda(\boldsymbol{\nu})$.  Therefore, the maximal word $\prod_{i
  = 1}^m a_i^{n_i}$ is always radical, and $v_r = \prod_{i = 1}^m
a_i^{n_i}$ for any $w$-compatible flag $V$.

In addition, note that if $\varepsilon > 0$, i.e.~if $L^\prime < A$,
then some Dyck words $w \in \mathcal{D}_{2c}$ for which there exist
$w$-compatible flags will satisfy $\tilde{L}_1 = 0$.  In this case,
$v_1 = \varnothing$ for any such flag.
\end{rem}

For $w \in \mathcal{D}_{2c}$, let $\mathcal{F}_w$ denote the set of
$w$-compatible flags. It will be convenient to organize the
information carried by an element of $\mathcal{F}_w$ in matrix form.
Given an element $V = \{ v_1 < \cdots < v_{r} \} \in \mathcal{F}_w$,
we let $v_0$ be the empty word and define $\rho_{ij}$, for $i \in [m]$
and $j \in [r]$, by
\begin{equation*}
\frac{v_j}{v_{j-1}} = \prod_{i = 1}^m a_i^{\rho_{ij}}.
\end{equation*}
In this way, the flag $V$ gives rise to a matrix
${\rho}(V) \in \mathrm{Mat}_{m,r}(\N_0)$.  Conversely, given a
matrix ${\rho} \in \mathrm{Mat}_{m,r}(\N_0)$, we consider
the cumulative sums of its rows: for $i \in [m]$ and $j \in [r]$,
define
\begin{equation}\label{def:rho}
  \Rho_{ij} = \sum_{k = 1}^j \rho_{ik}.
\end{equation}

\begin{dfn} \label{dfn:adm.comp} Let
  $\mathcal M_{\underline{n}, w}\subseteq \mathrm{Mat}_{m,r}(\N_0)$
  be the set of $({\underline{n}} , w)$-\emph{admissible
    compositions}, namely of matrices $\rho$ satisfying the following
  two properties:
\begin{enumerate}
\item $\ell(\prod_{i = 1}^m a_i^{\Rho_{ij}}) = \widetilde{L}_j$ for all $j \in [r]$.
\item The word $\prod_{i = 1}^m a_i^{\Rho_{ij}}$ is radical for all $j \in [r]$.
\end{enumerate}
\end{dfn}

By Remark~\ref{rmk:rho.ir}, these properties imply that $\Rho_{ir} = n_i$ for all $i \in [m]$.
Set $w_j = \prod_{i = 1}^m a_i^{\Rho_{ij}}$ for all $j \in [r]$. It
is easy to see that the map
$\mathcal{F}_w \rarr \mathcal M_{\underline{n}, w}$ given by
$V \mapsto {\rho}$ is a bijection, with inverse
${\rho} \mapsto \left\{w_1 < \dots < w_{r} \right\}$.  Denote
$$\Rho_i = \{ \Rho_{ij} \mid j \in [r] \}$$ for all $i \in [m]$.  For
$j \in [r]$, we denote by $\underline{\rho_j}$ the following composition
with $m$ parts:
\begin{equation} \label{equ:def.kj} \underline{\rho_j} = (\rho_{1j}, \dots,
  \rho_{mj}).
\end{equation}
Recall from Definition~\ref{def:igusa} the notion of Igusa function
and from Definition~\ref{def:igusa.wo} the notion of generalized Igusa
function
$I^{\textup{wo}}_{\underline{\rho_j}}(\bfY;\bfX) \in \Q(Y_1,\dots,Y_{m};(X_v)_{v
  \leq w_j})$.

\begin{dfn} \label{dfn:dwk} Let
  $ \rho \in \mathcal M_{\underline n, w}$.  We define
\begin{multline*}
  D_{w, \rho}(q,t)= \left(\prod_{i=1}^m
  \binom{n_i}{\Rho_i}_{q_i^{-1}}\right) \prod_{j=1}^{r} \left(
  \binom{L_j-M_{j-1}}{L_j-M_j} _{q^{-1}}
  I^{\textup{wo}}_{\underline{\rho_j}}(q_1^{-1}, \dots,
  q_m^{-1};\bfy^{(j)}) \right)\cdot
  \\\prod_{j=1}^{r-1}I_{M_j-M_{j-1}}^\circ (q^{-1};x_{M_{j-1}+1},
  \dots, x_{M_j}) I_{M_r-M_{r-1}}(q^{-1};x_{M_{r-1}+1}, \dots,
  x_{M_r}), \end{multline*} with numerical data defined as
follows.  For a
subword $v=\prod_{i=1}^m a_i^{\alpha_i}$ of $\prod_{i=1}^m
a_i^{\rho_{ij}}$ we set $\alpha_i^{(j)} = \Rho_{i,j-1} + \alpha_i$ and
${v}^{(j)} = v \cdot w_{j-1} = \prod_{i = 1}^m a_i^{\alpha_i^{(j)}}$.
Set
$${\delta_v^{(j)}} = \begin{cases}0, &  \textup{ if }\ell(v^{(j)}) = \ell(w_{j-1}),\\  1,&  \textup{ otherwise,} \end{cases}$$ and define
$$ B_v^{(j)} = \begin{cases}
  \sum_{i = 1}^m f_i\alpha_i (n_i - \alpha_i) , & \textup{ if } {\delta}_v^{(j)} = 0, \\
  \sum_{i = 1}^m f_i\alpha_i^{(j)}(n_i - \alpha_i^{(j)}) , & \textup{
    if }{\delta}_v^{(j)} = 1.
\end{cases}
$$
Finally, we set  
$$y_v^{(j)} = q^{\delta_v^{(j)} M_{j-1}(n + \ell(v^{(j)}) + \varepsilon - M_{j-1}) + B_v^{(j)} } t^{\sum_{i = 1}^m \alpha_i f_i + \delta_v^{(j)}(M_{j-1} + \sum_{i = 1}^m \Rho_{i,j-1} f_i)},$$
where $\ell(v^{(j)})$ is given explicitly by Lemma~\ref{lem:length.function}.  For $k \in [M_{j-1} + 1, M_j]$, we set 
\begin{equation*}
  x_k   =  q^{k(n + L_j - k) + \sum_{i = 1}^m f_i\Rho_{ij}(n_i - \Rho_{ij})} t^{k + \sum_{i = 1}^m f_i \Rho_{ij}}.
\end{equation*}
\end{dfn}

\begin{pro}\label{pro:dwfeq}
The following functional equation holds:
\begin{equation*}
 D_{w, \rho}(q^{-1},t^{-1})= (-1)^{c + \sum_{i = 1}^m n_i} q^{\binom{n+c}{2} - \binom{n}{2} + \sum_{i = 1}^m f_i  \binom{n_i}{2}} t^{c + 2 \sum_{i = 1}^m n_i f_i}
D_{w, \rho}(q,t).
\end{equation*}
\end{pro}
\begin{proof}
  The proof is a straightforward computation using the functional
  equations of
  \begin{enumerate}
  \item Gaussian binomials \eqref{eq:funeq.gauss},
  \item classical Igusa functions \eqref{eq:funeq.igusa},
    \eqref{eq:funeq.igusa.circ}, and
  \item generalized Igusa functions {given in
    Theorem~\ref{thm:funeq}},
  \end{enumerate}
  as well as the definition \eqref{def:rho} of~$\Rho_{ij}$ and the
  observation that $\Rho_{ir} = n_i$ for all $i \in [m]$.
\end{proof}

Recall the functions $D_w(q,t)$ introduced
in~\eqref{equ:dw.definition} and used to describe the $\lri$-ideal
zeta function of $L$ in~\eqref{eq:zeta.dw}.  The following result,
which constitutes the technical heart of the computation of the ideal
zeta function $\zidealo_L(s)$, relates $D_w(q,t)$ with the explicit
functions $D_{w,\rho}(q,t)$ of Definition~\ref{dfn:dwk}.  We defer its
proof to the next section.

\begin{pro} \label{thm:beef.is.here} Let $w \in \mathcal{D}_{2c}$ be a
  Dyck word.  Then
  \begin{equation*}
D_w(q,t) = \sum_{{\rho} \in \mathcal{M}_{\underline{n},w}} D_{w, {\rho}}(q,t).
\end{equation*}
\end{pro}

\begin{thm} \label{thm:zeta.explicit} The $\lri$-ideal zeta function
  of $L$ is
\begin{equation*}
\zidealo_L(s) = \frac{\zeta_{\mfo^{n}}(s)}{\prod_{i = 1}^m \zeta_{\mfO_i^{n_i}}(s)}\sum_{w \in \mathcal{D}_{2c}} \sum_{{\rho} \in \mathcal{M}_{\underline{n},w}} D_{w, {\rho}}(q,t).
\end{equation*}
\end{thm}
\begin{proof}
The claim is immediate from~\eqref{eq:zeta.dw} and
Proposition~\ref{thm:beef.is.here}.
\end{proof}

\begin{cor} \label{cor:funct.eq} 
Suppose that the extension $\mfO_i / \mfo$ is unramified for all $i
\in [m]$.  Then the $\lri$-ideal zeta function of $L$ satisfies the
functional equation
\begin{equation*}
\left.\zidealo_L(s)\right\rvert_{q\rarr q^{-1}} =
(-1)^{\mathrm{rk}_{\lri}(L)}q^{\binom{\mathrm{rk}_{\lri}(L)}{2}}t^{\mathrm{rk}_{\lri}(L) + \mathrm{rk}_{\lri}(L/Z(L))} \zidealo_L(s).
\end{equation*}
\end{cor}
\begin{proof} 
Recall that $n + c = \mathrm{rk}_{\lri}(L/A) + \mathrm{rk}_{\lri}(A) =
\mathrm{rk}_{\lri}(L)$.  Observe that the symmetry factor in
Proposition~\ref{pro:dwfeq} is independent of $w$ and
${\rho}$. Consequently, the sum $\sum_{w \in \mathcal{D}_{2c}}
\sum_{{\rho} \in \mathcal{M}_{\underline{n},w}} D_{w, {\rho}}(q,t)$
itself satisfies a functional equation with the same symmetry factor.
The remaining factors in Theorem~\ref{thm:zeta.explicit} satisfy
$$\left. \frac{\zeta_{\mfo^{n}}(s)}{\prod_{i = 1}^m
  \zeta_{\mfO_i^{n_i}}(s)}\right\rvert_{q\rarr
  q^{-1}}=\frac{(-1)^{n}q^{\binom{n}{2}}t^{n}}{\prod_{i=1}^{m}(-1)^{n_i}q^{f_i
    \binom{n_i}{2}}t^{n_i f_i}} \cdot \frac{\zeta_{\mfo^{n}}(s)}{\prod_{i =
    1}^m \zeta_{\mfO_i^{n_i}}(s) }.$$ This yields the functional
equation
\begin{equation*}
\left.\zidealo_L(s)\right\rvert_{q\rarr q^{-1}} =
(-1)^{\mathrm{rk}_{\lri}(L)}q^{\binom{\mathrm{rk}_{\lri}(L)}{2}}t^{\mathrm{rk}_{\lri}(L) + \sum_{i = 1}^m n_i f_i} \zidealo_L(s).
\end{equation*}
Since we have assumed that all the extensions $\mfO_i / \mfo$ are
unramified, our claim is now immediate from~\eqref{equ:cocenter}.
\end{proof}

\begin{rem}\label{rem:local.abs.con}
  The explicit formula given in Theorem~\ref{thm:zeta.explicit} allows
  one to determine, in principle, the (local) \emph{abscissa of
    convergence} $\alpha^{\triangleleft\,\lri}_L$ of the $\lri$-ideal
  zeta function $\zeta^{\triangleleft\,\lri}_{L}(s)$, viz.
$$\alpha^{\triangleleft\,\lri}_L := \inf\left\{ \alpha\in \R_{>0} \mid
\zeta^{\triangleleft\,\lri}_{L}(s) \textup{ converges on }\{ s\in\C
\mid \Re(s) > \alpha\} \right\} \in \Q_{>0}.$$

Indeed, if one writes
the rational function $\zeta^{\triangleleft\,\lri}_{L}(s)$ over a common
denominator of the form $\prod_{(a,b)\in I} (1-q^at^b)$, with $a,b$
given by the numerical data given in Definition~\ref{dfn:dwk}, then
  $$\alpha^{\triangleleft\,\lri}_L = \max\left\{ n, \frac{a}{b}\mid
(a,b)\in I\right\}.$$ This reflects the facts that $a/b$ is the
abscissa of convergence of the geometric progression
$q^{a-bs}/(1-q^{a-bs})$ and that each of the $D_w(q,t)$ is a non-negative
linear combination of products of such geometric progressions.
\end{rem}

\begin{rem}
  Observe that if $L$ is replaced by the $\beta$-fold direct product
  $L^\beta$, then $c$ is replaced by $\beta c$, and the number of
  summands on the right-hand side of Theorem~\ref{thm:zeta.explicit}
  grows super-exponentially in $\beta$.  Cancellations may occur, as
  in Remark~\ref{rmk:grenham.explicit} below, that cause the
  complexity of $\zeta_{L^\beta}^{\ideal \mathfrak{o}} (s)$ to grow
  less rapidly with respect to $\beta$; however, explicit computations
  in the case of the Heisenberg Lie algebra suggest that the growth
  can indeed be this rapid.  By contrast, if $L^{\ast \beta}$ is the
  $\beta$-fold amalgamation of $L$ over its derived subring, then the
  complexity of $\zeta_{L^{\ast \beta}}^{\ideal \mathfrak{o}} (s)$
  grows in a precisely controlled way~\cite[Theorem~1.1]{BS/23}.
\end{rem}

\subsection{Proof of Proposition~\ref{thm:beef.is.here}}\label{subsec:proof.of.beef} 
We start with a lemma involving the notions of
Definition~\ref{dfn:rad}.  This observation is simple but crucial
to the method of the article.

\begin{lem} \label{lem:def.radical} Let $v \in C_{\underline{n}}$.
  There is a unique radical subword $\sqrt{v} \leq v$ such that
  $\ell(\sqrt{v}) = \ell(v)$.
\end{lem}
\begin{proof} 
Suppose $v = \prod_{i = 1}^m a_i^{\alpha_i}$.  If a binomial
coefficient $\binom{\alpha}{\sigma}$ is positive, then it will
decrease if $\alpha$ is decreased.  It follows that if the $k$-th term
in the sum in the statement of Lemma~\ref{lem:length.function} is
positive, then in any subword $v^\prime \leq v$ satisfying
$\ell(v^\prime) = \ell(v)$ all the variables $a_{s_{kj}}$ must appear
with exponent $\alpha_{s_{kj}}$.  Hence we are led to define the set
$$ \mathcal{K}_v = \left\{ k \in [\klim] \mid \alpha_{s_{kj}} \geq
\sigma_{kj} \, \text{for all} \, j \in [\tau_k] \right\}.$$
Furthermore, we put $\mathfrak{S}_v = \bigcup_{k \in \mathcal{K}_v}
\mathfrak{S}_k$ and finally define $\sqrt{v} = \prod_{i \in
  \mathfrak{S}_v} a_i^{\alpha_i}$.  It is clear from the preceding
discussion that a subword $v^\prime \leq v$ satisfies $\ell(v^\prime)
= \ell(v)$ if and only if $\sqrt{v} \leq v^\prime \leq v$.  The
claimed existence and uniqueness follow.
\end{proof}

\begin{cor} \label{lem:radical.well.def}
Suppose that $v_1 < v_2$ are two elements of $C_{\underline{n}}$ such
that $\ell(v_1) = \ell(v_2)$.  Then $\sqrt{v_1} = \sqrt{v_2}$.
\end{cor}
\begin{proof}
This is immediate from the construction of $\sqrt{v}$ in the proof of
Lemma~\ref{lem:def.radical}.
\end{proof}

Fix a Dyck word $w \in \mathcal{D}_{2c}$.  We aim to evaluate the
function $D_w(q,t)$ of~\eqref{equ:dw.definition}.  Let
$\boldsymbol{\nu} = ({\nu}^{(1)}, \dots, {\nu}^{(m)})$ be an $m$-tuple
of partitions, where, for each $i \in [m]$, the partition
${\nu}^{(i)}$ has $n_i$ parts.  Let ${\mu}$ be a partition with $c$
parts such that ${\mu} \leq {\lambda}(\boldsymbol{\nu})$ and
$w({\lambda}(\boldsymbol{\nu}), {\mu}) = w$, in the sense of
Definitions~\ref{dfn:dominance.of.partitions} and~\ref{dfn:overlaps}.
To simplify the notation, we write ${\lambda}$ for
${\lambda}(\boldsymbol{\nu})$.

Now let $\{ L_j, M_j \}_{j \in [r]}$ be the parameters associated with
the Dyck word $w$.  Recall that we have set $L_0 = M_0 = 0$.  It
follows from the assumption $w({\lambda}, {\mu}) = w$ that
$\lambda_{\widetilde{L}_j} > \lambda_{\widetilde{L}_j + 1}$ for all
$j \in [r-1]$, hence that all the positive $\widetilde{L}_j$ appear as
parts of the dual partition ${\lambda}^\prime$.  By the observations
before Definition~\ref{dfn:rad}, there exists a subflag
$\kappa_1 < \kappa_2 \cdots < \kappa_r$ of $V(\boldsymbol{\nu})$ such
that $\ell(\kappa_j) = \widetilde{L}_j$ for every $j \in [r]$; if
$\widetilde{L}_1 = 0$, then we may take $\kappa_1 = \varnothing$.
This subflag need not be unique, and its constituent words need not be
radical.  However, the flag
$\sqrt{\kappa_1} < \cdots < \sqrt{\kappa_r}$ is well-defined by
Corollary~\ref{lem:radical.well.def}.  Moreover, it is clear that this
flag is an element of $\mathcal{F}_w$ and thus corresponds to an
$(\underline{n}, w)$-admissible composition
${\rho}(\boldsymbol{\nu}) \in \mathcal{M}_{\underline{n}, w}$.

For every ${\rho} \in \mathcal{M}_{\underline{n},w}$ we define the function
\begin{equation} \label{equ:deltaw.definition}
 \Delta_{w, {\rho}}(q,t) = \sum_{\boldsymbol{\nu} \atop
   {\rho}(\boldsymbol{\nu}) = {\rho}} \sum_{{\mu} \leq
   \lambda(\boldsymbol{\nu}) \atop w({{\lambda}}, {\mu}) = w} \left(
 \prod_{i = 1}^m \beta({\nu}^{(i)};q_i) \right)
 \alpha({{\lambda}}(\boldsymbol{\nu}), {\mu};q) (q^{n}t)^{\sum_{k
     = 1}^{c} \mu_k} t^{\sum_{i = 1}^m \sum_{j = 1}^{n_i}
   \nu^{(i)}_{j}}.
\end{equation}
Clearly,
$D_w(q,t) = \sum_{{\rho} \in \mathcal{M}_{\underline{n}, w}}
\Delta_{w, {\rho}}(q,t)$.  Hence, to prove
Proposition~\ref{thm:beef.is.here} it suffices to show the following:

\begin{lem} \label{lem:most.technical.lemma}
The equality $\Delta_{w, {\rho}}(q,t) = D_{w, {\rho}}(q,t)$ holds for all ${\rho} \in \mathcal{M}_{\underline{n},  w}$.  
\end{lem}
\begin{proof}
Fix ${\rho} \in \mathcal{M}_{\underline{n}, w}$. For each $j \in [r]$ we define a multiset
\begin{equation*}
\mathcal{S}_j =  \bigcup_{i = 1}^m \left\{ \nu^{(i)}_{k} \mid k \in [\Rho_{i,j-1} + 1, \Rho_{ij}] \right\}.
\end{equation*}
Recall the compositions $\underline{\rho_j}$ defined
in~\eqref{equ:def.kj} above, which depend only on ${\rho}$.  For each
$j \in [r]$, the natural ordering of the elements of $\mathcal{S}_j$
provides a weak order $v_j \in \WO_{\underline{\rho_j}}$.  Again,
these depend only on the projection data $\boldsymbol{\nu}$, so we
denote them $v_j(\boldsymbol{\nu})$ and set
$\boldsymbol{v}(\boldsymbol{\nu}) = (v_1(\boldsymbol{\nu}), \dots,
v_r(\boldsymbol{\nu}))$.  As in the previous section, we define
$w_j = \prod_{i = 1}^m a_i^{\Rho_{ij}} \in C_{\underline{n}}$.

Now fix an $r$-tuple
$(v_1, \dots, v_r) \in \prod_{j = 1}^r \WO_{\underline{\rho_j}}$.  For
every $j \in [r]$, suppose that $v_j$ includes the word
$\prod_{i = 1}^m a_i^{\rho_{ij}}$ (except when $\underline{\rho_1}$ is
the zero composition, in which case $v_1$ is empty).  Write
\begin{equation*}
v_j = \{ v_{j1} < v_{j2} < \cdots < v_{j, \ell_j} \}
\end{equation*}
for some $\ell_j \in \N_0$.  We define
$\widetilde{v}_{jk} = w_{j-1} \cdot v_{jk} \in C_{\underline{n}}$.
Consider the set $S_{\widetilde{v}_{jk}}$ and its minimal element
$m(\widetilde{v}_{jk})$ as in~\eqref{equ:sr.def}.  Note that
$v_{j,\ell_j} = \prod_{i = 1}^m a_i^{\rho_{ij}}$ and that consequently
$m(\widetilde{v}_{j, \ell_j}) = \lambda_{\widetilde{L}_j}$.  Let
$\epsi_j \in \N$ be the minimal positive integer such that
$\ell(\widetilde{v}_{j, \epsi_j}) > \widetilde{L}_{j-1}$.  Then
$m(\widetilde{v}_{j, \epsi_j}) = \lambda_{\widetilde{L}_{j-1} + 1}$.
Observe that $\delta^{(j)}_{v_{jk}} = 0$ in Definition~\ref{dfn:dwk}
if and only if $k < \epsi_j$; in this case, $m(\widetilde{v}_{jk})$ is
equal to a part of $\boldsymbol{\nu}$ that does not appear in the
partition $\lambda(\boldsymbol{\nu})$.

For every element $v_{jk} = \prod_{i = 1}^m a_i^{\gamma_i} \in
C_{\underline{\rho_j}}$, define $$m(v_{jk}) = \min \{ \nu^{(i)}_u \mid
u \in [\Rho_{i,j-1} + 1, \Rho_{i,j-1} + \gamma_i] \}.$$ Note that the
elements of the set $\{ m(v_{jk}) \mid j \in [r], k \in [\ell_j] \}$ are
exactly the parts of the projection data $\boldsymbol{\nu}$.
Moreover, if $\delta^{(j)}_{v_{jk}} = 1$, then $m(v_{jk}) =
m(\widetilde{v}_{jk})$.  Otherwise, it may happen that $m(v_{jk}) >
m(\widetilde{v}_{jk})$, as the set defining $m(v_{jk})$ consists
entirely of parts of $\boldsymbol{\nu}$ that do not appear in
$\lambda(\boldsymbol{\nu})$ and may all be larger than the minimal
element of the disjoint set $S_{w_{j-1}}$.

We now
define a collection of differences that will provide a convenient
parametrization of the pairs $(\boldsymbol{\nu}, {\mu})$ that we are
considering:
\begin{align*}
  s_{jk} & =  \begin{cases}
    m({v}_{jk}) - m({v}_{j,k+1}), &\textup{for } k < \ell_j, \\
    m({v}_{jk}) - \mu_{M_{j-1} + 1}, &\textup{for } k = \ell_j,
\end{cases} 
  \\
  r_k & =  \begin{cases}
    \mu_k - m({v}_{j+1, \epsi_{j+1}}), &\textup{ for }k \in \{ M_1, \dots, M_{r-1} \}, \\
    \mu_k, &\textup{ for } k = M_r, \\
    \mu_k - \mu_{k+1}, &\textup { otherwise.}
\end{cases}
\end{align*}
Here the indices of the $r_k$ run over the set $[M_r] = [c]$, whereas
the indices of the $s_{jk}$ satisfy $j \in [r]$ and $k \in [\ell_j]$.
We emphasize that the $r_k$ have no connection with the parameter $r$
defined earlier.  Observe that the $s_{jk}$ and the $r_k$ are all
non-negative integers.  Moreover, if we allow all the $s_{jk}$ to run
over $\N_0$ and all the $r_k$ to run over $\N$ if $k \in \{ M_1,
\dots, M_{r-1} \}$ and over $\N_0$ otherwise, then we obtain
precisely the pairs $(\boldsymbol{\nu}, {\mu})$ satisfying the following
three conditions:
\begin{enumerate}
\item $w({\lambda}(\boldsymbol{\nu}), {\mu}) = w$
\item ${\rho}(\boldsymbol{\nu}) = {\rho}$
\item $\boldsymbol{v}(\boldsymbol{\nu}) = (v_1, \dots, v_r)$.
\end{enumerate}

Let $\Delta_{w,{\rho}, \boldsymbol{v}}(q,t)$ be the function defined
by the right-hand side of~\eqref{equ:deltaw.definition}, except that
the sum runs only over the data $\boldsymbol{\nu}$ satisfying
$\boldsymbol{v}(\boldsymbol{\nu}) = (v_1, \dots, v_r)$.  Our task is now
to rewrite the ingredients of $\Delta_{w,{\rho},
  \boldsymbol{v}}(q,t)$, and hence the function itself, in terms of
the parameters $s_{jk}$ and $r_k$.
Consider
the following collection of intervals:
\begin{align} \label{align:intervals}
& [\mu_k - r_k + 1, \mu_k],  && k \in [c], \\ \nonumber
& [m({v}_{jk}) - s_{jk} + 1, m({v}_{jk})], & & j \in [2,r], \, k \in [\epsi_j, \ell_j].
\end{align}
The reader will easily verify that these intervals are disjoint and
that their union is the interval $[1, \mu_1]$.  It follows from this
observation that
\begin{equation} \label{equ:mu.differences}
\mu_k = \sum_{b = k}^c r_b + \sum_{b = j+1}^r \sum_{u = \epsi_b}^{\ell_b} s_{bu}
\end{equation}
if $k \in [M_{j-1} + 1, M_j]$, whereas if $\nu^{(i)}_d = m(v_{jk})$, then
\begin{equation} \label{equ:nu.differences}
\nu^{(i)}_d = \sum_{u = k}^{\ell_j} s_{ju} + \sum_{b = j+1}^r \sum_{u = \epsi_b}^{\ell_b} s_{bu} + \sum_{b = M_{j-1} + 1}^c r_b.
\end{equation}

We now treat the ingredients of
$\Delta_{w, \rho, \boldsymbol{\nu}}(q,t)$, starting with the
$\beta(\nu^{(i)};q_i)$.  Since ${\rho}(\boldsymbol{\nu}) = {\rho}$, it
follows that
$\{ \Rho_{ij} \mid j \in [r-1] \} \subseteq J_{\nu^{(i)}}$ for all
$i \in [m]$.  For every $j \in [r]$ define the set
\begin{equation*}
J_{\nu^{(i)}}^{(j)}  =  \{ k - \Rho_{i,j-1} \mid k \in J_{\nu^{(i)}} \cap (\Rho_{i,j-1}, \Rho_{ij}) \}.
\end{equation*}
Lemma~\ref{lem:binom} implies that
\begin{equation} \label{equ:beta.binomial}
\binom{n_i}{J_{\nu^{(i)}}}_{q_i^{-1}}  =  \binom{n_i}{\Rho_i}_{q_i^{-1}} \prod_{j = 1}^r \binom{\rho_{ij}}{J_{\nu^{(i)}}^{(j)}}_{q_i^{-1}}. 
\end{equation}
Using~\eqref{equ:mu.differences} and~\eqref{equ:nu.differences}, the
differences $\nu^{(i)}_{d} - \nu^{(i)}_{d+1}$ appearing in the
exponents in $\beta({\nu}^{(i)};q_i)$, as defined
in~\eqref{equ:beta.definition}, can be expressed as sums of distinct
parameters $s_{jk}$ and $r_k$.  In particular, we observe that the
elements of $J_{\nu^{(i)}}^{(j)}$ are precisely the exponents of the
variable $a_i$ that occur in the weak order $v_j$.  It then follows
from~\eqref{equ:beta.binomial} that
\begin{equation*}
\prod_{i = 1}^m \binom{n_i}{J_{\nu^{(i)}}}_{q_i^{-1}} = \prod_{i = 1}^m \binom{n_i}{\Rho_i}_{q_i^{-1}}  \prod_{j = 1}^r \binom{\underline{\rho_j}}{v_j}_{\bfY},
\end{equation*}
where $\bfY = (q_1^{-1}, \dots, q_m^{-1})$ and
$\binom{\underline{\rho_j}}{v_j}_{\bfY}$ is as in
Definition~\ref{dfn:chain.binomial}.  This completes our analysis of
the factors $\beta({\nu}^{(i)};q_i)$.

We now consider the factors
$\alpha({\lambda}(\boldsymbol{\nu}), {\mu};q)$, using the idea behind
the proofs of~\cite[Lemmata 2.16 and 2.17]{SV1/15}.  The range of
parameters $k$ over which the infinite product
of~\eqref{def:birkhoff.poly} giving
$\alpha({\lambda}(\boldsymbol{\nu}), {\mu}; q) =
\alpha(\widetilde{\lambda}, \mu; q)$ may have non-trivial factors is
precisely $[1, \mu_1]$.  Recall that
$\widetilde{\lambda}^\prime_k = \lambda_k^\prime + \epsi$ for all $k$
and observe that the dual partitions $\widetilde{\lambda}^\prime$ and
${\mu}^\prime$ are constant on each of the intervals
of~\eqref{align:intervals}.  Indeed, if
$d \in [\mu_k - r_k + 1, \mu_k]$, where $k \in [M_{j-1} + 1, M_j],$
then $\widetilde{\lambda}^\prime_d = L_j$ and $\mu_d^\prime = k$.
Similarly, if $d \in [m({v}_{jk}) - s_{jk} + 1, m({v}_{jk})]$ with
$k \in [\epsi_j, \ell_j]$, then
$\lambda_d^\prime = \ell(\widetilde{v}_{jk})$, hence
$\widetilde{\lambda}_d^\prime = \ell(\widetilde{v}_{jk}) +
\varepsilon$, and $\mu_d^\prime = M_{j-1}$.  By manipulations with
Gaussian binomials analogous to those above we find that
\begin{equation*}
\prod_{k = 1}^\infty \binom{\widetilde{\lambda}_k^\prime - \mu_{k+1}^\prime}{\widetilde{\lambda}_k^\prime - \mu_k^\prime}_{q^{-1}} = \prod_{j = 1}^r \binom{L_j - M_{j-1}}{L_j - M_j}_{q^{-1}} \binom{M_j - M_{j-1}}{I_{\mu}^{(j)}}_{q^{-1}},
\end{equation*}
where $I_\mu^{(j)} = \{ k - M_{j-1} \mid k \in J_\mu \cap (M_{j-1},
M_j) \} \subset [M_j - M_{j-1} - 1]$.  Combining these observations,
we obtain
\begin{multline*}
\alpha(\lambda(\boldsymbol{\nu}), \mu; q) = \\ \prod_{j=1}^r \left( 
\binom{L_j - M_{j-1}}{L_j - M_j}_{q^{-1}} \binom{M_j - M_{j-1}}{I_{\mu}^{(j)}}_{q^{-1}} \prod_{k = M_{j-1} + 1}^{M_j} q^{k(L_j - k)r_k} \prod_{k = \epsi_j}^{\ell_j} q^{M_{j-1}(\ell(\widetilde{v}_{jk}) + \varepsilon - M_{j-1}) s_{jk}} \right).
\end{multline*}

The exponents in the remaining factor
$(q^{n}t)^{\sum_{k = 1}^{c} \mu_k} t^{\sum_{i = 1}^m \sum_{j =
    1}^{n_i} \nu^{(i)}_j}$ of the right-hand side
of~\eqref{equ:deltaw.definition} are again readily expressed as sums
of parameters $r_k$ and $s_{jk}$ using~\eqref{equ:mu.differences}
and~\eqref{equ:nu.differences}.  We leave the final assembly as an
exercise for the reader.  Summing the parameters $r_k$ and $s_{jk}$
over the ranges indicated above, we obtain
\begin{multline*}
\Delta_{w, {\rho}, \boldsymbol{v}}(q,t) = \left(\prod_{i = 1}^m
\binom{n_i}{\Rho_i}_{q_i^{-1}}\right) \prod_{j = 1}^r \left( \binom{L_j -
  M_{j-1}}{L_j - M_j}_{q^{-1}} \binom{\underline{\rho_j}}{v_j}_{\bfY}
\prod_{k = 1}^{\ell_j} \frac{y_{v_{jk}}^{(j)}}{1 - y_{v_{jk}}^{(j)}}
\right) \cdot \\ \prod_{j = 1}^{r-1} I_{M_j - M_{j - 1}}^\circ
(q^{-1}; x_{M_{j-1}+1}, \dots, x_{M_j}) \cdot I_{M_r - M_{r-1}}
(q^{-1}; x_{M_{t-1} + 1}, \dots, x_{M_t}),\label{eq:Delta}
\end{multline*}
where the numerical data $x_k$ and $y_{v_{jk}}^{(j)}$ are as given in
Definition~\ref{dfn:dwk}.  In particular, note that $y_{v_{jk}}^{(j)}$
depends only on the word $\widetilde{v}_{jk}$ and not on the weak
order $v_j$.  Summation over all $r$-tuples $\boldsymbol{v} = (v_1,
\dots, v_r) \in \prod_{j=1}^r \WO_{\underline{\rho_j}}$ now completes
the proof of Lemma~\ref{lem:most.technical.lemma}, and hence of
Proposition~\ref{thm:beef.is.here}.
\end{proof}

\section{Application to the class $\mfL$ -- proof of Theorem~\ref{thm:main.local}}\label{sec:application}
In order to deduce Theorem~\ref{thm:main.local} from the results of
the previous section, namely Theorem~\ref{thm:zeta.explicit} and
Corollary~\ref{cor:funct.eq}, it remains to show that
Hypothesis~\ref{hypothesis} is satisfied for $\mfo$-Lie algebras $L$
as in the statement of Theorem~\ref{thm:main.local}.  We noted in
Remark~\ref{rem:direct.product} that the hypothesis is stable under
direct products.  Hence it suffices to verify the hypothesis in the
case $L = \mcL(\mfO_1)\times \dots \times \mcL(\mfO_g)$, where $\mcL$
is a Lie ring from one of the three defining subclasses in
Definition~\ref{def:L} and $\mfO_i$ is a finite extension of $\lri$,
for each $i\in[g]$. It is enough to compute the $\lri$-ideal zeta
function of $L$; indeed, the $\Lri$-ideal zeta function of $L(\Lri)$
is obtained from the $\lri$-ideal zeta function of $L$ by substituting
$q^f$ for $q$, where $f$ is the inertia degree of $\Lri/\lri$.  This
verification (and more) is done in
Sections~\ref{subsec:free.nilpotent},~\ref{subsec:rel.free.prod},
and~\ref{subsec:higher.heisenberg}.  We recover, \emph{en passant},
the results of previous work by several authors.

\subsection{Abelian Lie rings}
It is instructive to consider the output of
Theorem~\ref{thm:zeta.explicit} for the basic example of the abelian
$\lri$-Lie algebra $L = \mfo^b$.  Its zeta function is well-known;
cf.\ \eqref{equ:abelian}. Let $A \leq L$ be an $\mfo$-sublattice of
rank $c$ with a torsion-free quotient $L/A \simeq \mfo^{n}$; here
$n = b-c$.  Now, let $\pnum \in \N$ and $n_i, e_i, f_i$, for
$i \in [\pnum]$, be natural numbers such that
$\sum_{i = 1}^{\pnum} n_i e_i f_i = n$, and let
$\Lri_1,\dots,\Lri_\pnum$ be arbitrary finite extensions of $\lri$
with ramification indices $e_i$ and inertia degrees $f_i$.  Then we
may express
$L / A \simeq \mfO_1^{n_1} \times \cdots \times
\mfO_{\pnum}^{n_{\pnum}}$ as in~\eqref{equ:abelianization.condition}.
Hypothesis~\ref{hypothesis} is satisfied vacuously, as $c^\prime = 0$.
Moreover, $m = 0$ in the sense of Definition~\ref{dfn:m}.  As
$\epsi = c$, it follows from Remark~\ref{rem:dyck.word.epsilon} that
the only Dyck word $w \in \mathcal{D}_{2c}$ for which
$D_w(q,t) \neq 0$ is the ``trivial'' word $w = \bfz^c \bfo^c$.  Since
the composition $\underline{n}$ is empty, the only
$(\underline{n},w)$-admissible partition is the empty one.  We then
read off from Theorem~\ref{thm:zeta.explicit} that
\begin{equation*}
\zidealo_L(s) = \zeta_{\mfo^{n}}(s)  I_c(q^{-1}; x_1, \dots, x_c),
\end{equation*}
where the
numerical data are given by
$x_k = q^{k(n+c-k)} t^k = q^{k(b-k)} t^k$. 
Indeed, it is immediate from~\eqref{equ:abelian} and~\eqref{equ:on.igusa.identity} that 
$$ I_c(q^{-1} ; x_1, \dots, x_c) = \zeta_{\mfo^c}(s-n) = \prod_{i = n}^{b-1} \frac{1}{1 - q^i t}
= \frac{\zeta_{\mfo^b}(s)}{\zeta_{\mfo^n}(s)}.$$

\subsection{Free class-$2$-nilpotent Lie
  rings} \label{subsec:free.nilpotent} Let $\mff_{2,d}$ denote the
free class-$2$-nilpotent Lie ring on $d$ generators.  If $\mfO$ is a
finite extension of $\mfo$ with ramification index $e$ and inertia
degree $f$, then the derived subalgebra of $\mff_{2,d}(\mfO)$ is
isolated and has $\lri$-rank $\binom{d}{2} ef$ and abelianization of
$\lri$-rank~$def$.  We will now implement the general framework
developed in Section~\ref{sec:main.results} to compute the
$\lri$-ideal zeta function of the direct product
$$L = \mff_{2,d_1}(\mfO_1) \times \cdots \times
\mff_{2,d_m}(\mfO_m),$$ where $d_i \in \N$ and $\mfO_i$ is a finite
extension of $\mfo$ for all $i \in [m]$.  The abelianization of
$\mff_{2,d_i}(\mfO_i)$ is isomorphic to $\mfO_i^{d_i}$ as an
$\mfo$-module.  Thus $L$
satisfies~\eqref{equ:abelianization.condition}, with $A = L^\prime =
Z(L)$ and $n_i = d_i$ for every $i \in [m]$. We set $\ol{L} = L/L'$
and let $\pi_i : \overline{L} \to \mfO_i^{d_i}$ be the projections as
in Section~\ref{subsec:projection.data}.  Let $\Lambda \leq
\overline{L}$ be a finite-index $\mfo$-sublattice and
$\nu(\pi_i(\Lambda))$ be the elementary divisor type of the
$\mfO_i$-sublattice of $\mfO_i^{d_i}$ generated by $\pi_i(\Lambda)$.
To use the method of the previous section, we must compute the
elementary divisor type of the commutator lattice $[\Lambda, L]$.

\begin{lem}\label{lem:free.lambda}
Let $L = \mff_{2,d_1}(\mfO_1) \times \cdots \times
\mff_{2,d_m}(\mfO_m)$ and let $\Lambda \leq \overline{L}$ be an
$\mfo$-sublattice.  For every $i \in [m]$, let $\nu^{(i)} =
\nu(\pi_i(\Lambda)) = (\nu^{(i)}_{1}, \dots, \nu^{(i)}_{d_i})$. Then
the $\lri$-elementary divisor type $\lambda(\Lambda)$ of the
commutator $[\Lambda,L]\leq L'$ is obtained from the following
multiset with $c = \sum_{i = 1}^m \binom{d_i}{2} e_i f_i$
elements:
$$ \coprod_{i = 1}^m \coprod_{1 \leq j < k \leq d_i} \{ \min \{
\nu^{(i)}_{j}, \nu^{(i)}_{k} \} \}_{e_i, f_i}.$$
\end{lem}

\begin{proof}
  Let $(x^{(i)}_1, \dots, x^{(i)}_{d_i})$ be an $\mfO_i$-basis of
  $\mff_{2,d_i}(\mfO_i)$ with respect to which $\pi_i(\Lambda)$ is
  diagonal:
$$ \pi_i(\Lambda) = \langle \Pi_i^{\nu^{(i)}_1} x^{(i)}_1, \dots, \Pi_i^{\nu^{(i)}_{d_i}} x^{(i)}_{d_i} \rangle_{\mfO_i},$$
where $\Pi_i \in \mfO_i$ is a uniformizer.  Observe that the
collection of commutators
\begin{equation*}
\left\{ [x^{(i)}_j, x^{(i)}_k] \right\}_{1 \leq j < k \leq d_i}
\end{equation*}
provides an $\Lri_i$-basis of the derived subalgebra of
$\mff_{2,d_i}(\mfO_i)$.  Clearly, the commutator subalgebra
$[\pi_i(\Lambda), \pi_i(L)]$ is the $\mfO_i$-lattice spanned by the
elements $ \{ \Pi_i^{\nu^{(i)}_j} [x_j^{(i)}, x_k^{(i)}] \}_{j \neq
  k}$.  The $\mfO_i$-elementary divisor type of this lattice is the
partition with parts $\min \{ \nu^{(i)}_j, \nu^{(i)}_k \}$, as
observed already just before~\cite[Lemma~5.2]{GSS/88}.  The elementary
divisor type of the same object, viewed as a lattice over $\mfo$, is
given by the multiset
$$ \coprod_{1 \leq j < k \leq d_i} \{ \min \{ \nu^{(i)}_{j},
\nu^{(i)}_{k} \} \}_{e_i, f_i} $$ by Lemma~\ref{lem:oedt.to.zp}.  To
complete the proof, we observe that the direct product structure of
$L$ implies that $[\Lambda, L] = \bigoplus_{i = 1}^m [\pi_i(\Lambda),
  \pi_i(L)]$.
\end{proof}

\begin{rem} \label{rem:f2d.unramified}
Observe that $\{ \nu \}_{1,f}$ is simply the multiset consisting of
the element $\nu$ with multiplicity $f$.  Therefore, if the extensions
$\mfO_i / \mfo$ are all unramified (i.e.~$e_i = 1$ for all $i$) then
it is immediate from Lemma~\ref{lem:free.lambda} that $L$ satisfies
Hypothesis~\ref{hypothesis}.  Indeed, we may set $\klim = \sum_{i =
  1}^m f_i$ and let the collection $\widetilde{\mathfrak{S}}_1, \dots,
\widetilde{\mathfrak{S}}_{\klim}$ consist of $f_i$ copies of the pair
$(\{i\}, 2)$ for every $i \in [m]$.  Moreover, our decomposition of
$L/A$ satisfies the conditions of Remark~\ref{rem:ram}.  Therefore,
Hypothesis~\ref{hypothesis} necessarily fails if any of the extensions
$\mfO_i / \mfo$ are ramified, and the method of
Section~\ref{sec:main.results} is inapplicable.  \emph{We therefore
  assume for the remainder of Section~\ref{subsec:free.nilpotent} that
  all the $\mfO_i$ are unramified over $\mfo$.}
\end{rem}

As at the beginning of Section~\ref{sec:statement}, the possible
orderings of the projection data $\bsnu = (\nu^{(1)}, \dots,
\nu^{(m)})$ are parametrized by the the chain complex
$\WO_{\underline{n}}$ of $C_{\underline{n}}$.  Recall the function
$\ell(v)$ of Definition~\ref{dfn:rad}.

\begin{lem} \label{lem:f2d.length}
Let $v = \prod_{i = 1}^m a_i^{\alpha_i} \in C_{\underline{n}}$.  Then
$\ell(v) = \sum_{i = 1}^m \binom{\alpha_i}{2} f_i$.
\end{lem}
\begin{proof}
Let $i \in [m]$.  There are exactly $\alpha_i$ parts of the partition
$\nu(\pi_i(\Lambda))$ that are not less than~$m(v)$, and hence there
are $\binom{\alpha_i}{2}$ pairwise minima that are not less than
$m(v)$.  Each of these minima appears in ${\lambda}(\Lambda)$ with
multiplicity $f_i$.  Alternatively, apply
Lemma~\ref{lem:length.function} and the description of the
sets $\widetilde{\mathfrak{S}}_1, \dots,
\widetilde{\mathfrak{S}}_{\klim}$ given in Remark~\ref{rem:f2d.unramified} above.
\end{proof}

We now have all the ingredients necessary to apply
Definition~\ref{dfn:dwk} and Theorem~\ref{thm:zeta.explicit} to obtain
an explicit expression for $\zeta_L^{\vartriangleleft\,\lri} (s)$.

\begin{exm}
  We recover an expression for the $\Zp$-ideal zeta function of
  $\mff_{2,d}(\Z_p)$, where $d \geq 2$, which was computed by the
  third author in~\cite{Voll/05a}.  The expressions of
  Theorem~\ref{thm:zeta.explicit} reduce to a particularly simple form
  in this case.  Here $m = 1$ and $\mfo = \Z_p$, and, given a
  $\Z_p$-sublattice $\Lambda \leq \overline{L}$, there is only one
  relevant projection datum, namely the elementary divisor type
  ${\nu} = (\nu_1, \dots, \nu_d)$ of $\Lambda$ itself.  The derived
  subalgebra has rank $c = \binom{d}{2}$.  In view of
  Lemma~\ref{lem:f2d.length}, the parts of the dual partition
  $\lambda(\Lambda)^\prime = {\lambda}({\nu})^\prime$ are all
  triangular numbers.  In particular, if
  $w \in \mathcal{D}_{2c} = \mathcal{D}_{d(d-1)}$ is a Dyck word, then
  $D_{w}(p,t) = 0$ unless all the parameters $L_1, \dots, L_r$
  associated to $w$ are triangular numbers.

So suppose that $w \in \mathcal{D}_{d(d-1)}$ is such that $L_j =
\binom{\gamma_j}{2}$ for all $j \in [r]$.  It is easy to see from
Definition~\ref{dfn:adm.comp} that there is only one
$(d,w)$-admissible composition, namely $\rho_{1j} = \gamma_j -
\gamma_{j-1}$ for all $j \in [r]$ (where we have set $\gamma_0 = 0$).
Thus $\Rho_{1j} = \gamma_j$ for all $j$.  Noting from
Example~\ref{exm:gen.igusa} that the generalized Igusa function
associated to a composition with one part is a classical Igusa
function in the sense of Definition~\ref{def:igusa}, we read off from
Definition~\ref{dfn:dwk} that
\begin{multline*}
D_w (p,t) = \prod_{j = 1}^r \left( \binom{L_j - M_{j-1}}{L_j -
  M_j}_{p^{-1}} \binom{d}{\gamma_j}_{p^{-1}} I_{\gamma_j -
  \gamma_{j-1}} (p^{-1}; y_1^{(j)}, \dots, y_{\gamma_j -
  \gamma_{j-1}}^{(j)}) \right) \cdot
\\ \prod_{j=1}^{r-1}I_{M_j-M_{j-1}}^\circ (p^{-1};x_{M_{j-1}+1},
\dots, x_{M_j}) \cdot I_{M_r-M_{r-1}}(p^{-1};x_{M_{r-1}+1}, \dots,
x_{M_r}),
\end{multline*}
where
\begin{align*}
y_k^{(j)} & = p^{M_{j-1}(d + \binom{\gamma_{j-1} + k}{2} - M_{j-1}) +
  (\gamma_{j-1} + k)(d - \gamma_{j-1} - k)} t^{\gamma_{j-1} + k +
  M_{j-1}}, \\ x_k & = p^{k(d + \binom{\gamma_j}{2} - k) + \gamma_j(d -
  \gamma_j)} t^{k + \gamma_j}.
\end{align*}
Here, as usual, we have $k \in [M_{j-1} + 1, M_j]$ in the definition
of $x_k$.  Indeed, observe that the only instance of two distinct
subwords $v_1, v_2 \leq a_1^{d}$ satisfying $\ell(v_1) = \ell(v_2)$ is
$\ell(\varnothing) = \ell(a_1) = 0$.  Thus we always have
$\delta_v^{(j)} = 1$ except in the case $\delta_{a_1}^{(1)} = 0$, but
it is easy to verify that the uniform expressions given above for the
numerical data hold.  Finally, by 
Theorem~\ref{thm:zeta.explicit}, 
$$ \zeta^\vartriangleleft_{\mff_{2,d}(\Z_p)} (s) = \sum_{w \in
  \mathcal{D}_{d(d-1)}} D_{w} (p,t).$$ We leave it as an exercise for
the reader to unwind the definitions of~\cite{Voll/05a} and verify
that this formula matches~\cite[Theorem~4]{Voll/05a}.
\end{exm}

\subsection{Free class-$2$-nilpotent products of abelian Lie
  rings} \label{subsec:rel.free.prod} Let $L_1$ and $L_2$ be abelian Lie rings of ranks
$d$ and $d^\prime$, respectively.  We denote by
$\mathfrak{g}_{d,d^\prime}$ the free class-$2$-nilpotent product of
$L_1$ and $L_2$ of nilpotency class at most two.  This is the Lie ring
version of a group-theoretical construction considered by
Levi~\cite{Levi/44} (see also~\cite{Golovin/50}), which is itself a
special case of a varietal product as
in~\cite[Section~1.8]{Neumann/67}.  Concretely, a presentation of
$\mathfrak{g}_{d,d'}$ is given by
$$\mathfrak{g}_{d, d^\prime} = \la x_1,\dots,x_d, y_1,\dots,y_{d'},
\left(z_{ij}\right)_{i\in[d], j\in[d^\prime]} \mid [x_i,y_j]=z_{ij}\ra ,$$
where all Lie brackets not following from the relations above vanish.
\begin{exm}
\ \begin{enumerate}
\item $\mfg_{1,1}$ is the Heisenberg Lie ring.
\item $\mfg_{d,1}$ is the Grenham Lie ring of degree $d$.
\item $\mfg_{d,0} = \Z^d$ is the abelian Lie ring of rank $d$.
\item $\mfg_{d,d} = \mathcal{G}_d$ is the Lie ring featuring in
  \cite[Definition~1.2]{StasinskiVoll/14}.
\end{enumerate}
\end{exm}

We fix $g \in \N$ and $g$-tuples $\underline{d} = (d_1, \dots, d_g)$
and $\underline{d}^\prime = (d^\prime_1, \dots, d^\prime_g)$ of
natural numbers.  Let $\mfO_1, \dots, \mfO_g$ be finite extensions of
$\mfo$ with ramification indices $e_i$ and inertia degrees $f_i$,
respectively.  Consider the $\lri$-Lie algebra
\begin{equation*}
L = \g_{d_1, d^\prime_1}(\mfO_1) \times \cdots \times \g_{d_g, d^\prime_g} (\mfO_g).
\end{equation*}

Define $d = \sum_{i = 1}^g d_i e_i f_i$ and $d^\prime = \sum_{i = 1}^g
d^\prime_i e_i f_i$, and set $c = \sum_{i = 1}^g d_i d^\prime_i e_i
f_i$.  Observe that, as an $\lri$-module, $L$ is free of rank $d +
d^\prime + c$.  Let $L^\prime$ denote the derived subalgebra of $L$,
and let
$$\overline{L} = L / L^\prime \simeq (\Lri_1^{d_1}\times
\Lri_1^{d^\prime_1})\times (\Lri_2^{d_2}\times \Lri_2^{d^\prime_2})\times \dots
\times (\Lri_g^{d_g}\times \Lri_g^{d^\prime_g})$$
be its abelianization.  For
each $i \in [g]$, consider the usual basis
$ \left\{ x_k^{(i)}, y_\ell^{(i)}, z_{k \ell}^{(i)} \right\}_{k \in
  [d_i] \atop \ell \in [d^\prime_i]}$ of $\g_{d_i, d^\prime_i}(\Lri_i)$ as an
$\Lri_i$-module.  Consider the natural linear projections
\begin{align*}
\pi_i : \overline{L} & \to \langle x_1^{(i)}, \dots, x_{d_i}^{(i)}
\rangle_{\Lri_i} \simeq \Lri_i^{d_i} \\ \pi^\prime_{i} : \overline{L} & \to
\langle y_1^{(i)}, \dots, y_{d^\prime_i}^{(i)} \rangle_{\Lri_i} \simeq
\Lri_i^{d^\prime_i}.
\end{align*}

For each $i \in [g]$, fix an $\lri$-basis $(\alpha_1^{(i)}, \dots,
\alpha_{e_i f_i}^{(i)})$ of $\Lri_i$.  Then $ \left\{ \alpha_j^{(i)}
x_k^{(i)}, \alpha_j^{(i)} y_\ell^{(i)}, \alpha_j^{(i)} z_{k
  \ell}^{(i)} \right\}_{k \in [d_i], \ell \in [d^\prime_i] \atop j \in [e_i
    f_i]}$ is an $\lri$-basis of $\g_{d_i, d^\prime_i}(\Lri_i)$ and the union
of these bases is an $\lri$-basis of $L$.

Let $\Lambda \leq \overline{L}$ be an $\lri$-sublattice.  For each $i
\in [g]$, we let $\nu^{(i)}$, a partition with $d_i$ parts, be the
elementary divisor type of the $\Lri_i$-sublattice of $\Lri_i^{d_i}$
generated by $\pi_i(\Lambda)$.  Similarly, we set $\nu^{(i+g)}$ to be
the elementary divisor type of the $\Lri_i$-sublattice of
$\Lri_i^{d^\prime_i}$ generated by $\pi^\prime_{i}(\Lambda)$. In other words,
\begin{equation}\label{equ:nu.grenham}
  \bsnu = \bsnu(\Lambda) =
  (\nu^{(1)},\nu^{(1+g)},\nu^{(2)},\nu^{(2+g)},\dots,\nu^{(g)},\nu^{(2g)})
  \end{equation}
  is the projection data of $\Lambda$ as an $\lri$-sublattice of
  $\overline{L}$.

  \begin{lem} \label{lem:commutator.edt} Let
    $L = \g_{d_1, d^\prime_1}(\mfO_1) \times \cdots \times \g_{d_g,
      d^\prime_g} (\mfO_g)$ and let $\Lambda \leq \overline{L}$ be an
    $\lri$-sublattice. Let $\bsnu(\Lambda)$ be as in
    \eqref{equ:nu.grenham} above. Then the $\lri$-elementary divisor
    type $\lambda(\Lambda)$ of the commutator $[\Lambda,L] \leq L'$ is
    obtained from the following multiset with $c = \sum_{i=1}^g d_i
    d_i' e_if_i$ elements:
 $$ \coprod_{i = 1}^g \coprod_{k = 1}^{d_i d^\prime_i} \{ ({\nu}^{(i)} \ast
 {\nu}^{({i+g})})_k \}_{e_i, f_i},$$ where the operation $\ast$ is
 explained in Definition~\ref{def:partition.ast} and the sets $\{ a \}_{e_i,
   f_i}$, for $a\in\N$, are as in Definition~\ref{dfn:dictionary}.
\end{lem}
\begin{proof}
  For every $i\in[g]$, let $\Pi_i$ denote a uniformizer of
  $\Lri_i$. Let $(\xi_1^{(i)}, \dots, \xi_{d_i}^{(i)})$ and
  $(\upsilon_1^{(i)}, \dots, \upsilon_{d_i^\prime}^{(i)})$ be bases of
  $\Lri_i^{d_i}$ and $\Lri_i^{d_i^\prime}$, respectively, such that
\begin{eqnarray*}
\langle \pi_i(\Lambda) \rangle_{\Lri_i} & = & \langle \Pi_i^{\nu^{(i)}_1} \xi_1^{(i)}, \dots, \Pi_i^{\nu^{(i)}_{d_i}} \xi_{d_i}^{(i)} \rangle_{\Lri_i} \\  
\langle \pi^\prime_i(\Lambda) \rangle_{\Lri_i} & = & \langle \Pi_i^{\nu^{(i+g)}_1} \upsilon_1^{(i)}, \dots, \Pi_i^{\nu^{(i+g)}_{d^\prime_i}} \upsilon_{d^\prime_i}^{(i)} \rangle_{\Lri_i}.
\end{eqnarray*}
Observe that the commutators $[\xi^{(i)}_k, \upsilon_\ell^{(i)}]$ form
an $\Lri_i$-basis of the subspace
$\langle z_{k \ell}^{(i)} \rangle_{\Lri_i}$ of $L^\prime$.  Fixing
$k \in [d_i]$, we find that
\begin{equation*}
[\Pi_i^{\nu_k^{(i)}} \xi_k^{(i)}, \overline{L}]  =  \bigoplus_{\ell \in [d_i^\prime]} \Pi_i^{\nu_k^{(i)}} \Lri_i [\xi_k^{(i)}, y^{(i)}_\ell] = \bigoplus_{\ell \in [d_i^\prime]} \Pi_i^{\nu_k^{(i)}} \Lri_i [\xi_k^{(i)}, \upsilon^{(i)}_\ell].
\end{equation*}
Similarly, for a fixed $\ell \in [d_i^\prime]$ we obtain
\begin{equation*}
[\Pi_i^{\nu_{\ell}^{(i+g)}} \upsilon_\ell^{(i)}, \overline{L}]  =  \bigoplus_{k \in [d_i]} \Pi_i^{\nu_{\ell}^{(i+g)}} \Lri_i [x_k^{(i)}, \upsilon^{(i)}_\ell] = \bigoplus_{k \in [d_i]} \Pi_i^{\nu_{\ell}^{(i+g)}} \Lri_i [\xi_k^{(i)}, \upsilon^{(i)}_\ell].
\end{equation*}
From this we conclude that
\begin{equation*}
[\ol{\mfg_{d_i,d^\prime_i}(\Lri_i)},\Lambda] = \bigoplus_{k\in[d_i], \ell \in[d^\prime_i]} \Pi_i^{\min\{\nu^{(i)}_{k},\nu^{(i+g)}_{\ell}\}} \Lri_i [\xi_k^{(i)}, \upsilon^{(i)}_\ell] = \bigoplus_{k\in[d_i], \ell \in[d^\prime_i]} \Pi_i^{\min\{\nu^{(i)}_{k},\nu^{(i+g)}_{\ell}\}} \Lri_i z_{k \ell}^{(i)}
\end{equation*}
as $\Lri_i$-modules, where $\ol{\mfg_{d_i,d^\prime_i}(\Lri_i)}$ is the
abelianization of $\mfg_{d_i,d^\prime_i}(\Lri_i)$.  Therefore,
\begin{equation*}
[\ol{L}, \Lambda] = \bigoplus_{i\in [g], k\in[d_i], \ell \in[d^\prime_i]}
  \Pi_i^{\min\{\nu^{(i)}_{k},\nu^{(i+g)}_{\ell}\}} \Lri_i z_{k \ell}^{(i)}
\end{equation*}
as $\lri$-modules. The claim follows.
\end{proof}

Set $m = 2g$.  For $i \in [g]$, set $\Lri_{i+g} = \Lri_i$ and define
$n_i = d_i$ and $n_{i+g} = d^\prime_i$.  It is clear from
Lemma~\ref{lem:commutator.edt} that the Lie ring $L$ fits the general
framework of the beginning of Section~\ref{subsec:rewriting}.
Moreover, we see analogously to Remark~\ref{rem:f2d.unramified} that
if all the $\Lri_i$ are unramified over $\lri$, then
Hypothesis~\ref{hypothesis} is satisfied.  In this case, we take
$\klim = \sum_{i = 1}^g f_i$; the collection
$\widetilde{\mathfrak{S}}_1, \dots, \widetilde{\mathfrak{S}}_{\klim}$
consists of $f_i$ copies of the pair $( ( i, i+g ), (1,1))$ for every
$i \in [g]$. \emph{Thus we assume for the remainder of this section
  that {{all the $\Lri_i$ are unramified} over $\lri$}.}

Consider the composition $\underline{n} = (n_1, \dots, n_{2g})$.  Then
the natural ordering among all the parts of the projection data
$\boldsymbol{\nu} = ({\nu}^{(1)}, \dots, {\nu}^{(2g)})$ corresponds to
an element of $\WO_{\underline{n}}$.

\begin{lem} \label{lem:lambda.dual}
 Let $v = \prod_{i = 1}^{2g} a_i^{\alpha_i} \in C_{\underline{n}}$.
 Then $\ell(v) = \sum_{i = 1}^g \alpha_i \alpha_{i+g} f_i$.
\end{lem}
\begin{proof}
Let $v \in C_{\underline{n}}$ as above.  For any $i \in [g]$, the $d_i
d^\prime_i$ parts of ${\nu}^{(i)} \ast {\nu}^{(i+g)}$ are, by
definition, the minima $\min \{ \nu^{(i)}_{k}, \nu^{(i+g)}_\ell \}_{k
  \in [d_i], \ell \in [d^\prime_i]}$.  Clearly, $\min \{
\nu^{(i)}_{k}, \nu^{(i+g)}_\ell \} \geq m(v)$ if and only if both
elements of the pair $(\nu^{(i)}_{k}, \nu^{(i+g)}_\ell)$ are contained
in $S_v$, and it is clear from~\eqref{equ:sr.def} that there are
$\alpha_i \alpha_{i+g}$ such pairs.  Finally, since we have assumed
all $\Lri_i / \lri$ to be unramified, every part of ${\nu}^{(i)} \ast
{\nu}^{(i+g)}$ appears in ${\lambda}(\boldsymbol{\nu})$ with
multiplicity $f_i$.  Alternatively, use
Lemma~\ref{lem:length.function}.
\end{proof}

The $\lri$-ideal zeta function $\zidealo_L(s)$ may now be read off from
Theorem~\ref{thm:zeta.explicit}.  

\subsubsection{Grenham Lie rings over unramified extensions}
As an example, we will treat the case $L = \g_{d,1}(\Lri)$, where
$\g_{d,1}$ is the Grenham Lie ring of degree $d$ and $\Lri / \lri$ is
unramified of degree~$f$.  In the case $d = f = 2$, this zeta function
was computed previously by Bauer, using methods analogous to those
of~\cite{Voll/05} and quite different from the current paper's
approach.

Observe that $L^\prime = Z(L)$, so necessarily we have $A = L^\prime$
and thus $c = c^\prime = df$ and $\varepsilon = 0$ in the notation of
Section~\ref{subsec:rewriting}.  The non-empty radical words $v \in
C_{(d,1)}$ are exactly those of the form $v = a_1^{\alpha_1} a_2$ with
$\alpha_1 > 0$.  If $w \in \mathcal{D}_{2c}$ is a Dyck word with
associated parameters $L_1, \dots, L_r$ and $M_1, \dots, M_r$, then
clearly there are no $((d,1),w)$-admissible compositions (recall
Definition~\ref{dfn:adm.comp}) unless all the $L_i$ are divisible by
$f$.  Otherwise, there is a unique $((d,1),w)$-admissible composition
$\rho \in \mathrm{Mat}_{2,r}$; it satisfies $\Rho_{1j} = L_j / f$ and
$\Rho_{2j} = 1$ for all $j \in [r]$.  Equivalently, $\rho_{1j} = (L_j
- L_{j-1})/f$ for all $j \in [r]$, while $\rho_{21} = 1$ and
$\rho_{2j} = 0$ for all $j > 1$.

Let $\mathcal{D}_{2c}(f)$ be the set of Dyck words $w \in
\mathcal{D}_{2c}$ such that $f | L_i$ for all $i \in [r]$.  Given $w
\in \mathcal{D}_{2c}(f)$, set $\boldsymbol{L}_w/f = \{ L_i / f \mid i
\in [r-1] \}$.  The following explicit statement is now immediate from
Theorem~\ref{thm:zeta.explicit}.

\begin{pro} \label{pro:grenham.explicit}
Let $L = \g_{d,1}(\Lri)$, where $\Lri / \lri$ is an unramified
extension of degree~$f$.  Then
\begin{equation*}
\zidealo_L(s) = \frac{\zeta_{\lri^{(d+1)f}}(s)}{\zeta_{\Lri}(s) \zeta_{\Lri^d}(s)} \sum_{w \in \mathcal{D}_{2c}(f)} D_w(q,t),
\end{equation*}
where
\begin{multline*}
D_w(q,t) = \binom{d}{\boldsymbol{L}_w / f}_{q^{-f}} \prod_{j = 1}^r \binom{L_j - M_{j-1}}{L_j - M_j}_{q^{-1}} I^{\mathrm{wo}}_{(L_1 / f, 1)}(q^{-f}, q^{-f}; \boldsymbol{y}^{(1)}) \cdot \\ \prod_{j = 2}^{r} I_{(L_j - L_{j-1})/f}(q^{-f}; y_{(L_{j-1}/f) + 1}, \dots, y_{L_j / f}) \cdot \\
\prod_{j = 1}^{r-1} I^{\circ}_{M_j - M_{j-1}}(q^{-1}; x_{M_{j-1} + 1}, \dots, x_{M_j}) I_{M_r - M_{r-1}}(q^{-1}; x_{M_{r-1} + 1}, \dots, x_{M_r}).
\end{multline*}
Here the numerical data are given by
\begin{align*}
x_k & = q^{k((d+1)f + L_j - k) + L_j(d - L_j / f)} t^{k + f + L_j}, \, k \in [M_{j-1} + 1, M_j], \\
y^{(1)}_{a_1^{\alpha_1} a_2^{\alpha_2}} & = q^{f \alpha_1 (d - \alpha_1)} t^{f(\alpha_1 + \alpha_2)}, \\
y_{k} & = q^{M_{j-1}((d+k+ k(d-k) + 1)f - M_{j-1})} t^{f(k +1) + M_{j-1}}, \, k \in [(L_{j-1}/f) + 1, L_j / f].
\end{align*}
\end{pro}
\begin{rem} \label{rmk:grenham.explicit} Using
  Proposition~\ref{pro:grenham.explicit} to compute
  $\zeta^{\ideal}_{\mathfrak{g}_{d,1}(\Z_p)}(s)$ produces a sum
  parametrized by the $\frac{1}{d+1} \binom{2d}{d}$ elements of
  $\mathcal{D}_{2d}$.  Yet~\cite[Theorem~5]{Voll/05}, translated to
  the notation of the present paper, gives the much simpler expression
$$ \zeta^{\ideal}_{\mathfrak{g}_{d,1}(\Z_p)}(s) = \zeta_{\Z_p^{d+1}}(s) I_{d} (p^{-1}; z_1, \dots, z_d),$$
where $z_i = p^{i(2d + 1 - i)} t^{2i + 1}$ for $i \in [d]$.  We have
checked that these expressions coincide for $d \leq 3$, but a direct
proof of their equality would involve proving an identity of
generalized Igusa functions with conditions on the numerical data, in
the spirit of Proposition~\ref{pro:igusas.match}; see also
Remark~\ref{rmk:free.products.match} below.  This example shows that
expressions derived from Theorem~\ref{thm:zeta.explicit} sometimes
admit dramatic cancellation.
\end{rem}

\subsubsection{The Lie ring $\g_{2,2}$}\label{subsec:g22}
Paajanen~\cite[Theorem~11.1]{Paajanen/08} computed the ideal zeta
function of the $\lri$-Lie algebra $L = \g_{2,2}(\mfo)$.  We recover
this computation as a special case of our results.  By
Theorem~\ref{thm:zeta.explicit} we have
\begin{equation*}
\zidealo_{L} (s) = \frac{\zeta_{\lri^4}(s)}{(\zeta_{\lri^2}(s))^2}
\sum_{w \in \mathcal{D}_8} \sum_{{\rho} \in
  \mathcal{M}_{(2,2),w}} D_{w, {\rho}} (q,t) =
\frac{(1-t)(1-qt)}{(1 - q^2t)(1 - q^3t)} \sum_{w \in \mathcal{D}_8}
\sum_{{\rho} \in \mathcal{M}_{(2,2),w}} D_{w,
  {\rho}} (q,t).
\end{equation*}

There are fourteen Dyck words of length $8$, but it is easy to check
that there are only five Dyck words $w \in \mathcal{D}_8$ for which
there exist $w$-compatible flags of subwords of the word $a_1^2a_2^2$.
For simplicity, for the rest of this example we will write $a$ instead
of $a_1$ and $b$ instead of $a_2$.  We tabulate these Dyck words,
together with the associated functions $D_{w, {\rho}} (q,t)$.  Observe
that there are three Dyck words with two compatible flags, and that in
each of these cases both flags give rise to the same function $D_{w,
  {\rho}} (q,t)$.  This is a consequence of the symmetries of $L =
\g_{2,2}(\mfo)$ and is not a general phenomenon.

For brevity,
we use the notation $\gpf{x} = \frac{x}{1-x}$ and $\gpzero{x} =
\frac{1}{1-x}$.

\begin{center}
\renewcommand{\arraystretch}{1.2}{
\begin{longtable}{|c | l | l |}  
\hline
Dyck word & Flag & $D_{w, {\rho}} (q,t)$ \\
\hline\hline
$\bfz \bfz \bfz \bfz \bfo \bfo \bfo \bfo$ & $a^2b^2$ & $I^{\mathrm{wo}}_{(2,2)} (q^{-1}; \mathbf{y}) I_4(q^{-1}; q^7 t^5, q^{12} t^6, q^{15} t^7, q^{16} t^8)$ \\
\hline
\multirow{2}{*}{$\bfz \bfz \bfo \bfz \bfz \bfo \bfo \bfo$} & $a^2b < a^2b^2$ & \multirow{2}{*}{$\binom{2}{1}_{q^{-1}}^2 I^{\mathrm{wo}}_{(2,1)} (q^{-1}; \mathbf{y}) \gpf{q^6 t^4} \gpzero{q^7t^5} I_3(q^{-1};q^{12}t^6, q^{15}t^7, q^{16}t^8)$}\\
\cline{2-2}
& $ab^2 < a^2b^2$ & \\
\hline
\multirow{2}{*}{$\bfz \bfz \bfo \bfo \bfz \bfz \bfo \bfo$} & $a^2b < a^2b^2$ & 
\multirow{2}{*}{{$ \binom{2}{1}_{q^{-1}} I^{\mathrm{wo}}_{(2,1)} (q^{-1}; \mathbf{y}) I_2^\circ(q^{-1}; q^6 t^4, q^9 t^5) \gpzero{q^{12}t^6} I_2(q^{-1}; q^{15} t^7, q^{16} t^8)$}}
\\
\cline{2-2}
& $ab^2 < a^2b^2$ & \\
\hline
$\bfz \bfo \bfz \bfz \bfz \bfo \bfo \bfo$ & $ab < a^2b^2$ & 
{$\binom{2}{1}_{q^{-1}}^2 I^{\mathrm{wo}}_{(1,1)}(q^{-1}, \mathbf{y}) \gpf{q^6 t^3} I^{\mathrm{wo}}_{(1,1)} (q^{-1}; \mathbf{z}) I_3(q^{-1}; q^{12}t^6, q^{15}t^7, q^{16}t^8)$}
\\
\hline
\multirow{2}{*}{$\bfz \bfo \bfz \bfo \bfz \bfz \bfo \bfo$} & \tiny{$ab < a^2b < a^2b^2$} & 
\multirow{2}{*}{\makecell{$\binom{2}{1}_{q^{-1}}^2 I^{\mathrm{wo}}_{(1,1)} (q^{-1}; \mathbf{y}) \gpf{q^6 t^3} \gpzero{q^6 t^4} \gpf{q^9 t^5}\gpzero{q^{12}t^6} \cdot$ \\ $ I_2(q^{-1}; q^{15}t^7, q^{16}t^8)$}}
\\
\cline{2-2}
& \tiny{$ab < ab^2 < a^2b^2$} & \\
\hline
    \end{longtable} }
\end{center}
Here the numerical data $\mathbf{y}$ and $\mathbf{z}$ are defined as follows:
$$
\begin{array}{lllll}
y_a = y_b = qt & y_{a^2} = y_{b^2} = t^2 & y_{ab} = q^2 t^2 &
y_{a^2b} = y_{ab^2} = qt^3 & y_{a^2 b^2} = t^4,  \\
z_a = z_b = q^6 t^4 & z_{ab} = q^7 t^5. & & 
\end{array}
$$

\subsubsection{The Heisenberg Lie ring} \label{subsec:heisenberg}
The relatively free product $\g_{1,1}$ is the Heisenberg Lie
ring~$\mathfrak{h}$.  This ring is spanned over $\Z$ by three
generators $x,y,z$, with the relations $[x,y] = z$,
$[x,z]=[y,z]=0$. It is among the  smallest non-abelian nilpotent Lie rings.
It was studied by two of the authors in~\cite{SV1/15}, in the case
$\lri = \Z_p$; the zeta functions computed there can be recovered as
special cases of the analysis in this section. Indeed, consider
$$L = \mathfrak{h}(\Lri_1) \times \cdots \times
\mathfrak{h}(\Lri_g),$$ where the $\Lri_i$ are unramified over $\lri$
so that Hypothesis~\ref{hypothesis} holds.  Then $c = \sum_{i = 1}^g
f_i$, while $n = 2c$.  Note that the quantity denoted $n$
in~\cite{SV1/15} is called $c$ in the current paper. 
The composition $\underline{n}$ defined just before the statement of
Lemma~\ref{lem:lambda.dual} is $\underline{n} = (1, 1, \dots, 1)$,
with $2g$ parts. Thus the elements of $C_{\underline{n}}$ correspond
to subwords of the word $a_1 \cdots a_{2g}$.  The radical subwords are
the words of the form $\prod_{i \in J} a_i a_{i+g}$ for some $J
\subseteq [g]$.  Thus radical subwords are in bijection with subsets
of $[g]$.  Moreover, if $w \in \mathcal{D}_{2c}$ is a Dyck word, then
a $w$-compatible flag $V = (v_1 < \cdots < v_r) \in \mathcal{F}_w$
corresponds to a sequence of subsets $J_1 \subset \cdots \subset
J_r = [g]$ such that $\sum_{i \in J_j \setminus J_{j-1}} f_i = L_j -
L_{j-1}$ for all $j \in [r]$.  Setting $\mathcal{A}_j = J_j \setminus
J_{j-1}$, we obtain precisely the set partitions of $[g]$ that are
compatible with $w$, in the sense of~\cite[Definition~3.4]{SV1/15}.
Recall that the set of set partitions compatible with $w$ was denoted
$\mathcal{P}_w$ in~\cite{SV1/15}.

We see from Theorem~\ref{thm:zeta.explicit}, applied to $L =
\mathfrak{g}_{1,1}(\Lri_1) \times \cdots \times \mathfrak{g}_{1,1}(\Lri_g)$, that
\begin{equation*}
\zidealo_L (s) = \frac{\zeta_{\lri^{2c}}(s)}{\prod_{i =
    1}^g \zeta_{\Lri_i}(s)^2} \sum_{w \in \mathcal{D}_{2c}}
\sum_{{\rho} \in \mathcal{M}_{\underline{n}, w}} D_{w,
  {\rho}}(q,t) = \zeta_{\lri^{2c}}(s) \left( \prod_{i = 1}^g
(1 - t^{f_i})^2 \right) \sum_{w \in \mathcal{D}_{2c} \atop
  {\rho} \in \mathcal{M}_{\underline{n}, w}} D_{w,
  {\rho}}(q,t).
\end{equation*}
Now set $\lri = \Z_p$; in particular, $q = p$.
A comparison with \cite[eq.~(2.20)]{SV1/15} and the displayed equation
immediately before \cite[Theorem~3.6]{SV1/15} shows that, to recover
the results obtained there, it suffices to prove that if
${\rho} \in \mathcal{M}_{\underline{n}, w}$ is associated to
a set partition $\{ \mathcal{A}_j \}_{j \in [r]} \in \mathcal{P}_w$,
then
\begin{equation} \label{equ:suffices.for.heisenberg}
\left( \prod_{i = 1}^g (1 - t^{f_i})^2 \right) D_{w, {\rho}}(p,t) = \left( \prod_{i = 1}^g (1 - t^{2f_i}) \right) D^{\mathbf{f}}_{w, \mathcal{A}}(p,t),
\end{equation}
where $D^{\mathbf{f}}_{w, \mathcal{A}}(p,t)$ is defined
by~\cite[(3.12)]{SV1/15}.

We read off from Definition~\ref{dfn:dwk} that,
for ${\rho} \in \mathcal{M}_{\underline{n}, w}$,
\begin{multline} \label{equ:heisenberg.case}
D_{w, {\rho}}(p,t) = \prod_{j = 1}^r \left( \binom{L_j - M_{j-1}}{L_j
  - M_j}_{p^{-1}} I_{\prod_{k \in \mathcal{A}_j} a_k
  a_{k+g}}^{\mathrm{wo}} (\mathbf{y}^{(j)}) \right) \cdot
\\ \left(\prod_{j = 1}^{r-1} I^\circ_{M_j - M_{j-1}}(p^{-1};
x_{M_{j-1} + 1}, \dots, x_{M_j}) \right)I_{M_r - M_{r-1}}(p^{-1};
x_{M_{r-1}+1}, \dots, x_{M_r}),
\end{multline}
with the numerical data specified there.  Since the parameters
$q_i^{-1}$ do not actually appear in the relevant generalized Igusa
functions, we have omitted them from the notation (just as in
Proposition~\ref{pro:igusas.match}).  Observe that the numerical data
$x_k$ in~\eqref{equ:heisenberg.case} match those in the formula for
$D^{\mathbf{f}}_{w, \mathcal{A}}(p,t)$ given
in~\cite[Theorem~3.6]{SV1/15}.  Moreover, if $r_{\mathcal I} =
\prod_{k \in \mathcal{I}} a_k a_{k+g}$ is a radical subword of
$\prod_{k \in \mathcal{A}_j} a_k a_{k+g}$, then the numerical datum
$y_{r_{\mathcal{I}}}^{(j)}$ matches the numerical datum
$y_{\mathcal{I}}^{(j)}$ of~\cite[Theorem 3.6]{SV1/15}.  In addition,
we observe that the numerical data of Definition~\ref{dfn:dwk} satisfy
the hypothesis of Proposition~\ref{pro:igusas.match}.  Recalling from
Example~\ref{exm:gen.igusa} how to express the weak order Igusa
functions of~\cite[Definition 2.9]{SV1/15} in terms of the generalized
Igusa functions of Definition~\ref{def:igusa.wo} above, we find that
Proposition~\ref{pro:igusas.match} indeed
implies~\eqref{equ:suffices.for.heisenberg}.

\begin{rem} \label{rmk:free.products.match}
Observe that $\mathfrak{h} = \mathfrak{f}_{2,2}$.  Thus we can view $L
= \mathfrak{f}_{2,2}(\Lri_1) \times \cdots \times
\mathfrak{f}_{2,2}(\Lri_g)$ and obtain an expression for $\zidealo_L
(s)$ by specializing the analysis of
Section~\ref{subsec:free.nilpotent}.  This expression is not obviously
equal to the one obtained above by considering $L =
\mathfrak{g}_{1,1}(\Lri_1) \times \cdots \times
\mathfrak{g}_{1,1}(\Lri_g)$ and using the approach of
Section~\ref{subsec:rel.free.prod}, or to that
of~\cite[Theorem~3.6]{SV1/15}.  To verify the equality directly, one
has to prove identities between generalized Igusa functions that
depend on the numerical data, in the style of
Proposition~\ref{pro:igusas.match}.  We leave this as an exercise for
the reader.
\end{rem}

\subsection{The higher Heisenberg Lie rings} \label{subsec:higher.heisenberg}
Let $d \in \N$.  The higher Heisenberg Lie ring $\mathfrak{h}_d$
consists of $d$ copies of the Heisenberg Lie ring $\mathfrak{h}$,
amalgamated over their centres; in particular $\mathfrak{h}_1 = \mathfrak{h}$.  More precisely, $\mathfrak{h}_d$ is
spanned over $\Z$ by $2d+1$ elements $x_1, \dots, x_d, y_1, \dots,
y_d, z$, with the relations $[x_i, y_i] = z$ for all $i \in [d]$; all
other pairs of generators commute.  Let
$$L = \mathfrak{h}_{d_1}(\Lri_1) \times \cdots \times
\mathfrak{h}_{d_g}(\Lri_g),$$ where $(d_1, \dots, d_g) \in \N^g$ and
each $\Lri_i$ is a finite, not necessarily unramified extension of
$\lri$.  In the case of $d_1 = \cdots = d_g$ and $\lri = \Lri_1 =
\cdots = \Lri_g = \Z_p$, the zeta function $\zidealo_L(s)$ was
computed by Bauer in his unpublished M.Sc.~thesis~\cite{Bauer/13} by adapting the methods
of~\cite{SV1/15}.
Observe that
\begin{equation} \label{equ:higher.heisenberg.decomp}
\overline{L} \simeq  \Lri_1^{d_1} \times \Lri_1^{d_1} \times \cdots \times \Lri_g^{d_g} \times \Lri_g^{d_g} = 
 \underbrace{\Lri_1 \times \cdots \times \Lri_1}_{2d_1 \, \text{copies}} \times \cdots \times
\underbrace{\Lri_g \times \cdots \times \Lri_g}_{2d_g \, \text{copies}}.
\end{equation}
Set $S_i = \sum_{j = 1}^i 2d_j$.  We have naturally expressed
$\overline{L}$ as a product of $S_g$ submodules, giving rise to
projections $\pi_1, \dots, \pi_{S_g}$ as in
Section~\ref{subsec:rewriting}, where $\pi_k : \overline{L} \to
\Lri_i$ when $S_{i-1} < k \leq S_i$.  Let $\Lambda \leq \overline{L}$
be an $\lri$-sublattice, and let $\boldsymbol{\nu}(\Lambda) =
(\nu^{(1)}, \dots, \nu^{(S_g)})$ be the corresponding projection data
with respect to~\eqref{equ:higher.heisenberg.decomp}; each of these
$S_g$ partitions has only one part.  Note that $L^\prime = Z(L)$ has
rank $c = \sum_{i = 1}^g e_i f_i$ as an $\lri$-module.

\begin{lem}
Let $\Lambda \leq \overline{L}$ be an $\lri$-sublattice.  The
$\lri$-elementary divisor type $\lambda(\Lambda)$ of the commutator
$[\Lambda, \overline{L}] \leq L^\prime$ is obtained from the following
multiset with $c$ elements:
$$ \coprod_{i = 1}^g \left\{ \min \{ \nu^{(S_{i - 1} + 1)}_1, \nu^{(S_{i-1} + 2)}_1, \dots, \nu^{(S_i)}_1 \}_{e_i, f_i} \right\}.$$
\end{lem}
\begin{proof}
Let $(x_1^{(i)}, \dots, x_{d_i}^{(i)}, y_1^{(i)}, \dots,
y_{d_i}^{(i)}, z^{(i)})$ be the natural basis of
$\mathfrak{h}_{d_i}(\Lri_i)$ as an $\Lri_i$-module.  Let the
decomposition~\eqref{equ:higher.heisenberg.decomp} be such that, for
every $k \in [d_i]$, the images of $\pi_{S_{i-1} + k}$ and
$\pi_{S_{i-1} + d_i + k}$ are $\Lri_i x_k^{(i)}$ and $\Lri_i
y_k^{(i)}$, respectively.  If $\Pi_i \in \Lri_i$ is a uniformizer,
then it is clear that, for all $i \in [g]$ and all $k \in [d_i]$,
\begin{align*}
[\Lambda, \Lri_i x^{(i)}_k]  &=  \Pi_i^{\nu_1^{(S_{i-1} + g_i + k)}} \Lri_i z^{(i)} \\
[\Lambda, \Lri_i y^{(i)}_k]  &=  \Pi_i^{\nu_1^{(S_{i-1} + k)}} \Lri_i z^{(i)}.
\end{align*}
The claim follows.
\end{proof}

It is immediate from the previous lemma that
Hypothesis~\ref{hypothesis} is satisfied if all the extensions $\Lri_i
/ \lri$ are unramified.  In this case, we set $\klim = \sum_{i = 1}^g f_i$ and take the collection $\widetilde{\mathfrak{S}}_1, \dots, \widetilde{\mathfrak{S}}_{\klim}$ to consist of $f_i$ copies of the pair $([S_{i-1} + 1, S_i], (1,1,\dots,1))$ for every $i \in [g]$.  The following is then given by Lemma~\ref{lem:length.function}.
\begin{lem}
Let $v = \prod_{k = 1}^{S_g} a_k^{\alpha_k} \in C_{\underline{n}}$. Then $\ell(v) = \sum_{i = 1}^g \left( \prod_{k = S_{i - 1} + 1}^{S_i} \alpha_k \right) f_i$.
\end{lem}
An explicit expression for $\zidealo_L (s)$ can now be obtained from Theorem~\ref{thm:zeta.explicit}.

\begin{acknowledgements}
  The research of all three authors was supported by a grant from the
  GIF, the German-Israeli Foundation for Scientific Research and
  Development (1246/2014). {An extended abstract of
    this work for the FPSAC~2020 conference has appeared
    as~\cite{CSV_FPSAC/20}.}

  AC gratefully acknowledges the support of the Erwin Schr\"odinger
  International Institute for Mathematics and Physics (Vienna) and the
  Irish Research Council through grant no.\ GOIPD/2018/319.  The Emmy
  Noether Minerva Research Institute at Bar-Ilan University supported
  a visit by CV during the preliminary stages of this project. AC and
  CV are grateful to the University of Auckland for its hospitality
  during several phases of this project.
  
  We are grateful to Tomer Bauer for sharing with us some computations
  that provided important initial pointers, and to Tomer Bauer and the anonymous referee for careful
  readings of the text.
  \end{acknowledgements}

\def\cprime{$'$}
\providecommand{\bysame}{\leavevmode\hbox to3em{\hrulefill}\thinspace}
\providecommand{\href}[2]{#2}

\end{document}